\documentclass[12pt,reqno]{amsart}

\usepackage{mathrsfs}
\usepackage{amsfonts}
\usepackage{amssymb}
\usepackage{amsxtra}
\usepackage{dsfont}
\usepackage[dvipsnames]{xcolor}
\usepackage[english,polish]{babel}
\usepackage{enumerate}
\usepackage{bbm}
\usepackage{footmisc}

\usepackage{tikz}
\usetikzlibrary{arrows}
\usetikzlibrary{cd}
\usetikzlibrary{babel}

\usepackage[compress, sort]{cite}

\usepackage{newtxmath}
\usepackage{newtxtext}

\usepackage[margin=2.5cm, centering]{geometry}
\usepackage[colorlinks,citecolor=blue,urlcolor=blue,bookmarks=true]{hyperref}
\hypersetup{
pdfpagemode=UseNone,
pdfstartview=FitH,
pdfdisplaydoctitle=true,
pdfborder={0 0 0}, 
pdflang=en-US
}

\newcommand{\CC}{\mathbb{C}}
\newcommand{\DD}{\mathbb{D}}

\newcommand{\NN}{\mathbb{N}}
\newcommand{\RR}{\mathbb{R}}

\newcommand{\xx}{\mathbbm{x}}

\newcommand{\calL}{\mathcal{L}}
\newcommand{\calS}{\mathcal{S}}
\newcommand{\calF}{\mathcal{F}}

\newcommand{\calM}{\mathcal{M}}
\newcommand{\calB}{\mathcal{B}}

\newcommand{\calD}{\mathcal{D}}
\newcommand{\calH}{\mathcal{H}}

\newcommand{\calT}{\mathcal{T}}

\newcommand{\frakM}{\mathfrak{M}}

\newcommand{\Bor}{\text{Bor}}

\newcommand{\tr}{\operatorname{tr}}

\DeclareMathOperator{\lin}{lin}
\DeclareMathOperator{\Ext}{Ext}
\DeclareMathOperator{\Ran}{Ran}
\DeclareMathOperator{\Ker}{Ker}
\DeclareMathOperator{\Pol}{Pol}
\DeclareMathOperator{\supp}{supp}
\newcommand{\Dom}{\mathrm{D}}
\DeclareMathOperator{\Orb}{Orb}

\newcommand{\Mat}[2]{\operatorname{M}_{#1}(#2)}

\newcommand{\restr}[2]{#1 \!\! \restriction_{#2}}
\newcommand{\ext}{\mathrm{ext}}
\newcommand{\dScalP}[1]{\langle \! \langle #1 \rangle \! \rangle}
\newcommand{\tnorm}[1]{|\!|\!| #1 |\!|\!|}
\newcommand{\bClassAbs}[1]{\textbf{[} #1 \textbf{]}}
\newcommand{\dist}[1]{\mathrm{dist} (#1)}
\DeclareMathOperator{\VCF}{VCF}
\DeclareMathOperator{\VE}{VE}
\DeclareMathOperator{\VNE}{VNE}

\newcommand{\sigmaP}{\sigma_{\mathrm{p}}}

\newcommand{\sigmaAC}{\sigma_{\mathrm{ac}}}

\newcommand{\muAC}{\mu_{\mathrm{ac}}}
\newcommand{\muS}{\mu_{\mathrm{sing}}}

\newcommand{\ud}{{\: \mathrm{d}}}

\newtheorem{theorem}{Theorem}[section]

\newtheorem{proposition}[theorem]{Proposition}
\newtheorem{lemma}[theorem]{Lemma}

\newtheorem{observation}[theorem]{Observation}
\newtheorem{fact}[theorem]{Fact}

\theoremstyle{plain}
\newcounter{thm}

\numberwithin{equation}{section}

\theoremstyle{definition}
\newtheorem{example}[theorem]{Example}
\newtheorem{remark}[theorem]{Remark}
\newtheorem{definition}[theorem]{Definition}

\title[Spectral Theory of s.a. Finitely Cyclic Operators]
{Spectral Theory of Self-adjoint Finitely Cyclic Operators \\ and \\ Introduction to Matrix-measure $L^2$-spaces}

\author{Marcin Moszy\'{n}ski}
\address{
    Marcin Moszy\'{n}ski \\
    Faculty of Mathematics, Informatics and Mechanics \\
    University of Warsaw \\
    ul. Stefana Banacha 2 \\
    02-097 Warsaw, Poland
}
\email{mmoszyns@mimuw.edu.pl}

\newcommand{\bde}{\begin{description}}
\newcommand{\ede}{\end{description}}

\newcommand{\ben}{\begin{enumerate}}
\newcommand{\een}{\end{enumerate}}

\newcommand{\bite}{\begin{itemize}}
\newcommand{\eite}{\end{itemize}}

\newcommand{\beq}{\begin{equation}}
\newcommand{\beql}[1]{\begin{equation}\label{#1}}
\newcommand{\eeq}{\end{equation}}

\newcommand{\strz}{\longrightarrow}

\newcommand{\sublin}{\underset{\lin}{\subset}}

\newcommand{\opcje}[4]{\left\{
    \begin{array}{ll}#1&\mbox{for $#2$}\\#3&\mbox{for $#4$}\end{array}\right.}

\newcommand{\si}{\sigma}

\newcommand{\om}{\omega}

\newcommand{\Mdc}{\Mat{d}{\CC}}

\newcommand{\lesc}{\left\langle} 
\newcommand{\risc}{\right\rangle}

\newcommand{\eqaemu}{\underset{\mu\mathrm{-a.e.}}{=}}
    

\begin{document}
\selectlanguage{english}


\maketitle

\vspace{17ex}

\begin{abstract}
We study  finitely cyclic self-adjoint  operators in a Hilbert space, i.e.  s.a. operators that posses such a finite subset in the   domain that the orbits of all its elements  with respect to the operator are linearly dense in the space.

One of the main goals here is to obtain  the representation theorem for such   operators in a form analogous to the one  well-known in the  cyclic s.a. operators case.

To do this, we present here a detailed introduction to matrix measures,   to the matrix measure $L^2$ spaces, and to the multiplication by  scalar functions operators in such spaces. This allows us to formulate and prove in all the details the less known  representation result, saying that  the  finitely cyclic s.a.  operator is unitary equivalent to the  multiplication by the identity $\xx$ function on $\RR$ in the appropriate matrix measure $L^2$ space. 

We study also some detailed spectral problems for finitely cyclic s.a.  operators, like the absolute continuity.
\end{abstract}

\newpage

$$
\mbox{\vspace{10ex}}
$$

\begin{center}
{\large Acknowledgments} \\

\vspace{3ex}

\bite
\item 
The author wishes to thank {\bf EIMI --- Leonhard Euler International Mathematical Institute in Saint Petersburg} --- for the invitation and hospitality. 
Some  parts of the ``early version'' of the manuscript of  this work were presented  in EIMI as  the lecture mini-course in the program ``Spectral Theory and Mathematical Physics'' in Jully 2021, during my stay in EIMI as the  ``long-term visitor''. 

My mini-course (see \  https://indico.eimi.ru/event/360/)\\ ``Spectral theory for self-adjoint cyclic operators \& its analog for {\em finitely cyclic} ones with rigorous introduction to matrix measure $L^2$ spaces'' \ \ \ 
was  dedicated  to the memory of my friend {\bf Sergei Naboko}, died in December 2020.
\item
The author also wishes to express his gratitude to {\bf Grzegorz Świderski}: for his creative mathematical  cooperation and for his  invaluable assistance in the preparation of the LaTeX version of the manuscript.
\eite
\end{center}

\newpage

\tableofcontents

\newpage

\section{Introduction} \label{sec:I}
This work is mainly devoted to a detailed presentation of an important representation theorem  for so-called ``finitely cyclic'' self-adjoint  operators in a Hilbert space. The representation of such an  operator is of the form of the   multiplication by the identity $\xx$ function on $\RR$ in the appropriate space determined by the operator. This, generally, poorly known space is something we devote a lot of attention  here. 

The notion of  finitely cyclic operator is a natural generalisation of the well-known cyclic operator, i.e.,  of such a linear operator $A$  (possibly unbounded) in a linear topological space $\calH$ for which there exists a vector $\varphi \in \calH$ (a {\em cyclic vector for $A$}) such that the linear span of ``the orbit'' \ $\{A^n\varphi: n \in \NN_0\}$ \ is dense in $\calH$. Roughly speaking, for the generalized finitely cyclic notion, instead of one cyclic vector $\varphi$ we should consider a finite system of vectors and we should require the density of the  linear span of the usual set sum of their orbits for $A$. Note, that for  the case of  self-adjoint operators  the names ``operator with finite multiplicity spectrum'' and ``simple spectrum operator'' are  also used in spectral theory literature instead of  ``finitely cyclic operator'' and ``cyclic operator'', respectively.

A well-known result\footnote{Originally it was  probably one of the Stone's results.} (see e.g.  \cite[Lemma 1, Section VII.2]{Reed1980} for the bounded operator case or \cite[Proposition 5.18]{Schmudgen2012} for the general case) says that each cyclic self-adjoint operator can be represented in a unitary equivalent form  in a very fundamental form, in a sense a ''diagonal'' form with a ``canonical'' diagonal. Namely, as the operator of the multiplication by the identity function $\xx$ on $\RR$ in the space $L^2(\mu)$, for some Borel finite  measure  $\mu$ on $\RR$. Moreover, $\mu$ can be taken as  the scalar  spectral measure for the operator  $A$ and any fixed cyclic vector $\varphi$ for $A$. The appropriate representation theorem is called here {\em the $\xx$-multiplication Unitary Equivalence ($\xx$MUE) Theorem  (for the cyclic case)}.

Our main task --- {\em $\xx$MUE Theorem  (for the finitely cyclic case)} --- is the  generalisation of the above  result to the  finitely cyclic case, and it is in many ways  a really analogous result. But to understand this we we have to solve  an introductory  problem of finding  the appropriate background, which allows to formulate the  generalisation. The finitely cyclic case needs some particular,  less-known and more ``sophisticated''  Hilbert spaces than the space $L^2(\mu)$ for  measures $\mu$ (the ``scalar, non-negative'' ones). It requires also some spectral  theory  for ``multiplication by function  operators'' in such spaces.

One of our aims is to show that after solving the above mentioned problem of the background, we can not only formulate the main result in a similar manner, but also the main ideas of the proof would be the same. So, we try to make this this work  self-contained   to a great extent. 

A large part of  of the material  presented here can be found in some  existing, but strongly scattered, mathematical  literature  of  various degree of strictness,  detailedness and even correctness or truthfulness. We give some references, discussing  the appropriate parts of the theory.

In Section  \ref{sec:II} we formulate and prove the classical cyclic version of $\xx$MUE Theorem. We show here also some important details of the proof for the ``full version unbounded case'' based on the so-called {\em polynomial approach} \footnote{In  literature one can find mainly  proofs for the case of $A$ bounded or for a slightly more general special case: when the linear span of the orbit of the cyclic vector is a core space for $A$ (i.e., the closure of the restriction of $A$ to this linear span equals $A$). One of the exceptions is the monograph \cite[subsection 5.4]{Schmudgen2012}, however its  formulations (including also the definition of cyclicity) are based on different notions and are expressed in slightly  different ``non-polynomial''terms.}.

Section  \ref{sec:III} is devoted to the backgrounds mentioned above. We first define and develop   the notion of {\em matrix measure}, being a surprisingly good and convenient generalisation of finite scalar measures, as a foundation to build an ``$L^2$- Hilbert space''.  So,  we introduce the semi-scalar product  space $\calL^2(M)$ of 
$\CC^d$-vector functions for a matrix measure $M$. Using the standard quotient space argument  we then get the space $L^2(M)$  being   our most important  Hilbert space here. 

Somewhat similar, but different  approach, can be found, e.g., in  \cite[subsections XIII.5.6  -- XII.5.11.]{DunSchwII}. 
 
 We would like also to stress that 
this theory  possesses a lot of analogies with their classical $L^2(\mu)$ counterparts based on ``non-negative scalar measures'' or --- equivalently --- the classical $d=1$ case,  but it has also some interesting new delicate and deep features.

An important part \ref{sec:III:5} of Section  \ref{sec:III} is devoted to multiplication by function operators in $L^2(M)$ spaces and to its spectral properties. We present it also with details of the proofs, however it seems to be an almost direct repetition of the standard arguments from the case $d=1$.

Section \ref{sec:IV} contains  formulation  of the  main result ---  $\xx$MUE Theorem  for the finitely cyclic case and its rigorous proof made in the manner being  analog to the one presented for the case $d=1$ in Section  \ref{sec:II}. Some other formulations with no proofs or for some special cases only, or with only  sketches of the proof etc., can be found e.g.  in \cite[Section VI, subsection 86]{AchizerGlazmanII}, \cite[subsection 2.5]{Teshl2000}.

Appendix  \ref{sec:A},  \ref{sec:B} and  \ref{sec:C} contain some extra material concerning several particular spectral results   for $L^2(M)$ spaces and multiplication by function operators. The first two parts have more technical character.  But the main goal of this material
is the last part --  
Appendix   \ref{sec:C}.  It is devoted to the important spectral  notion of {\em the absolute continuity (in a  $G\subset\RR$)} of  self-adjoint operators. We prove  an abstract result (Theorem \ref{thm:C:7}) on the absolute continuity of a finitely cyclic operator and its absolutely continuous spectrum there.

\vspace{5ex}

\subsection*{Notation}\label{notation}
Let us  fix  here some basic  notation and terminology  used. The remaining notation will be  introduced ``locally'', as the need arises (see also List of symbols). 

By $\Pol(\RR)$ we denote the  set of all complex polynomial functions   $f:\RR\strz\CC$. So,   $$\Pol(\RR)=\lin\{\xx^n:\ n\in\NN_0\},$$
where 
\bite
\item $\lin Y$ or $\lin (Y)$ denotes the linear span in the complex sense of the set $Y$, for $Y$  being  a subset of a complex linear space $X$,  
\item   $\xx$ is the identity function on $\RR$ (i.e. $\xx(t)=t$ for $t\in\RR$), 
\item $f^n$ is the $n$-th power of the scalar function $f$,
\item $\NN_0:=\NN\cup\{0\}$, where $\NN$ is the set of natural numbers.
\eite

When we write \ \   $X'\sublin X$, \ it means that $X'$ is a linear subspace of   a linear space $X$ (--- all in the  complex linear sense). Generally, if not precised, linear spaces are assumed here to be  $\CC$-linear.

By $\Mdc$ we denote the space of all  $d \times d$ complex matrices with $d \in \NN$. For $A\in\Mdc$ and $i,j\in \{1, \ldots, d\}$ the term of $A$ from its  $i$-th row and $j$-th column is denoted by $A_{i,j}$.

Consider  some fixed set $\Omega$.  For $\omega \subset\Omega$ the symbol $\chi_\omega$ denotes the characteristic function on $\Omega$ of the set $\omega$, i.e., for $t\in\Omega$
$$
    \chi_\om(t) = \opcje{1}{t\in\om}{0}{t\notin\om.}
$$

For a function $f$ and a subset $\Omega'$ of its the domain $\restr{f}{S}$ denotes the restriction of $f$ to $S$.

Suppose that we have fixed  also   $\frakM$ --- a $\sigma$-algebra of subsets of $\Omega$.

\label{fn:III:2}

Let us fix some "measure types" terminology used here and concerning  appropriate functions defined on  $\frakM$:
\begin{itemize}
\item
\emph{measure} (without extra adjective or name as ``real'', ``vector'', ``complex'' and so on) = any countably additive function  into $[0; +\infty]$ (so, we shall  usually try to omit the adjectives  ``non-negative'', ``scalar'' etc.);
\item 
\emph{complex measure / real measure / $X$-vector measure (or $X$-measure) / vector measure} = any countable additive function into $\CC / \RR / X / $ a normed space (above: $X$ is a normed space);
\item 
\emph{finite measure}: such a measure $\mu$, that $\mu(\Omega) < +\infty$.
\end{itemize}
So, in particular, a vector or real, or complex measure  ``usually'' {\bf is not} a measure here! Note also that a measure is a complex measure iff it is a finite measure. In all the above cases we add ``\dots on $\frakM$'', when we want to precise the domain of the  function.

Suppose now that $\mu$ is a measure on $\frakM$, and $p\in [0; +\infty]$. By \  $L^p(\mu)$ we denote the standard ``$L^p$''  Banach spaces of {\bf\em classes} of scalar $\frakM$-measurable functions determined by $\mu$ and $p$. It is however  important that for our generalisation reasons we  distinguish here the space of classes  $L^p(\mu)$ and the appropriate space of just  ``$L^p$'' {\bf\em functions}  $\calL^p(\mu)$, being ``only'' a semi-norm linear space.

For the special --- Hilbert --- case $p=2$ we use the notation $[f]$ to denote the class  $[f]\in L^2(\mu)$ for  a function $f \in \calL^2(\mu)$  (note that $[f]$ depends on the measure $\mu$, but we omit it  in the notation if the choice of $\mu$ is clear, and in other cases we introduce some  extra notation). We use also the symbol   $\|[f]\|_\mu$     for the norm of $[f]$ in $L^2(\mu)$     and \  $\tnorm{f}_\mu$  \    for the semi-norm of $f$ in $\calL^2(\mu)$, i.e.
$$
\|[f]\|_\mu =
	\tnorm{f}_\mu =
	\bigg( \int_\Omega |f|^2 \ud \mu \bigg)^{1/2}.
$$

The symbol $\calB(X,Y)$ denotes the space of all the bounded linear operators $A$ on $X$ into $Y$ for   normed  spaces $X$ and $Y$, and the operator norm or each such $A$ is simply denoted by $\|A\|$, if the choice of $X$ and $Y$ is clear. As usual,  $\calB(X):=\calB(X,X)$. 
We consider here unbounded operators with the usual non-strict  meaning of the  unboundedness. That is  a linear operator in $X$ can be not necessary an element of  $\calB(X)$  (--- in particular --- it can be defined not {\bf on} the ``whole'' $X$). So,  $A$ is called here   {\em a (linear)  operator {\bf\em in} $X$}  iff   there exists $X'\sublin X$  such that   
   $A: \Dom(A) \strz Y$ is a linear function ($\Dom(A)$ is the domain of $A$). For such operators the standard algebraic operations  and notions are considered here (see e.g. \cite{Rudin1991}).  E.g. for $n\in \NN_0$  and an operator $A$ in $X$  the $n$-th power $A^n$ is defined, and \ \  $\Dom(A^\infty):= \bigcap_{n\in\NN_0 } \Dom(A^n)$.

We are here mostly interested in self-adjoint\footnote{We use here the usual abbreviation ``s.a.'' for ``self-adjoint''.} operators (unbounded) in  a Hilbert space. If  $\calH$ is  a Hilbert space and $A$ --- a   self-adjoint  operator $A$ in $\calH$, then  the projection-valued spectral measure (``the resolution of identity'') for $A$  we denote by $E_A$. In particular $E_A:\Bor(\RR)\strz \calB(\calH)$, where $\Bor(\RR)$ is the Borel  $\si$-algebra of $\RR$ (Note that we do not restrict here $E_A$ to the Borel subsets of the spectrum $\si(A)$ of $A$ only). 
If, moreover, $x,y\in\calH$,  then $E_{A,x,y}$ is the complex measure ({\em the spectral measure for $A$, $x$ and $y$}) given by
\beql{eq:I:spmes}
E_{A,x,y}(\om)=\lesc E_A(\om)x,y \risc, \quad \om\in  \Bor(\RR),
\eeq
and $E_{A,x}:=E_{A,x,x}$  is a finite measure ({\em the spectral measure for $A$ and $x$}.

If $S\in\Bor(\RR)$, $S\subset\si(A)$, and $f: S\strz\CC$ is a Borel measurable function, then by $f(A)$ we denote as usual  the function $f$ (equivalently: the function $\restr{f}{S}$ or the function $f$ extrapolated to the whole $\RR$ by zero.) of the operator $A$, in the sense of the standard  functional calculus for self-adjoint operators (see e.g.\cite[Section 13: Thm. 13.24 used for the resolution of unity $E_A$ for $A$ given by Thm. 13.30]{Rudin1991})

When $\frakM$ is a $\si$-algebra of subsets of $\Omega$, $\mu$ is a measure on $\frakM$, and $F:\Omega\strz\CC$ is an $\frakM$-measurable function, then \ $T_F$ denotes the (maximal) operator of the multiplication by the function $F$ in the space $L^2(\mu)$. Recall, that 
\begin{align}
	\label{eq:I:opmult}
	&\Dom(T_F) := \big\{ [f] \in L^2(\mu) :    f, F f \in \calL^2(\mu) \big\} \\
	\nonumber
	&T_F[f] := [F f] \quad
	\text{for } [f] \in \Dom(T_F).
\end{align}

Note that $T_F$ depends on the measure $\mu$, but we omit it  in the notation, and we add some extra  explanation or  notation, if the choice of $\mu$ is not clear. Note also that we shall extrapolate the notation $T_F$ in subsection \ref{sec:III:5} also for more general case of the 
operator of the multiplication by the function $F$ in the space $L^2(M)$, for matrix measure $M$.

We  frequently use here the abbreviations:   
STh, FCTh and STh+FCTh. -  This are just the references to the literature concerning the Spectral Theorem for self-adjoint operators, to the Functional Calculus Theorem (for  resolutions of the identity and/or for  self-adjoint operators), and to both of the, respectively. In fact we mainly  recommend this way \cite[Section 13: Thm. 13.24 as FCTh and  Thm. 13.30 as STh]{Rudin1991}. But see also \cite[subsections 5.3 and 5.3]{Schmudgen2012}.

\newpage

\section{A detailed proof of The "$\xx$MUE" Theorem for self-adjoint cyclic operators} \label{sec:II}

Let us begin here  from the strict definition of the cyclicity of an  operator. If  $A$ is a linear operator in a normed space $X$, and $\varphi\in\Dom(A^\infty)$, then we can consider {\em the orbit of $\varphi$ with respect to $A$}:
$$
\Orb_A(\varphi) := \left\{ A^n \varphi : n \in \NN_0 \right\}. 
$$

\begin{definition}
\emph{$\varphi$ is a cyclic vector  for $A$} iff 
\begin{equation}
	\label{eq:IV:1.1}
	\varphi \in \Dom(A^\infty) \quad \text{and} \quad
	\lin\Orb_A(\varphi)  \text{ is  dense in } X.
\end{equation}
(i.e., the orbit is linearly dense)
\emph{$A$ is cyclic} iff there exists a cyclic vector in $X$ for $A$.
\end{definition}

We present here  a detailed  proof  of the  representation theorem  for  cyclic self-adjoint  operators in  Hilbert space. As we shall see, the core-idea of the proof is almost trivial, if  we can base on  STh+FCTh. So the ``main part'' of the proof, sufficient to get the result, e.g., for bounded $A$, is presented in the first subsection. And some extra  arguments, which allows to get the general case,  ``more delicate'' ones, are placed in the second part of the section.    

The reason of presenting here this theorem and its proof is just to show a starting point for the generalisation of the result ``from one vector $\varphi$ to the finite system of vectors $(\varphi_1, \ldots\varphi_d)$.

\vspace{5ex}

\subsection{The $\xx$-multiplication Unitary Equivalence ("$\xx$MUE") Theorem --- The cyclic case\ } \ \label{sec:II:1}

We need here one more symbol. If $\mu$ is a measure on  $\Bor(\RR)$ such that
\beql{III:all-moments}
\xx^n \in \calL^2(\mu) \ \ \ \mbox{for any} \ \ n\in\NN_0
\eeq
(which, as one can easily check, equivalently means that all moments of  $\mu$ exist and are finite), 
then 
$$
\Pol_\mu(\RR):=\{[f]\in L^2(\mu): f\in\Pol(\RR)\}= \lin\{[\xx^n]\in L^2(\mu):n\in\NN_0\}.
$$
The classes $[\cdot]$ mean  above  the classes in the sense of $L^2(\mu)$, which  is clearly equivalent also  to the sense of $\mu$ - a.e. equality of Borel functions\footnote{More precisely: if $f\in\calL^2(\mu)$, then simply $[f]=\{g\in\Bor(\RR): f\eqaemu g \}$.}.

Let $\calH$ be a complex Hilbert space.

\begin{theorem}[$\xx$MUE] \label{thm:II:xMUE}
If $A$ is a cyclic s.a. operator in $\calH$, and $\varphi$ is a cyclic vector for $A$, then $A$ is unitary equivalent to $T_{\xx}$ in $L^2(\mu)$, where $\mu = E_{A, \varphi}$.

Moreover,
\begin{enumerate}[(i)]
\item \eqref{III:all-moments} holds and  $\Pol_\mu(\RR)$ is a dense subspace of $L^2(\mu)$;

\item $\exists!_{U \in \calB(L^2(\mu), \calH)} \ \forall_{n \in \NN_0} \ U([\xx^n]) = A^n \varphi$;

\item the unique $U$ from (ii) is a unitary operator from $L^2(\mu)$ onto $\calH$ and it is defined by
\begin{equation} 
	\label{eq:II:1.1}
	U([f]) = f(A) \varphi, \quad 
	f \in \calL^2(\mu),
\end{equation}
and
\begin{equation} 
	\label{eq:II:1.2}
	A = U T_{\xx} U^{-1}.
\end{equation}
\end{enumerate}
\end{theorem}
\begin{proof}[Proof -- Part 1 {the standard consequences of FC...}]
Define linear $\tilde{U} : \calL^2(\mu) \to \calH$ by formula
\[
	\tilde{U} f := f(A) \varphi, \quad
	f \in \calL^2(\mu).
\]
Note that the RHS above is a correct definition of an element of $\calH$ by STh + FCTh (and $\tilde{U}$ is linear), because
\begin{equation} 
	\label{eq:II:1.3}
	\Dom \big( f(A) \big) = 
	\Big\{ y \in \calH : \int_\RR |f|^2 \ud E_{A,y} < \infty \Big\}
\end{equation}
and for $y \in \Dom \big( f(A) \big)$
\begin{equation} 
	\label{eq:II:1.4}
	\| f(A) y \| = 
	\bigg( \int_\RR |f|^2 \ud E_{A,y}) \bigg)^{1/2} =
	\tnorm{f}_\mu,
\end{equation}
so by $f \in \calL^2(\mu)$ and $\mu = E_{A, \varphi}$ for $y=\varphi$ we get $\varphi \in \Dom \big( f(A) \big)$ and
\begin{equation} 
	\label{eq:II:1.5}
	\|\tilde{U} \varphi \|^2 = 
	\| f(A) \varphi \|^2 =
	\tnorm{f}^2_\mu = 
	\| [f] \|^2_{\mu}.
\end{equation}
Moreover, by \eqref{eq:II:1.5}, $\{ f \in \calL^2(\mu) : \tnorm{f}_\mu = 0 \} \subset \Ker{\tilde{U}}$, so $\tilde{U}$ can be factorized to $U: L^2(\mu) \to \calH$ by formula
\begin{equation} 
	\label{eq:II:1.6}
	U([f]) := \tilde{U}(f), \quad
	f \in \calL^2(\mu),
\end{equation}
and $U$ is bounded operator given by \eqref{eq:II:1.1} being also an isometry on its range (image).

One of the "standard" consequences (typical exercise) of STh+FCTh is:
\[
	\forall_{n \in \NN_0} \ \xx^n(A) = A^n,
\]
hence -- we start to use the cyclicity here -- knowing that $\varphi$ is cyclic vector, we have in particular that $\varphi \in \Dom(A^n)$ (for any $n \in \NN_0$), so by \eqref{eq:II:1.3} for $f = \xx^n$ and $y = \varphi$ we get $[\xx^n] \in L^2(\mu)$, and $\Pol_\mu(\RR) = \lin \{ [\xx^n] : n \in \NN_0 \} \subset L^2(\mu)$. We also have
\begin{equation} 
	\label{eq:II:1.6'}
	\forall_{n \in \NN_0} \ U([\xx^n]) = 
	\tilde{U} \xx^n = 
	\xx^n(A) \varphi = 
	A^n \varphi,
\end{equation}
so $U$ is one of the bounded operators as in (ii). Moreover, by the cyclicity of $\varphi$ for $A$ $\lin \{ A^n \varphi : n \in \NN_0 \}$ is dense in $\calH$, so the isometry $U$ is onto $\calH$, i.e. $U$ is a unitary transformation from $L^2(\mu)$ to $\calH$, and $\Pol_\mu(\RR)$ is also dense in $L^2(\mu)$.

Therefore we proved (i) and the "$\exists$" part of (ii), but the "!" (the uniqueness) part of (ii) is obvious by the continuity requirement for the operator, and by the density from (i) (just proved). And thus, it remains only to prove \eqref{eq:II:1.2} from (iii). -- We shall do this in the next part of the proof in full generality. But let us make here a remark, that this remaining part is quite obvious if $A$ is e.g. a bounded operator. Indeed -- observe, that for any $n \in \NN_0$ we obviously have:
\begin{equation}
	\label{eq:II:1.7}
		U^{-1} A (A^n \varphi) =
		U^{-1} A^{n+1} \varphi =
		[\xx^{n+1}] =
		[\xx \cdot \xx^n] =
		T_\xx[\xx^n] = 
		T_{\xx} U^{-1} (A^n \varphi),
\end{equation}
where the last equality follows from $U([\xx^n]) = A^n \varphi$, which means, that the equality
\[
	A y = U T_{\xx} U^{-1} y
\]
is true for \emph{any $y$} from a linearly dense set $\{ A^n \varphi : n \in \NN_0 \}$. So \eqref{eq:II:1.2} holds with the extra assumption: $A \in \calB(\calH)$. Note also, that the argumentation based on \eqref{eq:II:1.7} works with a weaker assumption than the boundedness of $A$. Namely, with the assumption, that $\varphi$ is an essential cyclic vector for $A$. In such a case, both $A$ and $U T_{\xx} U^{-1}$ are s.a. extensions of $\restr{A}{\lin \{ A^n \varphi : n \in \NN_0 \}}$, which is ess.s.a, so they are equal.

But we see, that it is not so obvious, that our assertion remains true, without the essential cyclicity assumption. The proof above does not work, because of "the domain details". But the result \emph{is} true, however! There exist several proofs of \eqref{eq:II:1.2}. One of them just directly "patches the holes" in the proof above, but it is long and "technical". We shall use another proof, based on a certain convenient idea (which we plan to use later -- for the generalized result).
\end{proof}

\vspace{5ex}

\subsection{The ``delicate'' part of the proof of ``$\xx$MUE Theorem''} \label{sec:II:2}
\begin{proof}[Proof of "$\xx$MUETh"-- Part 2 {the final (general case) part}]
The main idea of the proof of \eqref{eq:II:1.2} is to omit the difficulties related to the domains of unbounded operators, and to prove "\eqref{eq:II:1.2}" for "spectral projections of $A$ instead of $A$". -- Namely, we shall prove that
\begin{equation}
	\label{eq:II:2.1}
	\forall_{\omega \in \Bor(\RR)} \
	E_{T_{\xx}}(\omega) = E_{U^{-1} A U}(\omega),
\end{equation}
which would give \eqref{eq:II:1.2} by STh+FCTh ("s.a. operator is the spectral integral of $\xx$ with respect to its spectral measure"). So, fix $\omega \in \Bor(\RR)$. The spectral resolution of $T_{\xx}$ in $L^2(\mu)$ is well known:
\[
	E_{T_{\xx}}(\omega) = T_{\chi_\omega}.
\]
We also have $E_{U^{-1} A U}(\omega) = U^{-1} E_{A}(\omega) U$, hence we should prove only
\[
	U T_{\chi_\omega} = E_A(\omega) U.
\]
But both sides above are bounded operators on $L^2(\mu)$, so by the linear density of $\{ [\xx^n] : n \in \NN_0 \}$ (already proved) it suffices to check, that for any $n \in \NN_0$
\[
	U ([\chi_\omega \xx^n]) \overset{\text{def. of $T_f$}}{=}
	U T_{\chi_\omega}([\xx^n]) \overset{?}{=} 
	E_A U([\xx^n]) \overset{\text{see (ii)}}{=} 
	E_A(\omega)(A^n \varphi)
\]
So we need to check that
\begin{equation}
	\label{eq:II:2.2}
	U ([\chi_\omega \xx^n]) = E_A(\omega)(A^n \varphi).
\end{equation}
But we have 
\[
	U ([\chi_\omega \xx^n]) = (\chi_\omega \cdot \xx^n)(A) \varphi
\]
by \eqref{eq:II:1.6}. And also
\[
	E_A(\omega)(A^n \varphi) = 
	\big( \chi_\omega(A) \xx^n(A) \big) \varphi =
	(\chi_\omega \cdot \xx^n)(A) \varphi;
\]
(because "$f(A) \cdot g(A) \subset (f \cdot g)(A)$", FCTh.), so \eqref{eq:II:2.2} holds!
\end{proof}

The above idea and a similar proof of details will be used soon in the "more sophisticated" case -- the generalization of $\xx$MUETh to finitely cyclic s.a. operators.

\newpage

\section{Matrix measure, $L^2$-space for matrix measure and multiplication by function operators} \label{sec:III}

\vspace{3ex}

\subsection{The matrix measure and its trace measure} \label{sec:III:1}
Consider a set $\Omega$ and $\frakM$ -- a $\sigma$-algebra of subsets of $\Omega$. Let $d \in \NN$ and consider a $d \times d$ complex matrix valued function $M$ on $\frakM$, i.e.
\[
	M : \frakM \to \Mat{d}{\CC}.
\]

\begin{definition}[matrix measure]
\emph{$M$ is a matrix measure (on $\frakM, d \times d$)} iff
\begin{enumerate}[1)]
\item $M$ is countably additive [abrev.: c.add.] in the norm sense in $\Mat{d}{\CC}$\footnote{That is: $M$ is a $\Mat{d}{\CC}$ - vector measure, due to  our terminology here. But note the terminological difference with some other sources:   e.g., due to \cite{Diestel1977},  {\em vector measure} does not have to be countably additive, in general. So this is ``a countably additive vector measure'' in the \cite{Diestel1977} terminology.};
\item $M(\omega) \geq 0$\footnote{\label{fn:III:1} Here $A \geq B$ is in the Hilbert space $\CC^d$ sense:
\[
	A \geq B \quad \Leftrightarrow \quad
	\forall_{x \in \CC^d} \ 
	\RR \ni \langle A x, x \rangle_{\CC^d} \geq \langle B x, x \rangle_{\CC^d} \in \RR.
\]
In particular, both $A,B$ should be self-adjoint, and $A \geq B \Leftrightarrow A-B \geq 0$ (is non-negative). And $A \leq B$ means $B \geq A$...} for any $\omega \in \frakM$.
\end{enumerate}
\end{definition}

In other words, matrix measure is just the "$\Mat{d}{\CC}$-vector non-negative measure". Surely, norm sense of the c.add. above is equivalent to c.add. "on each matrix term".

A simplest example of matrix measure is the $d=1$ case, because then, if we identify $\Mat{1}{d}$ and $\CC$, matrix measure is just a finite measure. Note, that we omit "non-negative" in the name "matrix measure" (although that would be somewhat more consistent with the terminology explained below (as e.g. "$X$ measure")) by tradition.

For a $d \times d$ matrix measure $M$ and $i,j=1,\ldots, d$ let us define $M_{ij}: \frakM \to \CC$ by
\begin{equation}
	\label{eq:III:1.1}
	M_{ij}(\omega) := \big( M(\omega) \big)_{ij}, \quad 
	\omega \in \frakM.
\end{equation}

\begin{proposition} \label{prop:III:2}
If $M$ is a matrix measure, then:
\begin{enumerate}[(a)]
\item $M(\emptyset) = 0$;
\item (finite add.) $\forall_{\omega, \omega' \in \frakM} \ \big(\omega \cap \omega' = \emptyset \Rightarrow M(\omega \cup \omega') = M(\omega) + M(\omega') \big)$;
\item (monotonicity) $\forall_{\omega, \omega' \in \Omega} \ \big( \omega' \subset \omega \Rightarrow M(\omega') \leq M(\omega) \big)$;
\item $\forall_{i,j=1,\ldots,d} \ M_{ij}$ is a complex measure, $M_{ij} = \overline{M_{ji}}$, $M_{ii}$ is a finite measure.
\end{enumerate}
\end{proposition}

Note, that the point (c) gives the following important property of any matrix measure $M$:
\begin{equation}
	\label{eq:III:1.2}
	\forall_{\substack{\omega, \omega' \in \frakM\\\omega' \subset \omega}} \quad
	M(\omega) = 0 \Rightarrow M(\omega') = 0.
\end{equation}
The above follows from: $0 \leq A \leq 0 \Rightarrow A = 0$, which is true (and easy to check) for the "matrix $\leq$". The property \eqref{eq:III:1.2} shows how similar is any matrix measure to "usual" measures (observe that the analog of \eqref{eq:III:1.2} for complex or vector measures is generally not true!).

A "stronger similarity" to measures follows from the next fact, stating that any matrix measure is absolutely continuous with respect to a certain measure (moreover -- a finite measure). For $M$ -- a matrix measure, denote by $\tr_M$ the function from $\frakM$ into $\CC$ given by
\begin{equation}
	\label{eq:III:1.3}
	\tr_M(\omega) := \tr \big( M(\omega) \big), \quad
	\omega \in \frakM.
\end{equation}
By d) of Proposition~\ref{prop:III:2} we know, that $\tr_M$ is a finite measure, as sum of $d$-finite measures.

\begin{proposition} \label{prop:III:3}
If $M$ is a matrix measure, then $M$ and each $M_{ij}$ for $i,j=1,\ldots,d$ are absolutely continuous with respect to $\tr_M$, i.e.,
\[
	\forall_{\omega \in \frakM}\ 
	\Big( \tr_M(\omega) = 0 \Rightarrow \big( M(\omega) = 0 \text{ and }
	\forall_{i,j=1,\ldots,d} \ M_{ij}(\omega) = 0 \big) \Big).
\]
\end{proposition}
\begin{proof}
The assertion for $M$ follows directly from the following matrix property:
\begin{equation}
	\label{eq:III:1.4}
	\text{If } A \in \Mat{d}{\CC} \text{ and } A \geq 0, \text{ then }
	A \leq \tr(A) \cdot I
\end{equation}
(just diagonalize $A$ ...), and this obviously gives the assertion for any complex measure $M_{ij}$, too.
\end{proof}

We get more by \eqref{eq:III:1.4}: for a matrix measure $M$ on $\frakM$
\begin{equation}
	\label{eq:III:1.4'}
	\forall_{\omega \in \frakM} \ 0 \leq M(\omega) \leq \tr_M(\omega) I.
\end{equation}

Now, using Proposition~\ref{prop:III:3} and Radon-Nikodym Theorem, let us choose for each $i,j=1,\ldots,d$ a density $d_{M,i,j} : \Omega \to \CC$ of $M_{ij}$ with respect to $\tr_M$ -- note, that we have a free choice of $d_{M,i,j}$ up to the $\tr_M$ a.e. equality only. Then we define $D_{M} : \Omega \to \Mat{d}{\CC}$ by
\[
	\big( D_{M}(t) \big)_{i,j} := d_{M,i,j}(t), \quad
	t \in \Omega,\ i,j=1,\ldots,d.
\]
So, $D_M$ is a function defined also up to the $\tr_M$ a.e. equality. -- By $\DD_M$ we denote the class of all such measurable functions. Hence, finally, for any $D_M \in \DD_M$
\begin{equation}
	\label{eq:III:1.5}
	\forall_{\omega \in \frakM} \ M(\omega) = 
	\int_\omega D_M \ud \tr_M \quad 
	\bigg( = \int_\omega D_M(t) \ud \tr_M(t) \bigg),
\end{equation}
where the integral above is just the standard Lebesgue integral of the $\Mat{d}{\CC} \simeq \CC^{d^2}$-valued function with respect to the measure $\tr_M$ (the integral is defined coordinatewise, by the appropriate scalar valued function integrals).

Let us give some names, to the above denoted objects determined by matrix measure $M$.

\begin{definition}[trace measure, trace density, ...]
We call:
\begin{itemize}
\item $\tr_M$ -- \emph{the trace measure}
\item $\DD_M$ -- \emph{the trace density class}
\item each $D_M \in \DD_M$ -- \emph{the trace density}
\end{itemize}
\emph{of $M$}.
\end{definition}

The trace density cannot be just an arbitrary "matrix-$\calL^1(\tr_M)$" function ($L^1$ is just by Radon-Nikodym Theorem)!

\begin{proposition} \label{prop:III:5}
Any trace density $D_M$ of a matrix measure $M$ satisfies:
\[
	D_M(t) \text{ is self-adjoint and }
	0 \leq D_M(t) \leq I \text{ for } \tr_M \text{a.e.} t \in \Omega.
\]
\end{proposition}
\begin{proof}
We shall use here the following general property of matrix (finite dimensional vector) valued function integrals:
\begin{fact} \label{fact:III:1.6}
If $\mu$ is a measure in $(\Omega, \frakM)$, $\omega \in \Omega$, $\varphi: \CC^k \to \CC^\ell$ is linear and $f : \Omega \to \CC^k$ is integrable with respect to $\mu$, then $\varphi \circ f : \Omega \to \CC^\ell$ is also integrable and
\[
	\varphi \bigg( \int_{\omega} f \ud \mu \bigg) =
	\int_\omega \varphi \circ f \ud \mu.
\]
\end{fact}

Now let $\omega \in \frakM$ and $x \in \CC^d$. By \eqref{eq:III:1.5} and Fact~\ref{fact:III:1.6} used to $\varphi: \Mat{d}{\CC} \to \CC$, $\varphi(A) = \langle A x, x \rangle_{\CC^d}$ we have
\[
	0 \leq \langle M(\omega) x, x \rangle_{\CC^d} = 
	\bigg\langle \bigg( \int_\omega D_M \ud \tr_M \bigg) x, x \bigg\rangle_{\CC^d} =
	\int_\omega \langle D_M(t) x, x \rangle_{\CC^d} \ud \tr_M(t).
\]
Since $\omega$ was arbitrary, we get
\begin{equation}
	\label{eq:III:1.7}
	\forall_{x \in \CC^d} \text{ for } \tr_M \text{ a.e. } t \in \Omega \quad
	\langle D_M(t) x, x \rangle_{\CC^d} \geq 0.
\end{equation}
Similarly, by \eqref{eq:III:1.4'}
\begin{align*}
	0 
	\leq 
	\big\langle \big( \tr_M(\omega) - M(\omega) \big) x, x \big\rangle_{\CC^d} &=
	\bigg\langle \bigg(\int_\omega I \ud \tr_M \bigg) x, x \bigg\rangle_{\CC^d} -
	\bigg\langle \bigg(\int_\omega D_M \ud \tr_M \bigg) x, x \bigg\rangle_{\CC^d} \\
	&=
	\int_\omega \big\langle \big( I - D_M(t) \big) x, x \big\rangle_{\CC^d} \ud \tr_M(t)
\end{align*}
and
\begin{equation}
	\label{eq:III:1.8}
	\forall_{x \in \CC^d} \text{ for } \tr_M \text{ a.e. } t \in \Omega \quad
	\big\langle \big(I -D_M(t) \big) x, x \big\rangle_{\CC^d} \geq 0.
\end{equation}
Let now $C$ be a countable dense subset of $\CC^d$. By \eqref{eq:III:1.7}, \eqref{eq:III:1.8} and the countable additivity of $\tr_M$ we can choose a set $\omega_0 \in \frakM$ such that $\tr_M(\omega_0) = 0$ and
\[
	\forall_{t \in \Omega \setminus \omega_0}
	\forall_{x \in C} \
	\Big(
	\big\langle D_M(t) x, x \big\rangle_{\CC^d} \geq 0 \text{ and }
	\big\langle \big( I - D_M(t) \big) x, x \big\rangle_{\CC^d} \geq 0
	\Big).
\]
Hence, by the continuity of $D_M(t) : \CC^d \to \CC^d$ and the joint continuity of the scalar product $\langle \cdot, \cdot \rangle_{\CC^d}$ (used for "each fixed $t$") we get
\[
	\forall_{t \in \Omega \setminus \omega_0}
	\forall_{x \in \CC^d} \
	\Big(
	\big\langle D_M(t) x, x \big\rangle_{\CC^d} \geq 0 \text{ and }
	\big\langle \big( I - D_M(t) \big) x, x \big\rangle_{\CC^d} \geq 0
	\Big),
\]
so 
\[
	\forall_{t \in \Omega \setminus \omega_0} \ 0 \leq D_M(t) \leq I
\]
(in particular, for $t \in \Omega \setminus \omega_0$ we get the self-adjointness, since we consider here complex space\footnote{\label{fn:III:3} Recall that in complex Hilbert space $\calH$ for $A \in \calB(\calH)$
\[
	\big( \forall_{x \in \calH} \ \langle A x, x \rangle_{\calH} \in \RR \big) \Rightarrow A \text{ is s.a. (self-adjoint)}.
\]} $\CC^d$).
\end{proof}

\begin{remark} \label{rem:III:6}
Using polarization formula for the sesquilinear, conjugate-symmetric form $\CC^d \times \CC^d \ni (x,y) \mapsto \langle Ax, y \rangle_{\CC^d}$, where $A \in \Mat{d}{\CC}$ and $0 \leq A \leq I$ we easily get $\forall_{i,j=1,\ldots,d} \ |\langle A e_i, e_j \rangle| \leq 2$\footnote{But one can get even more, using   Cauchy-Schwarz inequality. Namely, the best  bound equal $1$.}.). Hence, all the terms $(D_M)_{ij}$, $i,j=1,\ldots,d$ of the trance density $D_M$ of a matrix measure $M$ are not only $\tr_M$ integrable functions, but they are also bounded by $2$, i.e. they satisfy:
\begin{equation}
	\label{eq:III:1.9}
	\forall_{i,j=1,\ldots,d} \ \big| \big( D_M(t) \big)_{ij} \big| \leq 2 \quad
	\text{for } \tr_M \text{ a.e. } t \in \Omega.
\end{equation}
\end{remark}

The formula \eqref{eq:III:1.5} says in particular that each matrix measure on $\frakM$ has the form
\begin{equation}
	\label{eq:III:1.10}
	M = F \ud \mu
\end{equation}
for a measure $\mu$ on $\frakM$ and a $d \times d$-matrix valued, non-negative $\calL^1(\mu)$ function. Recall, that the above "$\ud \mu$" notation is defined in such a general case just by:
\begin{equation}
	\label{eq:III:1.10'}
	\forall_{\omega \in \frakM} \quad M(\omega) = \int_\omega F \ud \mu.
\end{equation}
So, we can say, that each matrix measure has quite a simple, well-understandable form. 

Moreover, we have more information on "our" concrete $F$ and $\mu$, by our previous considerations: $\mu$ is a finite measure and $0 \leq F(t) \leq I$ for $\mu$-a.e. $t \in \Omega$. 

It is natural to ask, whether the opposite is true, i.e.: is each $M$ of the form \eqref{eq:III:1.10} a matrix measure? -- The answer is simple.

\begin{example}[The general example of a matrix measure] \label{example:III:7}
Let $\mu$ be a measure on $\frakM$ and $F: \Omega \to \Mat{d}{\CC}$ bet a non-negative $\calL^1(\mu)$-matrix valued function. Then $M : \frakM \to \Mat{d}{\CC}$ given by \eqref{eq:III:1.10} is a matrix measure.

Moreover, for such a matrix measure $M$, we have
\begin{align}
	\label{eq:III:1.11}
	&\tr_M = \tr F \ud \mu, \\
	\label{eq:III:1.12}
	&\tr_M \big( \{ t \in \Omega : \tr F(t) = 0 \big) = 0
\end{align}
and for each $D_M \in \DD_M$
\begin{equation}
	\label{eq:III:1.13}
	M(t) = \frac{1}{\tr F(t)} \cdot F(t) \quad \text{for } \tr_M \text{-a.e. } t \in \{ s \in \Omega : \tr F(s) \neq 0 \}.
\end{equation}
\end{example}
\begin{proof}
Non-negativity is a simple calculation based on Fact~\ref{fact:III:1.6} for $\varphi$ as in Proof of Proposition~\ref{prop:III:5}. And c.add. is a standard measure theory argument. Similarly \eqref{eq:III:1.11} follows from Fact~\ref{fact:III:1.6} for $\tr$-functional and hence, \eqref{eq:III:1.12} and \eqref{eq:III:1.13} are obvious.
\end{proof}

For a more concrete example see Example~\ref{example:III:15}.

We end this part of Section~\ref{sec:III} by the following observation concerning sets of zero $M$-matrix measure.

\begin{remark} \label{rem:III:8}
If $M$ is a matrix measure on $\frakM$ and $\omega \in \frakM$, then "TFCAE\footnote{\label{fn:III:4} The following conditions are equivalent}"
\begin{enumerate}[(i)]
\item $M(\omega) = 0$;
\item $\tr_M(\omega) = 0$;
\item $D_M(t) = 0$ for $\tr_M$ a.e $t \in \omega$ i.e. $\tr_M \big( \{ t \in \omega : D_M(t) \neq 0 \big) = 0$.
\end{enumerate}
Moreover, $\tr_M \big( \{ t \in \Omega : D_M(t) = 0 \} \big) = 0$.
\end{remark}
\begin{proof}
\fbox{(ii) $\Rightarrow$ (i)} is already known by Proposition~\ref{prop:III:3} (the abs. cont. of $M$ w.r. to $\tr_M$).

\fbox{(i) $\Rightarrow$ (ii)} -- obviously: $\tr_M(\omega) = \tr \big( M(\omega) \big) = \tr 0 = 0$, if (i) holds.

\fbox{(ii) $\Rightarrow$ (iii)} -- obvious, since $\tilde{\omega} := \{ t \in \omega : D_M(t) \neq 0 \} \subset \omega$.

\fbox{(iii) $\Rightarrow$ (i)}
\[
	M(\omega) = 
	\int_\omega D_M \ud \tr_M = 
	\int_{\tilde{\omega}} D_M \ud \tr_M +
	\int_{\omega \setminus \tilde{\omega}} 0 \ud \tr_M = 0+0 = 0.
\]
Now for $\Omega' := \{t \in \Omega : \tr_M(\Omega') = 0 \}$ we get $M(\Omega') = \int_{\Omega'} = \int_{\Omega'} D_M \ud \tr_M = 0$, so
$\tr_M(\Omega') = 0$ by (i) $\Leftrightarrow$ (ii).
\end{proof}

The equivalence "(i) $\Leftrightarrow$ (ii)" is very important. It says that we managed to get much more than we wanted to get, and we got in Proposition~\ref{prop:III:3} -- the absolute continuity of $M$ with respect to $\tr_M$. Here we got in some sense also the opposite: "the absolute continuity of $\tr_M$ with respect to $M$". It shows, that the choice of $\tr_M$ as a measure which "contains" possibly large information on our matrix measure $M$ was very good and precise.

\vspace{5ex}

\subsection{$\CC^d$-vector functions space $\calL^2(M)$, semi-norm, "semi-scalar" product in $\calL^2(M)$ and the zero-layer $\calL^2_0(M)$} \label{sec:III:2}
When $\mu$ is a measure on a $\sigma$-algebra $\frakM$ id a set $\Omega$, then the construction of the space $L^2(\mu)$ of the appropriate classes of scalar functions starts from defining of the function space $\calL^2(\mu)$ with the standard semi-norm $\tnorm{\cdot}_{\mu}$ and with the "semi-scalar product" $\dScalP{\cdot, \cdot}_{\mu}$, given by
\[
	\dScalP{f,g}_{\mu} = \int_{\Omega} f \overline{g} \ud \mu, \quad
	\tnorm{f}_\mu = \dScalP{f,f}_{\mu}^{1/2}.
\]
Then we define $\calL_0^2(\mu) := \{ f \in \calL^2(\mu) : \tnorm{f}_{\mu} = 0 \}$, to use the construction of quotient space $\calL^2(\mu)/\calL^2_0(\mu) =: L^2(\mu)$ which leads to a norm space with the norm $\|\cdot\|_{\mu}$ induced by the scalar product $\langle \cdot, \cdot \rangle_{\mu}$, given by
\[
	\langle [f], [g] \rangle_{\mu} := \dScalP{f,g}_{\mu}, \quad
	f,g \in \calL^2(\mu).
\]
Moreover, $L^2(\mu)$ is a Hilbert space.

Our main goal in Section~\ref{sec:III} is to "perform" the analog construction of the space $L^2(M)$ for a matrix measure $M$ instead of a measure $\mu$. And here, in subsection~\ref{sec:III:2} we shall start with the function space $\calL^2(M)$ consisting of some $\CC^d$-functions $f : \Omega \to \CC^d$ (here $M$ is a $d \times d$-matrix measure on $\frakM$). 

So, let us first define $\dScalP{f}_M \in [0, \infty]$ for each such a $\frakM$-measurable $f$ by the formula:
\begin{equation}
	\label{eq:III:2.1}
	\dScalP{f}_M := 
	\int_{\Omega} \langle D_M(t) f(t), f(t) \rangle_{\CC^d} \ud \tr_M(t).
\end{equation}
Note, that the integral above exists (but can be $+\infty$), since $f$ is measurable and $D_M(t) \geq 0$ for $\tr_M$ --a.e. $t \in \Omega$, and the choice of $D_M \in \DD_M$ does not matter here (and below...). Next we denote/define:
\begin{itemize}
\item $\calL^2(M) := \{ f : \Omega \to \CC^d : f \text{ is } \frakM-\text{measurable, } \dScalP{f}_{M} < \infty \}$;

\item $\tnorm{\cdot} : \calL^2(M) \to [0, \infty)$,
\begin{equation}
	\label{eq:III:2.2}
	\tnorm{f}_M := \dScalP{f}_M^{1/2}, \quad
	f \in \calL^2(M);
\end{equation}

\item $\dScalP{\cdot, \cdot}: \calL^2(M) \times \calL^2(M) \to \CC$,
\begin{equation}
	\label{eq:III:2.3}
	\dScalP{f,g}_M := \int_\Omega \langle D_M(t) f(t), f(t) \rangle_{\CC^d} \ud \tr_M(t), \quad
	f,g \in \calL^2,
\end{equation}
\end{itemize}
but to check, that the RHS of \eqref{eq:III:2.3} is properly defined fix $f,g \in \calL^2(M)$ and consider $\gamma_{f,g} : \Omega \to \CC$ given by
\begin{equation}
	\label{eq:III:2.4}
	\gamma_{f,g}(t) := \langle D_M(t) f(t), g(t) \rangle_{\CC^d}, \quad
	t \in \Omega.
\end{equation}
We need
\begin{lemma} \label{lem:III:9}
If $f,g \in \calL^2(M)$, then $\gamma_{f,g} \in \calL^1(\tr_M)$.
\end{lemma}
\begin{proof}
We can assume, that $D_M(t) \geq 0$ for any $t \in \Omega$. So, for $t \in \Omega$, using the $\sqrt{\cdot}$ of non-negative matrix, we get
\begin{align}
	\nonumber
	|\gamma_{f,g}(t)| 
	&= 
	\big| \big\langle \sqrt{D_M(t)} f(t), \sqrt{D_M(t)} g(t) \big\rangle_{\CC^d} \big| \\
	\nonumber
	&\leq
	\Big( \big\langle \sqrt{D_M(t)} f(t), \sqrt{D_M(t)} f(t) \big\rangle_{\CC^d} \Big)^{1/2} 
	\cdot
	\Big( \big\langle \sqrt{D_M(t)} g(t), \sqrt{D_M(t)} g(t) \big\rangle_{\CC^d} \Big)^{1/2} \\
	\label{eq:III:2.5}
	&=
	\Big( \big\langle D_M(t) f(t), f(t) \big\rangle_{\CC^d} \Big)^{1/2} \cdot
	\Big( \big\langle D_M(t) g(t), g(t) \big\rangle_{\CC^d} \Big)^{1/2}.
\end{align}
Now, using the Schwarz inequality for integrals, by \eqref{eq:III:2.1} we obtain
\[
	\int_\Omega |\gamma_{f,g}(t)| \ud \tr_M(t) \leq \dScalP{f}_M^{1/2} \dScalP{g}_M^{1/2}. \qedhere
\]
\end{proof}

Having all our main objects well-defined, we study their expected properties.
\begin{fact} \label{fact:III:10}
$\calL^2(M)$ is a $\CC$-linear subspace of the space of all $\frakM$-measurable $\CC^d$ functions on $\Omega$ (with the standard linear pointwise operations). Moreover, $\dScalP{\cdot, \cdot}_M$ is a sesquilinear, non-negative and conjugate-symmetric form on $\calL^2$.
\end{fact}
\begin{proof}
If $f \in \calL^2(M)$ and $z \in \CC$, then $\dScalP{z f}_M = |z|^2 \dScalP{f}_M$, by \eqref{eq:III:2.1}, so $z d \in \calL^2(M)$. If $f,g \in \calL^2(M)$, then $\dScalP{f+g}_M = \int_\Omega \gamma_{f+g, f+g} \ud \tr_M$ by \eqref{eq:III:2.4}, and we have
\[
	|\gamma_{f+g, f+g}(t)| \leq 
	|\gamma_{f, f}| + |\gamma_{g, g}| + |\gamma_{f, g}| + |\gamma_{g, f}|,
\]
so $\gamma_{f+g, f+g} \in \calL^1(\tr_M)$ by Lemma~\ref{lem:III:9}, i.e. $\dScalP{f+g}_M < \infty$, and $f+g \in \calL^2(M)$. Thus $\calL^2(M)$ is a $\CC$-linear subspace of the measurable $\CC^d$ functions space. 

Non-negativity is obvious from $D_M(t) \geq 0$ for $\tr_M$ a.e. $t$. Sesquilinearity and conjugate-symmetry is obvious by the same properties of $\langle \cdot, \cdot \rangle$ and by the self-adjointness of $D_M(t)$ for $\tr_M$ a.e. $t \in \Omega$.
\end{proof}

Note that (obviously ...), for any $f \in \calL^2(M)$
\begin{equation}
	\label{eq:III:2.5'}
	\dScalP{f}_{M} = \dScalP{f,f}_{M}
\end{equation}
which explains that the ambiguity of the notation $\dScalP{\cdot}_{M}$ is not very serious...

We can use now the following abstract result (more or less "well-known" ...).

\begin{proposition}[On the "quotient scalar product space"] \label{prop:III:11}
Let $\calL$ be a linear space over $\CC$ and let $\dScalP{\cdot, \cdot}$ be a sesquilinear, non-negative conjugate-symmetric\footnote{\label{fn:III:5} In fact, non-negativity follows from sesquilinearity + non-negativity (by polarisation formula).} form\footnote{\label{fn:III:6} such a form (ses., non-neg., con-symm.) is also called "semi scalar product".} on $\calL$ and let $\tnorm{\cdot} : \calL \to [0; +\infty)$ be given by
\[
	\tnorm{x} := \dScalP{x,x}^{1/2}, \quad
	x \in \calL.
\]
Then:
\begin{enumerate}[(1)]
\item $\tnorm{\cdot}$ is a seminorm in $\calL$;
\item the "semi-Schwarz" inequality holds:
\[
	|\dScalP{x,y}| \leq \tnorm{x} \cdot \tnorm{y}, \quad
	x,y \in \calL;
\]
\item $\calL_0 := \{ x \in \calL : \tnorm{x} = 0 \} \underset{\lin}{\subset} \calL$;
\item The quotient space $L := \calL/\calL_0$ is a normed scalar product space with the norm $\| \cdot \|$ and the scalar product $\langle \cdot, \cdot \rangle$ defined (and "well-defined") by
\[
	\| [x] \| := \tnorm{x}, \quad
	\langle [x], [y] \rangle := \dScalP{x,y}, \quad
	x,y \in \calL;
\]
\item the norm $\| \cdot \|$ in $L$ is the norm induced by $\langle \cdot, \cdot \rangle$ in $L$, i.e. $\| [x] \|^2 = \langle [x], [x] \rangle$ for any $[x] \in L$.
\end{enumerate}
\end{proposition}
We skip the $\pm$ standard proof and leave it as an exercise.

This result is just what we need, to construct the space $L^2(M)$, which we shall describe in the next subsection in more details. But by (1) we see in particular, that $\tnorm{\cdot}_M$ induced in \eqref{eq:III:2.2} is a seminorm in $\calL^2(M)$. Following the notation of (3) we define:
\begin{equation}
	\label{eq:III:2.6}
	\calL_0^2(M) := \{ f \in \calL^2(M) : \tnorm{f}_M = 0 \}.
\end{equation}
Below we characterize this subspace in terms of the trace measure and the trace density of $M$.

\begin{fact}["On the zero layer"] \label{fact:III:12}
Let $f : \Omega \to \CC^d$ be $\frakM$-measurable. TFCAE:
\begin{enumerate}[(a)]
\item $f \in \calL^2_0(M)$,
\item[(a')] $\dScalP{f}_M = 0$,
\item $f(t) \in \Ker D_M(t)$ for $\tr_M$-a.e. $t \in \Omega$.
\end{enumerate}
\end{fact}
We shall need an abstract result to prove this fact.

\begin{lemma} \label{lem:III:13}
Let $A$ be a non-negative bounded operator in a complex Hilbert space $\calH$. Then for $x \in \calH$ TFCAE:
\begin{enumerate}[(i)]
\item $x \in \Ker A$
\item $x \in \Ker \sqrt{A}$
\item $\langle A x, x \rangle = 0$.
\end{enumerate}
\end{lemma}
\begin{proof}
Exercise. Hint: $\langle A x, x \rangle = \big\| \sqrt{A} x \big\|^2$.
\end{proof}

\begin{proof}[Proof of Fact~\ref{fact:III:12}]
\fbox{(a) $\Leftrightarrow$ (b)} -- is obvious by the definition of $\calL^2(M)$ and by \eqref{eq:III:2.2}. Also \fbox{(b) $\Rightarrow$ (a')} is simple since denoting $\Omega_0 = \{ t \in \Omega : f(t) \in \Ker D_M(t) \}$ we get
\[
	\dScalP{f}_M = 
	\int_{\Omega} \big\langle D_M(t) f(t), f(t) \big\rangle_{\CC^d} \ud \tr_M(t) =
	\int_{\Omega \setminus \Omega_0} \big\langle D_M(t) f(t), f(t) \big\rangle_{\CC^d} \ud \tr_M(t)
\]
which equals to $0$ if $\tr_M(\Omega \setminus \Omega_0) = 0$ (i.e. b)) holds.

To get \fbox{(a') $\Rightarrow$ (b)} assume (a'), we have
\[
	0 = \dScalP{f}_M =
	\int_{\Omega} \langle D_M(t) f(t), f(t) \rangle_{\CC^d} \ud \tr_M(t)
\]
but the function $\Omega \ni t \mapsto \langle D_M(t) f(t), f(t) \rangle_{\CC^d}$ is non-negative for $\tr_M$ a.e. $t \in \Omega$, so when the integral is $0$, then also $\langle D_M(t) f(t), f(t) \rangle_{\CC^d} = 0$ for $\tr_M$ a.e. $t$. Now, using Lemma~\ref{lem:III:13} for $A=D_M(t)$ we get $f(t) \in \Ker D_M(t)$ for $\tr_M$ a.e. $t$, i.e. (b) holds.
\end{proof}

Let us discuss now "the problem of being in $\calL_0^2(M)$".

The natural question is \emph{whether in $\calL^2(M)$ the characterization for $\calL^2_0(M)$, analogous to the "scalar" case}, i.e., in $\calL^2(\nu)$ for $\calL^2_0(\nu) = \{ f \in \calL^2(\nu) : \int_{\Omega} |f|^2 \ud \nu = 0 \}$, \emph{is true}? Recall, that when $\nu$ is a measure in $\Omega$ (on $\frakM$), then for $f \in \calL^2(\nu)$ we have $f \in \calL^2_0(\nu)$ iff
\begin{equation}
	\label{eq:III:2.7}
	\nu \big( \{ t \in \Omega : f(t) \neq 0 \} \big) = 0.
\end{equation}
So, we could expect, that "the analogous condition" for the matrix measure $\calL^2(M)$-space would be
\begin{equation}
	\label{eq:III:2.7'}
	M \big( \{ t \in \Omega : f(t) \neq 0 \} \big) = 0
\end{equation}
(here $f \in \calL^2(M)$).

"Unfortunately", the answer for such a question is "NOT"!

\begin{remark} \label{rem:III:14}
Let $M$ be a $d \times d$ matrix measure on $\frakM$ and let $f \in \calL^2(M)$. Denote:
\begin{align*}
	\supp_{\calL^2(M)}(f) 
	&:= 
	\{ t \in \Omega : f(t) \notin \Ker D_M(t) \}, \\
	\supp_M(f) 
	&:=
	\{ t \in \Omega : f(t) \neq 0 \text{ and } D_M(t) \neq 0 \}.
\end{align*}
Then:
\begin{enumerate}[(a)]
\item $\supp_{\calL^2}(f) \subset \supp_M(f)$,
\item $f \in \calL^2_0(M) \Leftrightarrow \tr_M \big( \supp_{\calL^2(M)}(f) \big) = 0$,
\item $f$ satisfies \eqref{eq:III:2.7'} $\Leftrightarrow$ $\tr_M \big( \supp_M(f) \big) = 0$,
\item \eqref{eq:III:2.7'} $\Rightarrow f \in \calL_0^2(M)$ but "$\Leftarrow$" generally does not hold, if $d > 1$.
\end{enumerate}
\end{remark}
\begin{proof}
(a) is clear (since $\supset$ holds for the complements), (b) is just Fact~\ref{fact:III:12} and (c) follows from Remark~\ref{rem:III:8}. Hence "$\Rightarrow$" is true. To see, that "$\Leftarrow$" can be untrue, let us see the following example.
\end{proof}

\begin{example} \label{example:III:15}
Let: $\Omega = [-1;1]$, $\frakM = \Bor([-1;1])$, $\mu$ -- the Lebesgue measure "restricted" to $[-1;1]$, $d=2$, $F : \Omega \to \Mat{2}{\CC}$
\[
	F(t) := 
		\begin{bmatrix}
			1 & 0 \\
			0 & 0
		\end{bmatrix} \chi_{[-1; 0]}(t)
		+
		\begin{bmatrix}
			0 & 0 \\
			0 & 1
		\end{bmatrix} \chi_{(0; 1]}(t)
\]
and let $f : \Omega \to \CC^2$
\[
	f(t) := 
		\begin{pmatrix}
			0 \\
			1
		\end{pmatrix} \chi_{[-1; 0]}(t)
		+
		\begin{pmatrix}
			1 \\
			0
		\end{pmatrix} \chi_{(0; 1]}(t).
\]
Consider $M = F \ud \mu$ -- see Example~\ref{example:III:7}. We have $\tr F \equiv 1$ so $\tr_M = \mathds{1} \ud \mu = \mu$ and $D_M = F$. Moreover, 
\[
	\dScalP{f}_M = 
	\frac{1}{2} \int_{[-1;1]} \langle D_M(t) f(t), f(t) \rangle_{\CC^2} \ud t = 0,
\]
so $f \in \calL^2(M)$, and $f \in \calL^2_0(M)$ since $\supp_{\calL^2_M}(f) = \emptyset$. And we have $\supp_M(f) = \Omega$, so $\tr_M \big( \supp_M(f) \big) = \mu \big( \supp_M(f) \big) = \mu(\Omega) = 2$. And
\[
	M \big( \{ t \in \Omega : f(t) \neq 0 \big) = 
	M(\Omega) =
	\int_{[-1,0]} 
		\begin{bmatrix}
			1 & 0 \\
			0 & 0
		\end{bmatrix}
	\ud t + 
	\int_{[0,1]} 
		\begin{bmatrix}
			0 & 0 \\
			0 & 1
		\end{bmatrix}
	\ud t
	=
	\begin{bmatrix}
		1 & 0 \\
		0 & 0
	\end{bmatrix}
	+
	\begin{bmatrix}
		0 & 0 \\
		0 & 1
	\end{bmatrix}
	=
	I.
\]
So, here $\tnorm{f}_M = 0$, but $M \big( \{t \in \Omega : f(t) \neq 0 \} \big) = I$. I.e., some "trivial" elements of $\calL^2(M)$ can be "nontrivial" functions from the point of view of the matrix measure $M$ (or the measure $\tr_M$). Note also, that our matrix measure is not only "some" matrix measure. -- It is a so-called \emph{probabilistic matrix measure}, i.e. it satisfies an extra condition
\begin{equation}
	\label{eq:III:2.8}
	M(\Omega) = I.
\end{equation}
\end{example}

This is one of the reasons, which make the construction of the vectors of $L^2(M)$ from functions of $\calL^2(M)$ much more delicate/sophisticated, than the well-known construction of the vectors of $L^2(\nu)$ from functions of $\calL^2(\nu)$.

Mathematicians usually do not formally distinguish between $\calL^2(\nu)$ and $L^2(\nu)$, i.e. they say/write a "function f from $L^2(\nu)$" instead of a "class of functions $[f]$", and everybody is used to it, and understands what is "going on". But, because of the above delicacy (subtlety...) of the elements of $\calL_0(M)$, we \emph{shall not} use the similar informality for $L^2(M)$ space!

\vspace{5ex}

\subsection{The $L^2(M)$ space. Completeness.} \label{sec:III:3}
Let us fix a $d \times d$ matrix measure $M$ on $\frakM$ -- a $\sigma$-algebra for $\Omega$. We follow the construction of scalar product space from a semi scalar one. In subsection~\ref{sec:III:2} the linear space $\calL^2(M)$ with the "semi scalar product" $\dScalP{\cdot, \cdot}_M$ and the related seminorm $\tnorm{_M}$ were defined. We also defined $\calL_0^2(M) := \{ f \in \calL^2(M) : \tnorm{f}_M = 0 \}$. Hence, according to Proposition~\ref{prop:III:11}, defining the new space
\begin{equation}
	\label{eq:III:3.1}
	\begin{split}
		&L^2(M) := \calL^2(M) / \calL^2_0(M) \\
		&\langle \cdot, \cdot \rangle_M : L^2(M) \times L^2(M) \to \CC \\
		&\langle [f], [g] \rangle_M := \dScalP{f,g}_M, \quad 
		[f],[g] \in L^2(M)
	\end{split}
\end{equation}
we obtain a scalar product space with the norm
\begin{equation}
	\label{eq:III:3.2}
	\| [f] \|_M := \tnorm{f}_M = \dScalP{f,f}_M^{1/2}, \quad
	[f] \in L^2(M).
\end{equation}
Recall, that by the definition of the quotient linear space, here, for any $f \in \calL^2(M)$
\begin{equation}
	\label{eq:III:3.3}
	[f] := \big\{ f+g : h \in \calL^2_0(M) \big\},
\end{equation}
and\footnote{\label{fn:III:7} In particular, the operation "$[]$" depends on the choice of $M$, similarly as in the case of $L^2(\mu)$ spaces, but we omit $M$ (or $\mu$) in the notation when it does not lead to any misunderstandings.} the linear operations on the classes -- the vectors of $L^2(M)$ -- are as follows: $[f] + [g] := [f+g]$ and $z \cdot [f] := [z f]$ for $[f], [g] \in L^2(M)$.

Our main goal now is to prove that in fact $L^2(M)$ is a Hilbert space.

Our proof is based on a convenient transformation of elements of $\calL^2(M)$, using the following simple, but crucial
\begin{observation} \label{obs:III:3}
If $f \Omega \to \CC^d$ is measurable, then (see \eqref{eq:III:2.1}) if $f \in \calL^2(M)$, then
\begin{align}
	\label{eq:III:3.3.1}
	\tnorm{f}_M^2 = \dScalP{f}_M 
	&= 
	\int_{\Omega} \langle D_M(t) f(t), f(t) \rangle_{\CC^d} \ud \tr_M(t) \\
	\label{eq:III:3.3.2}
	&=
	\int_{\Omega} \langle \sqrt{D_M(t)} f(t), \sqrt{D_M(t)} f(t) \rangle_{\CC^d} \ud \tr_M(t) \\
	\label{eq:III:3.3.3}
	&=
	\int_{\Omega} \sum_{j=1}^d \big| \big( \sqrt{D_M(t)} f(t) \big)_j \big|^2 \ud \tr_M(t) \\
	\label{eq:III:3.3.4}
	&=
	\sum_{j=1}^d \dScalP{(\calF f)_j}_{\tr_M},
\end{align}
where $\calF$ denotes now a "new" function, $\calF: \Omega \to \CC^d$, given by
\begin{equation}
	\label{eq:III:3.4}
	(\calF f)(t) := \sqrt{D_M(t)} f(t), \quad
	t \in \Omega
\end{equation}
and for any measure $\mu$ on $\frakM$ and a measurable $g : \Omega \to \CC$ we denote (see also Section~\ref{sec:III:2})
\begin{equation}
	\label{eq:III:3.5}
	\dScalP{g}_\mu := \int_{\Omega} |g|^2 \ud \mu \in [0, \infty].
\end{equation}
\end{observation}
Recall that $\calL^2(\mu) = \{ g : \Omega \to \CC : g\text{-measurable and } \dScalP{f}_\mu \in \RR \}$ and $\tnorm{g}_\mu = \dScalP{g}_\mu^{1/2}$ for $g \in \calL^2(\mu)$. Note, that -- as usual -- we choose $D_M \in \DD_M$ such, that $\forall_{t \in \Omega} \ D_M(t) \geq 0$, so $\calF f : \Omega \to \CC^d$ is well (and "everywhere") defined. But this "simple calculation" \eqref{eq:III:3.3.1}--\eqref{eq:III:3.3.4} is somewhat delicate. We should be careful, because the last step \eqref{eq:III:3.3.3}--\eqref{eq:III:3.3.4} \emph{is not obvious}, and need special explanation!

The question is whether we "have right" to claim, that $\dScalP{(\calF f)_j}_{\tr_M}$ have sense for any $j=1,\ldots,d$. And this is not a "quantitative" problem, but rather a "qualitative" one -- the problem of measurability of all the functions $(\calF f)_{j}$, $j=1,\ldots,d$ (i.e. -- the measurability with respect to $\frakM$ -- in $\Omega$ and the Borel $\sigma$-algebra -- in $\CC$) or -- equivalently -- the measurability of $\calF f$ from $\Omega$ into $\CC^d$ (i.e. -- with respect to $\frakM$ -- in $\Omega$ and the Borel $\sigma$-algebra -- in $\CC^d$). Note, that by the measurability of $f$, we can be sure that the lines \eqref{eq:III:3.3.1}--\eqref{eq:III:3.3.3} are OK, because the function given by $\Omega \ni t \mapsto \sum_{j=1}^d \big| \big( \calF f(t) \big)_j \big|^2$ is measurable, which does not "automatically" give the measurability of each term $(\calF f)_j$. But knowing that they all are measurable, we also get the crucial step \eqref{eq:III:3.3.3}--\eqref{eq:III:3.3.4} (-- being this more "quantitative" problem).

So, to get the measurability of $\calF$, we shall use the following general result for bounded self-adjoint operators in a Hilbert space $\calH$ (here we shall use it just for a.s. matrices from $\Mat{d}(\CC)$, i.e. $\calH = \CC^d$).

\begin{lemma} \label{lem:III:16}
Let $\emptyset \neq J \subset$ be an interval (= a convex subset) and $f : J \to \CC$ -- a continuous function. Consider the transformation $\Phi_f : \calB_J(\calH) \to \calB(\calH)$, where $\calB_J(\calH) := \{ A \in \calB(\calH) : A \text{ is s.a and } \sigma(A) \subset J \}$, given by $\Phi_f(A) = f(A)$\footnote{\label{fn:III:8} Note, that $\sigma(A)$ is compact, $\sigma(A) \subset J$, so $\restr{f}{\sigma(A)}$ is a bounded function, hence $f(A) \in \calB(\calH)$.}. Then $\Phi_f$ is continuous in operator-norm topology in $\calB(\calH)$ and in the subspace topology sense in $\calB_J(\calH) \subset \calB(\calH)$.
\end{lemma}
\begin{proof}

We leave this result temporarily as an exercise for the reader, however  the proof is planned in one of the next versions (in Appendix).
\end{proof}

Now, we shall use the lemma to $J := [0; \infty)$ and $f : J \to \RR$, $f(\cdot) = \sqrt{\cdot}$. We know (see the construction of $D_M$ in Section~\ref{sec:III:1}), that $D_M : \Omega \to \Mat{d}{\CC}$ is measurable (in $\frakM-\Bor(\Mat{d}{\CC})$ sense) and $\forall_{t \in \Omega} \ D_M(t) \in \calB_{[0, \infty)}(\CC^d)$. By Lemma~\ref{lem:III:16} $\sqrt{\cdot} : \calB_{[0, \infty)}(\CC^d) \to \Mat{d}{\CC} = \calB(\CC^d)$ is continuous, hence also Borel, so the composition $\sqrt{D_M} : \Omega \to \Mat{d}{\CC}$ is also measurable. Therefore, for fixed measurable $f$ also $\calF f : \Omega \to \CC^d$ given by \eqref{eq:III:3.4} is measurable. Obviously, the measurability of $\calF$ is equivalent to the measurability of each of $(\calF f)_j$, $j=1,\ldots,d$. So, we have proven the following result.

\begin{fact} \label{fact:III:17}
Suppose that $D_M \in \DD_M$ is chose in such a way, that
\begin{equation}
	\label{eq:III:3.6}
	\forall_{t \in \Omega} \ D_M(t) \geq 0.
\end{equation}
Then:
\begin{enumerate}[1)]
\item the matrix valued function $\sqrt{D_M} : \Omega \to \Mat{d}{\CC}$ is measurable
\item If $f : \Omega \to \CC^d$ is measurable, then $\calF f : \Omega \to \CC^d$ is also measurable.
\item If $f : \Omega \to \CC^d$ is measurable, then
\begin{enumerate}[(a)]
\item $f \in \calL^2(M) \Leftrightarrow (\calF f) \in \big( \calL^2(\tr_M) \big)^d$
\item $f \in \calL^2_0(M) \Leftrightarrow (\calF f) \in \big( \calL^2_0(\tr_M) \big)^d$
\item $\dScalP{f}_M = \sum_{j=1}^d \dScalP{f_j}_{\tr_M}$
\end{enumerate}
\end{enumerate}
\end{fact}

Note, that by (c), if $f \in \calL^2(M)$, i.e. $[f] \in L^2(M)$, then
\begin{equation}
	\label{eq:III:3.7}
	\| [f] \|_{M} = 
	\tnorm{f}_M = 
	\tnorm{\calF f}_{\tr_M,d} =
	\| \bClassAbs{\calF f} \|_{\tr_M,d},
\end{equation}
where we use the following notation and conventional for any measure $\mu$ on $\frakM$ and for the functions\footnote{\label{fn:III:9} Recall the set theory notation: $X^Y$ is the set of \emph{all} the functions $f : Y \to X$, for sets $X,Y$.} $g \in (\CC^d)^\Omega$ and their classes.
\begin{enumerate}[(i)]
\item We identify, as usual any $g : \Omega \to \CC^d$ with the system $(g_1, \ldots, g_d) \in (\CC^\Omega)^d$ to omit the use of any special notation for some "trivial" mappings.
\item  We denote by $\bClassAbs{g}$\footnote{\label{fn:III:10} Observe that we use bold $\bClassAbs{\cdot}$ instead of the usual one for $L^2(M)$} the class of functions determined by $g$ in $\big( L^2(\mu) \big)^d$ for $g \in \big( \calL^2(\mu) \big)^d$, where similarly as in (i) we identify the Hilbert space product $\big( L^2(\mu) \big)^d$ of $d$-times $L^2(\mu)$ with the quotient space $\big( \calL^2(\mu) \big)^d / \big( \calL^2_0(\mu) \big)^d$ which is a "quotient scalar space" obtained by the construction from Proposition~\ref{prop:III:11} (see (iii) below for the seminorm etc. notation). In particular, $\bClassAbs{g}$ is identified with $([g_1]_\mu, \ldots, [g_d]_\mu)$.\footnote{\label{fn:III:11} It is an easy exercise to check that this identification is in fact an "obvious" unitary transformation of the Hilbert spaces...}
\item Let us denote $\dScalP{g}_{\mu,d} := \sum_{j=1}^d \dScalP{g_j}_\mu \in [0; +\infty]$ for each measurable $g : \Omega \to \CC^d$; for such $g$ we have $g \in \big( \calL^2(\mu) \big)^d$ iff $\dScalP{g}_{\mu,d} \neq +\infty$.

For $g \in \big( \calL^2(\mu) \big)^d$ and $\bClassAbs{g} \in \big( L^2(\mu) \big)^d$ we denote
\begin{equation}
	\label{eq:III:3.8}
	\tnorm{g}_{\mu,d} = 
	\| \bClassAbs{g} \|_{\mu,d} := 
	\dScalP{g}_{\mu,d}^{1/2}
\end{equation}
which finally explains the notation used in \eqref{eq:III:3.7}, with $\mu = \tr_M$. Let us keep here the notation $\mu := \tr_M$ and $D := D_M$ for short.
\end{enumerate}

Finally, after all those "formal-notation" explanations, thanks to the linearity of $f \mapsto \calF f$, we can define our key linear mapping:
\[
	\hat{\calF} : L^2(M) \to \big( L^2(\mu) \big)^d
\]
by the formula
\begin{equation}
	\label{eq:III:3.9}
	\hat{\calF}[f] = \bClassAbs{\calF f}, \quad
	f \in \calL^2(M) \ \big( [f] \in L^2(M) \big).
\end{equation}
We can be sure, by Fact~\ref{fact:III:17}, that it is a properly defined linear transformation and, moreover, it is an isometry by \eqref{eq:III:3.7}! So, since $\big( L^2(\tr_M) \big)^d$ is a Hilbert space (as a product of $d$ Hilbert spaces), we could try to check whether $\hat{\calF}$ is \emph{onto} $\big( L^2(\tr_M) \big)^d$ -- And, if it were true, the proof of the completeness of $L^2(M)$ would be complete ... But, "unfortunately", it is NOT true -- the range $\Ran \hat{\calF}$ is usually somewhat smaller... To understand this better it suffices to observe, that each $\calF f$, for $f \in \calL^2(M)$ satisfies the extra condition\footnote{\label{fn:III:12} There is one "fragility" related to the "$\mu$-a.e $t \in \Omega$" below -- see the end of the proof for explanations...}:
\begin{equation}
	\label{eq:III:3.10}
	\calF f \in \calL^2_D :=
	\big\{
		g \in \big( \calL^2(\mu) \big)^d : g(t) \in \Ran D(t) \text{ for } \mu \text{ a.e. } t \in \Omega
	\big\}.
\end{equation}
To check this, it suffices to recall \eqref{eq:III:3.4} and the fact that for any matrix $A \geq 0$ we have $\Ker \sqrt{A} = \Ker A$ (see e.g. Lemma~\ref{lem:III:13}), which, by s.a. of $\sqrt{A}$ and $A$, gives\footnote{\label{fn:III:13} However this part of argumentation is not very important -- with an equal effect we could put $\Ran \sqrt{D(t)}$ as well, instead of $\Ran D(t)$ in \eqref{eq:III:3.10}.}
\[
	\Ran \sqrt{A} = 
	(\Ker \sqrt{A})^\perp =
	(\Ker A)^\perp =
	\Ran A.
\]
Hence, knowing that it "can happen for many $t$", that $\Ran D(t) \neq \CC^d$, we define a "smaller" space:
\begin{equation}
	\label{eq:III:3.11}
	L^2_D := 
	\big\{ \bClassAbs{g} \in \big( L^2(\mu) \big)^d : g \in \calL^2_D \big\}.
\end{equation}

\begin{fact} \label{fact:III:18}
$L^2_D$ is a closed linear subspace of $L^2(M)$.
\end{fact}
\begin{proof}
Suppose that $g_n \in \calL^2_D$ for each $n \in \NN$ and $g \in \big( \calL^2 \big)^d$, and that $\bClassAbs{g_n} \to \bClassAbs{g}$ in $\big( L^2 \big)^d$. This means that for each $j=1,\ldots,d$ $\tnorm{(g_n)_j - (g)_j}_\mu \to 0$ (see \eqref{eq:III:3.8}). Thus we can use now the well-known result on subsequences of $L^p$-convergent sequences, which are a.e. convergent. We can choose a joint subsequence "on each $j$", i.e. let $\{ k_n \}_{k \in \NN}$ be a strictly increasing sequence from $\NN$ such that
\[
	g_{k_n}(t) \to g(t) \quad
	\text{for } \mu-\text{a.e. } t \in \Omega.
\]
But $\Ran D(t)$ is always a closed set (a linear subspace of $\CC^d$), so $\bigg( \forall_{n \in \NN} \ g_{k_n} \in \calL^2_D \bigg) \Rightarrow g \in \calL^2_D$, hence $\bClassAbs{g} \in L^2_D$. The linearity is obvious by the linearity of $\Ran D(t)$ for each $t$.
\end{proof}

Hence $L^2_D$ is also a Hilbert space, and by \eqref{eq:III:3.9}, \eqref{eq:III:3.10} and \eqref{eq:III:3.11} we have
\[
	\hat{\calF} : L^2(M) \to L^2_D.
\]
Now, to finish our proof of the completeness of $L^2(M)$ it suffices to prove the result below.

\begin{fact} \label{fact:III:19}
$\hat{\calF}$ is a unitary transformation between $L^2(M)$ and $L^2_D$.
\end{fact}
\begin{proof}
Having the isometry of $\hat{\calF}$ just proved, we need only to prove that $\Ran \hat{\calF} \supset L^2_D$.

Let $\bClassAbs{g} \in L^2_D$, i.e. $g \in \calL^2_D$ and define $f : \Omega \to \CC^d$ by the formula
\begin{equation}
	\label{eq:III:3.12}
	f(t) := G \big( D(t) \big) g(t), \quad
	t \in \Omega,
\end{equation}
where (as usual) $G(A)$ denotes the matrix (operator) being the function $G$ of the matrix (operator) $A \geq 0$, $A \in \Mat{d}{\CC}$, where $G : [0; +\infty) \to \RR$ is given by
\begin{equation}
	\label{eq:III:3.13}
	G(x) =
	\begin{cases}
		0 & x = 0 \\
		\frac{1}{\sqrt{x}} & x > 0.
	\end{cases}
\end{equation}
Note, that, as usual we choose $D(t) \geq 0$ for any $t \in \Omega$ (to make sense for \eqref{eq:III:3.12} for any $t$). To finish our proof we have to check:
\begin{align}
	\label{eq:III:3.14}
	&f \in \calL^2(M), \\
	\label{eq:III:3.15}
	&\bClassAbs{\calF f} = \bClassAbs{g},
\end{align}
because with these properties we obtain by \eqref{eq:III:3.9}
\[
	[f] \in L^2(M) \text{ and }
	\bClassAbs{g} = 
	\bClassAbs{\calF f} = 
	\hat{\calF}[f] \in \Ran \hat{\calF}.
\]
To get \eqref{eq:III:3.14} let us first check, that $f$ is measurable. If $G$ were continuous, then we would get is by the same argumentation, which proved that $\calF f$ was measurable for measurable $f$ (i.e. by Lemma~\ref{lem:III:16}). Our $G$ is not continuous, but lets "approximate" it, in a way, by the following continuous functions $G_n : [0, +\infty) \to \CC$:
\[
	G_n(x) :=
	\begin{cases}
		n \sqrt{n} x & \text{for } x < \frac{1}{n} \\
		\frac{1}{\sqrt{x}} & \text{for } x \geq \frac{1}{n}
	\end{cases},
	\quad x \geq 0, n \in \NN.
\]
So, $G_n$ and $G$ are equal after the restriction to $[\frac{1}{n}; +\infty) \cup \{0\} =: R_n$! Thus, for any matrix $A \geq 0$ in $\Mat{d}{\CC}$ we can fix some $N(A) \in \NN$ such that
\[
	\forall_{n \geq N(A)} \ \sigma(A) \subset R_n,
\]
because $\sigma(A)$ is a \emph{finite} set contained in $[0, \infty)$. Hence
\begin{equation}
	\label{eq:III:3.16}
	\forall_{n \geq N(A)} \ G_n(A) = G(A).
\end{equation}
Now defining $f_n : \Omega \to \CC^d$ for any $n \in \NN$ by
\begin{equation}
	\label{eq:III:3.17}
	f_n(t) := G_n \big( D(t) \big) g(t), \quad
	t \in \Omega
\end{equation}
we see that $f_n$ is measurable, by the continuity of $G_n$, and by \eqref{eq:III:3.16}, \eqref{eq:III:3.17} and \eqref{eq:III:3.12} $\{f_n \}_{n \in \NN}$ is pointwise convergent to $f$ because
\[
	\forall_{t \in \Omega} \ 
	\forall_{n \geq N} \big( D(t) \big) \
	f_n(t) = f(t).
\]
Now, having the measurability, by \eqref{eq:III:3.12} we have also
\[
	\dScalP{f}_M = 
	\int_\Omega \| \sqrt{D(t)} f(t) \|_{\CC^d}^2 \ud \mu(t) =
	\int_{\Omega} \|P(t) g(t) \|^2_{\CC^d} \ud \mu(t),
\]
with
\begin{equation}
	\label{eq:III:3.18}
	P(t) = \sqrt{D(t)} \cdot G \big( D(t) \big).
\end{equation}
By STh+FCTh for matrix $D(t)$, using the fact that for any $x \in [0; +\infty)$ we have 
\[
	\sqrt{x} G(x) = 
	\begin{cases}
		0 & x=0 \\
		1 & x > 0
	\end{cases}
\]
we see that $P(t) = E_{D(t)} \big( (0; +\infty) \big) = I - E_{D(t)}(\{0\}) = I - P_0(t)$ where $P_0(t) = E_{D(t)}(\{0\})$ is simply the orthogonal projection in $\CC^d$ on $\Ker D(t)$. Thus $P(t)$ is the orthogonal projection onto $\Ran D(t) = \big( \Ker D(t) \big)^\perp$. In particular for any $t \in \Omega$ $\| P(t) g(t) \|_{\CC^d} \leq \| g(t) \|_{\CC^d}$, so finally
\[
	\dScalP{f}_M \leq \dScalP{g}_{\mu, d} < \infty,
\]
and \eqref{eq:III:3.14} holds. But by \eqref{eq:III:3.4}, \eqref{eq:III:3.12} and \eqref{eq:III:3.18}
\[
	\forall_{t \in \Omega} \ 
	(\calF f)(t) = 
	\sqrt{D(t)} G \big( D(t) \big) g(t) = 
	P(t) g(t)
\]
and we assumed, that $g \in \calL^2_D$ which means that $g(t) \in \Ran D(t) = \Ran P(t)$ for $\mu$-a.e. $t \in \Omega$. Now, taking 
\[
	\omega := \{ t \in \Omega : g(t) \in \Ran D(t) \},
\]
and changing the values of $g$ in each $t \in \Omega \setminus \omega$ onto $0$ (which belongs to $\Ran D(t)$) we get
\[
	\calF f = g
\]
without changing the class $\bClassAbs{g}$, so \eqref{eq:III:3.15} also holds.

But ... -- be careful! -- Thee is still one "fragility"! We knew only, that $g \in \Ran D(t)$ for $\mu$ a.e. $t \in \Omega$ before this change on $\Omega \setminus \omega$ and this means ONLY, that $\Omega \setminus \omega$ is contained in a $\mu$-zero measure set. But $\mu$ can be not a complete measure, so it does not automatically means, that $\Omega \setminus \omega \in \frakM$ (equivalently $\omega \in \frakM$). And we need it, to know, that after our change we get really a measurable function $g$...

So to end the proof in all the details, we still need the following abstract result
\begin{lemma} \label{lem:III:20}
If $A : \Omega \to \Mat{d}{\CC}$ and $f : \Omega \to \CC^d$ are measurable, then the set $\{ t \in \Omega : f(t) \in \Ran A(t) \}$ is measurable.
\end{lemma}
\begin{proof}
An exercise. Some key hints. Let $v \in \CC^d$ and $B \in \Mat{d}{\CC}$
\begin{enumerate}[1)]
\item $v \in \Ran B$ iff $v \in \lin ( \{ B_1, \ldots, B_d \} )$ iff $\dist{v, \lin ( \{ B_1, \ldots, B_d \} } = 0$, where $B_i$ -- the $i$-th column of $B$
\item there exists a subset of $\{ B_1, \ldots, B_d \}$ forming a base system for $\lin ( \{ B_1, \ldots, B_d \} )$ (including the case of the $\emptyset$ subset for $\{0\}$ space ...)
\item there exists an "analytic" formula for $\dist{v, \lin ( \{ w_1, \ldots, B_k \} }$, when $(w_1, \ldots, w_k)$ is linearly independent (the formula uses: Gram determinants for some subsystems of vectors formed from $v, w_1, \ldots, w_k$). 
\end{enumerate}
\end{proof}
\end{proof}

\vspace{5ex}

\subsection{The subspace $S(M)$ of simple "functions" and $L^2_\Sigma(M)$} \label{sec:III:4}
We consider here two important subspaces of $L^2(M)$ for a matrix measure $M$ on $\frakM$ -- a $\sigma$-algebra of subsets of a set $\Omega$ (as before -- we fix here $\Omega$, $\frakM$ and $M$). The first $S(M)$ is the subspace consisting of $\CC^d$-simple "functions" (i.e., of their classes in $L^2(M)$ sense), and the second, larger one, is denoted by $L^2_\Sigma(M)$ (and does not have any special name...). It is worth to note, that $L^2_\Sigma(M)$ plays sometimes the similar role in $L^2(M)$ to the role of the Schwarz functions class in the standard $L^2(\RR)$ ($\RR$ with the Lebesgue measure) space. We shall prove here, that both subspaces are dense in $L^2(M)$.

Let us consider first so-called "vector characteristic functions", being just products $\chi_\omega c$ of scalar characteristic functions $\chi_\omega$ for $\omega \subset \Omega$ by vectors $c \in \CC^d$, i.e. $\chi_\omega c : \Omega \to \CC^d$,
\[
	(\chi_\omega c)(t) =
	\begin{cases}
		c & \text{for } t \in \omega \\
		0 & \text{for } t \notin \omega
	\end{cases}, \quad
	t \in \Omega,
\]
so $\chi_\omega$ is measurable, when $\omega \in \frakM$. Denote:
\begin{itemize}
\item $\VCF(\frakM) := \{ \chi_\omega c : c \in \CC^d, \omega \in \frakM \}$
\item $\calS(\frakM) := \lin \VCF(\frakM)$.
\end{itemize}

For any $\omega \in \frakM$ and $c \in \CC^d$, by Proposition~\ref{prop:III:5},
\begin{equation}
	\label{eq:III:4.1}
	\begin{split}
		\dScalP{\chi_\omega c}_M 
		&:=
		\int_\Omega \big\langle D_M(t) (\chi_\omega c)(t), (\chi_\omega c)(t) \big\rangle_{\CC^d} \ud \tr_M(t) \\
		&=
		\int_{\omega} \langle D_M(t) c, c \rangle_{\CC^d} \ud \tr_M(t) 
		\leq
		\|c\|_{\CC^d}^2 \cdot \tr_M(\omega).
	\end{split}
\end{equation}
Hence $\VCF(\frakM)$ and $\calS(\frakM)$ are included in $\calL^2(M)$, and $\calS(\frakM) \underset{\lin}{\subset} \calL^2(M)$.

Let us make here several simple (nomen omen...) observations.

\begin{fact} \label{fact:III:21}
\begin{enumerate}[1)]
\item $f \in \calS(\frakM) \Leftrightarrow \forall_{j=1,\ldots,d} \ f_j$ is a scalar ($\frakM$-)simple function\footnote{\label{fn:III:14} scalar simple function ($\frakM$-simple) -- i.e. a linear complex combination of characteristic functions $\chi_\omega$ with $\omega \in \frakM$ (or -- equivalently -- an $\frakM$-measurable function $g : \Omega \to \CC$ with $g(\Omega)$ -- finite).} $\Leftrightarrow$ $f$ is a $\CC^d$-vector $\frakM$-measurable function with finite image $f(\Omega)$.

\item If $\omega \in \frakM$, $c,c' \in \CC^d$, then
\begin{align}
	\label{eq:III:4.2}
	&\langle M(\omega) c, c' \rangle_{\CC^d} = \int_{\omega} \langle D_M(t) c, c' \rangle_{\CC^d} \ud \tr_M(t), \\
	\label{eq:III:4.3}
	&\dScalP{\chi_\omega c}_M = \tnorm{\chi_\omega c}^{2}_M = \langle M(\omega) c, c \rangle_{\CC^d}, \\
	\label{eq:III:4.4}
	&M(\omega) c = \int_{\omega} D_M(t) c \ud \tr_M(t).
\end{align}
\end{enumerate}
\end{fact}
\begin{proof}
Easy exercise.
\end{proof}

Now we define the subspace $S(M) \underset{\lin}{\subset} L^2(M)$ in the natural way\footnote{\label{fn:III:15} Note here a difference in notation: we used $\frakM$ for $\calS(\frakM)$, because it was determined only by $\frakM$, and not on the choice of $M$ on $\frakM$, but $M$ in $S(M)$ is important, because $[f]$ is the class in the "$L^2(M)$" sense, although we did not show this dependence on $M$ for the $[\cdot]$ notation.}
\begin{equation}
	\label{eq:III:4.5}
	S(M) := \big\{ [f] \in L^2(M) : f \in \calS(\frakM) \big\}.
\end{equation}
Let us denote by $\pi_M$ the quotient mapping
\[
	\pi_M : \calL^2(M) \to L^2(M),
\]
given by
\[
	\pi_M(f) = [f], \quad
	f \in \calL^2(M)
\]
By the diagram
\[
	\begin{tikzcd}
		\calL^2(M) \arrow[d, "\pi_M"] & [-28pt]\underset{\lin}\supset & [-28pt] \calS(\frakM) \ar[d, "\pi_M"] & [-28pt]= & [-28pt]\lin \VCF(\frakM) \ar[d, "\pi_M"] \\
		L^2(M) & [-28pt]\underset{\lin}\supset & S(M) &  [-28pt]= & [-28pt]\lin[\VCF]
	\end{tikzcd} 
\]
and linearity of $\pi_M$ we easily obtain 
\begin{equation}
	\label{eq:III:4.6}
	S(M) = \pi_M \big( \calS(\frakM) \big) = \lin [\VCF],
\end{equation}
where
\[
	[\VCF] := \big\{ [f] \in L^2(M) : f \in \VCF(\frakM) \} =
	\pi_M \big( \VCF(\frakM) \big).
\]

Let us define now the second subspace -- $L^2_\Sigma(M)$. First observe that if $f,g : \Omega \to \CC^d$ then for each $t \in \Omega$ and $\omega \in \frakM$
\begin{equation}
	\label{eq:III:4.7}
	\begin{split}
		\big\langle M(\omega) f(t), g(t) \rangle_{\CC^d} 
		&=
		\sum_{i=1}^d \big( M(\omega) f(t) \big)_i \overline{g_i(t)} \\
		&=
		\sum_{i,j=1,\ldots,d} f_j(t) \overline{g_i(t)} M_{ij}(\omega).
	\end{split}
\end{equation}
Now, inspired by the above formula, we would like to "integrate somehow with respect to $M$, instead of using a fixed $t$ and $\omega$", to get a properly defined semi-scalar product defined on a certain "large" set of functions $\calL^2_\Sigma(M)$ from $\Omega$ into $\CC^d$. Namely we would like to define $\dScalP{\cdot, \cdot}_{\Sigma} : \calL^2_\Sigma(M) \times \calL^2_\Sigma(M) \to \CC$ by (compare to \eqref{eq:III:4.7})
\begin{equation}
	\label{eq:III:4.8}
	\dScalP{f, g}_{\Sigma} :=
	\sum_{i,j=1,\ldots, d} \int_{\Omega} f_j \overline{g_i} \ud M_{ij}, \quad
	f,g \in \calL^2_\Sigma(M),
\end{equation}
but we should properly define the space $\calL^2_\Sigma(M)$ of functions. In particular, each integral from the RHS has to be a well-defined complex number. Recall\footnote{\label{fn:III:16} See Proposition~\ref{prop:III:2}.} that $M_{ij}(\omega) = \big(M(\omega) \big)_{ij}$, and each $M_{ij}$ is a complex measure on $\Omega$, so we need\footnote{\label{fn:III:17} Recall that $h \in \calL^p \Leftrightarrow \int_{_\Omega} |h|^p \ud |\mu| < \infty$ for complex measure $\mu$, where $|\mu|$ is the variation measure for $\mu$ and $p \in [1; +\infty)$ (see also \cite{Diestel1977}).} $f_j \overline{g_i} \in \calL^1(M_{ij})$ for any $i,j=1,\ldots,d$. In particular we need $f_i \overline{f_i} = |f_i|^2 \in \calL^1(M_{ii})$ that is $f_i \in \calL^2(M_{ii})$, equivalently, for any $i=1,\ldots,d$. And "we can be glad" that the above condition is sufficient for us. -- Namely, the following result holds

\begin{proposition} \label{prop:III:22}
If $h_1 \in \calL^2(M_{ii})$, $h_2 \in \calL^2(M_{jj})$ for some $i,j=1,\ldots,d$, then $h_1 h_2 \in \calL^1(M_{ij})$ and
\begin{equation}
	\label{eq:III:4.9}
	\bigg| \int_{\Omega} h_1 h_2 \ud M_{ij} \bigg| \leq
	\int_{\Omega} |h_1 h_2| \ud |M_{ij}| \leq
	\bigg( \int_{\Omega} |h_1|^2 \ud |M_{ii}| \bigg)^{1/2} \cdot
	\bigg( \int_{\Omega} |h_2|^2 \ud |M_{jj}| \bigg)^{1/2}.
\end{equation}
\end{proposition}

This is a direct consequence of the following two lemmas.
\begin{lemma} \label{lem:III:23}
If $M$ is a complex measure on $\frakM$, then for any $i,j = 1,\ldots,d$
\[
	\forall_{\omega \in \frakM} \ 
	|M_{ij}|(\omega) \leq 
	\big( M_{ii}(\omega) \big)^{1/2} \cdot \big( M_{jj}(\omega) \big)^{1/2}.
\]
\end{lemma}
\begin{proof}
Fix $i,j=1,\ldots,d$.

Let $e_s$ be the $s$-th standard base vector in $\CC^d$ for any $s=1,\ldots,d$. For any $\omega \in \frakM$, by the Schwarz inequality\footnote{\label{fn:III:18} Note, that "usually" $|M_{ij}(\omega)| \neq |M_{ij}|(\omega)$ ...}:
\begin{align*}
	|M_{ij}(\omega| 
	&= 
	|\langle M(\omega) e_j, e_j \rangle_{\CC^d}| \\
	&=
	\big| \big\langle \sqrt{M(\omega)} e_j, \sqrt{M(\omega)} e_i \big\rangle_{\CC^d} \big| \\
	&\leq
	\big( \big\langle \sqrt{M(\omega)} e_j, \sqrt{M(\omega)} e_j \big\rangle_{\CC^d} \big)^{1/2} \cdot
	\big( \big\langle \sqrt{M(\omega)} e_i, \sqrt{M(\omega)} e_i \big\rangle_{\CC^d} \big)^{1/2} \\
	&=
	\big( \big\langle M(\omega) e_j, e_j \big\rangle_{\CC^d} \big)^{1/2} \cdot
	\big( \big\langle M(\omega) e_i, e_i \big\rangle_{\CC^d} \big)^{1/2} \\
	&=
	\big( M_{ii}(\omega) \big)^{1/2} \cdot \big( M_{ii}(\omega) \big)^{1/2}.
\end{align*}
Now fix $\omega \in \frakM$ and consider a certain disjoint $\frakM$-decomposition $\omega_1, \ldots, \omega_n$ of $\omega$. Using the above estimate to each of $\omega_s$ ($s=1,\ldots,n$) we get (from the Schwarz inequality, again)
\begin{align*}
	\sum_{s=1}^n \big| M_{ij}(\omega_s) \big| 
	&\leq
	\sum_{s=1}^n \big( M_{ii}(\omega_s) \big)^{1/2} \cdot \big( M_{jj}(\omega_s) \big)^{1/2} \\
	&\leq
	\bigg( \sum_{s=1}^n  M_{ii}(\omega_s) \bigg)^{1/2} \cdot
	\bigg( \sum_{s=1}^n  M_{jj}(\omega_s) \bigg)^{1/2} \\
	&=
	\big( M_{ii}(\omega) \big)^{1/2} \cdot \big( M_{jj}(\omega) \big)^{1/2}.
\end{align*}
So, taking the supremum over all the decompositions, we get the assertion for $\omega$.
\end{proof}

\begin{lemma}["Generalized Schwarz inequality"] \label{lem:III:24}
Suppose, that $\mu, \mu_1, \mu_2$ are measures on $\frakM$, satisfying
\[
	\forall_{\omega \in \frakM} \ 
	\mu(\omega) \leq 
	\big( \mu_1(\omega) \mu_2(\omega) \big)^{1/2}.
\]
If $f_i \in \calL^2(\mu_i)$, $i=1,2$, then $f_1 \cdot f_2 \in \calL^1(\mu)$ and
\[
	\int_{\Omega} |f_1 f_2| \ud \mu \leq
	\bigg( \int_{\Omega} |f_1|^2 \ud \mu_1 \bigg)^{1/2} \cdot
	\bigg( \int_{\Omega} |f_2|^2 \ud \mu_2 \bigg)^{1/2}
\]
\end{lemma}
\begin{proof}

We leave this result temporarily as an exercise for the reader, however  the proof is planned in one of the next versions (in Appendix).
\end{proof}

Note, that taking $\mu = \mu_1 = \mu_2$ we get simply "usual Schwarz" result.

Now, having already Lemma~\ref{lem:III:24} we can give the necessary definition of $\calL^2_\Sigma(M)$:
\[
	\calL^2(M) := 
	\big\{ f \in \Omega \to \CC^d : f \text{is } \frakM\text{-measurable and } \forall_{j=1,\ldots,d} \ f_j \in \calL^2(M_{jj}) \big\}.
\]
With such $\calL^2_\Sigma(M)$ the formula \eqref{eq:III:4.8} for $\dScalP{f,g}_\Sigma$ has sense for $f,g \in \calL^2_\Sigma(M)$, by Lemma~\ref{lem:III:24}. Moreover, we have
\begin{fact}	\label{fact:III:25}
\begin{enumerate}[1)]
\item $\calL^2_\Sigma(M) \underset{\lin}{\subset} \calL^2(M)$
\item if $f,g \in \calL^2_\Sigma(M)$ then
\[
	\dScalP{f,g}_\Sigma = \dScalP{f,g}_M,
\]
\item $S(M) \underset{\lin}{\subset} \calL^2_\Sigma(M)$.
\end{enumerate}
\end{fact}
\begin{proof}
Suppose that $f,g \in \calL^2_\Sigma(M)$ . We have
\[
	\dScalP{f,g}_\Sigma =
	\sum_{i,j=1,\ldots,d} \int_{\Omega} f_j \overline{g_i} \ud M_{ij} =
	\sum_{i,j=1,\ldots,d} \int_{\Omega} \big( D_M(t) \big)_{ij} f_j(t) \overline{g_i(t)} \ud \tr_M(t).
\]
In particular (by Proposition~\ref{prop:III:22}) for any $i,j$ 
\[
	\big( D_M(\cdot) \big)_{ij} \cdot f_j \overline{g_i} \in \calL^1(\tr_M),
\] 
so the sum over all $i,j$ too, and
\begin{equation}
	\label{eq:III:4.10}
	\begin{split}
		\dScalP{f,g}_\Sigma 
		&=
		\int_{\Omega} \sum_{i,j=1,\ldots,d} \big( D_M(t) \big)_{ij} f_j(t) \overline{g_i(t)} \ud \tr_M(t) \\
		&=
		\int_{\Omega} \langle D_M(t) f(t), g(t) \rangle_{\CC^d} \ud \tr_M(t).
	\end{split}
\end{equation}
In particular, for $f=g$: $\CC \ni \dScalP{f,f}_\Sigma = \dScalP{f}_M$, so $\dScalP{f}_M < \infty$ for $f \in \calL^2_\Sigma(M)$, i.e. "$\subset$" from 1) is proved, so \eqref{eq:III:2.3} and \eqref{eq:III:4.10} gives 2). Obviously $\calL^2_\Sigma$ is a linear space, because each $\calL^2(M_{ii})$ is. Hence "$\underset{\lin}{\subset}$" in 1) holds.

Now, to prove 3) it suffices to check that for any $\omega \in \frakM$ and $c \in \CC^d$ $\chi_\omega c \in \calL^2_\Sigma$. But for $j=1,\ldots,d$ $(\chi_\omega)_j = c_j \chi_\omega \in \calL^2(M_{jj})$ since $M_{jj}$ is a finite measure (see Proposition~\ref{prop:III:2}d)).
\end{proof}

\begin{remark} \label{rem:III:26}
We see in particular, that the idea to define a semi-scalar product by \eqref{eq:III:4.8} leads just to the product $\dScalP{\cdot,\cdot}_M$, which we knew before and to the space $\calL^2_\Sigma(M)$ being "only" some part of our "main function space" $\calL^2(M)$.

But there is one advantage of the construction $\calL^2_\Sigma(M)$ made above: the definition of the semi-scalar product for the elements of $\calL^2_\Sigma(M)$ can be done without the use of the trace density $D_M$ for $M$.
\end{remark}

Now, analogically as for $S(M)$ we define
\[
	L^2_\Sigma(M) := 
	\pi_M \big( \calL^2_\Sigma(M) \big) =
	\big\{ [f] \in L^2(M) : f \in \calL^2_\Sigma(M) \big\}.
\]
So, by Fact~\ref{fact:III:25}
\begin{equation}
	\label{eq:III:4.11}
	S(M) \underset{\lin}{\subset} L^2(M).
\end{equation}

We end this part by the density result announced at the beginning of this subsection.

\begin{fact} \label{fact:III:27}
Both $S(M)$ and $L^2_\Sigma(M)$ are dense subspaces of $L^2(M)$.
\end{fact}
\begin{proof}
By \eqref{eq:III:4.11} we need only to prove that $\overline{S(M)} = L^2(M)$, i.e., that
\begin{equation}
	\label{eq:III:4.12}
	S(M)^\perp = \{0\}.
\end{equation}
Let $[f] \in S(M)^\perp$; $[VCF] \subset S(M)$, so
\begin{equation}
	\label{eq:III:4.13}
	\forall_{c \in \CC^d} \
	\forall_{\omega \in \frakM} \
	0 = 
	\langle [f], [\chi_\omega c] \rangle_M =
	\dScalP{f, \chi_\omega c}_M =
	\int_\omega \langle D_M(t) f(t), c \rangle_{\CC^d} \ud \tr_M(t).
\end{equation}
For $c \in \CC^d$ consider $\gamma_c : \Omega \to \CC$ given by
\[
	\gamma_c(t) := \langle D_M(t) f(t), c \rangle_{\CC^d}, \quad
	t \in \Omega.
\]
By Lemma~\ref{lem:III:9} we know that $\gamma_c \in \calL^1(\tr_M)$ and by \eqref{eq:III:4.13}
\begin{equation}
	\forall_{\omega \in \frakM} \
	\int_\omega \gamma_c \ud \tr_M = 0,
\end{equation}
and by the well-known fact from measure theory we get $\gamma_c(t) = 0$ for $\tr_M$-a.e. $t \in \Omega$. Hence, defining $Z_c := \{ t \in \Omega : \gamma_c(t) \neq 0 \}$ we get $\tr_M(Z_c) = 0$ for any $c \in \CC^d$. Choose any $\{ c_n \}_{n \geq 1}$ in $\CC^d$ such, that $\{ c_n : n \in \NN \}$ is dense in $\CC^d$ and let $Z = \bigcup_{n \geq 1} Z_{c_n}$. If $t \in \Omega \setminus Z = \bigcap_{n \geq 1} (\Omega \setminus Z_{c_n})$, then
\[
	\forall_{n \in \NN} \
	0 = \gamma_{c_n}(t) = \langle D_M(t) f(t), c_n \rangle_{\CC^d},
\]
so, by the density in $\CC^d$, $D_M(t) f(t) = 0$. But $\tr_M(Z) = 0$ ($Z$ is a countable sum of $\tr_M$-zero sets), hence we proved, that
\[
	f(t) \in \Ker D_M(t) \text{ for } \tr_M \text{-a.e. } t \in \Omega.
\]
Thus means, that $[f] = 0$ by Fact~\ref{fact:III:12}, i.e. \eqref{eq:III:4.12} holds.
\end{proof}

It is worth noting, that in some papers the definition of the space $L^2(M)$ is not rigorous enough. For instance, in some of them the authors define "only" $L^2_\Sigma(M)$ (or even only $\calL^2_\Sigma(M)$) instead of $L^2(M)$. Observe, that $L^2(M)$, itself, is not a Hilbert space often. It cannot be, as a dense subspace of a Hilbert space $L^2(M)$, as long as it is not the whole $L^2(M)$. -- For some $M$ it is true that $L^2_\Sigma(M) = L^2(M)$ (e.g. always for $d=1$ ...), but not for any $M$!

\vspace{5ex}

\subsection{Multiplication operators in $L^2(M)$} \label{sec:III:5}
Multiplication operators by measurable functions in $L^2(\mu)$, for measures $\mu$ are basic operators for Spectral Theory. It is natural to try to find an analog of such kind of operators for matrix measure $L^2$-spaces. 

So, assume, as usual, that $M$ is a $d \times d$ matrix measure $\frakM$ -- a $\sigma$-algebra of subsets of $\Omega$, and let us think first about some transformations of $\calL^2(M)$, which could be possibly treated as "natural generalisations" of multiplications by scalar function $F$ in $\calL^2(\mu)$. The first -- "most general" guess is a "multiplication" by a measurable matrix function $A : \Omega \to \Mat{d}{\CC}$ given by the formula
\[
	\calL^2(M) \underset{\lin}{\supset} \calD \ni f \mapsto Af \in \calL^2(M), \quad
	\text{where }
	(Af)(t) = A(t) f(t).
\]
But here an important problem arises! -- When we want to define the appropriate quotient factorization of such a transformation, we have to be sure, that for $f \in \calL_0(M)$ also $Af \in \calL_0(M)$. That is, assuming that $D_M(t) f(t) = 0$ for $\tr_M$-a.e. $t \in \Omega$, we would like to obtain $D_M(t) A(t) f(t) = 0$ for $\tr_M$-a.e. $t \in \Omega$. For general case of matrix measures $M$, i.e., quite a general class of different trace densities $D_M : \Omega \to \Mat{d}{\CC}$, it would be difficult to get this. -- Something $\pm$ like commutativity of $D_M(t)$ and $A(t)$ seems to be necessary. So -- the form of $A(t)$ should be, in general, "very simple". Like $A(t) = F(t) I$, with a scalar function $F$...

Let $F : \Omega \to \CC$ be a $\frakM$-measurable function. We define the \emph{multiplication by $F$ operator} $T_F$ in $L^2(M)$ as follows\footnote{\label{fn:III:19} for $f : \Omega \to \CC^d$ $F f : \Omega \to \CC$ and $(F f)(t) = F(t) f(t)$ for $t \in \Omega$.} (note, that we extend the meaning of the symbol "$T_F$" onto all the $L^2$ matrix spaces)
\begin{align}
	\label{eq:III:5.1}
	&\Dom(T_F) := \big\{ [f] \in L^2(M) : F f \in \calL^2(M) \big\} \\
	\nonumber
	&T_F[f] := [F f] \quad
	\text{for } [f] \in \Dom(T_F).
\end{align}
We should check, that this is a proper definition of a linear operator $T_F : \Dom(T_F) \to L^2(M)$. So observe first that the related operator $\calT_F$ in $\calL^2(M)$ with the domain $\calD := \big\{ f \in \calL^2(M) : F f \in \calL^2(M) \big\}$, given by $\calT_{F} f = F f$ for $f \in \calD$ is linear, and if $f \in \calL_0^2(M)$, then $D_M(t) f(t) = 0$ for $\tr_M$-a.e. $t \in \Omega$, so, the same is true for $F(t) \cdot D_M(t) f(t) = D_M(t) \big( F(t) f(t) \big)$, which means that\footnote{\label{fn:III:20} Hence, in particular: $\calL_0^2(M) \subset \calD$, and $\calT_F \big( \calL^2_0(M) \big) \subset \calL^2_0(M)$.} $F f \in \calL^2_0(M)$ (see Fact~\ref{fact:III:12}). Hence $[F f]$ does not depend on the choice of $f$ for $[f]$, and $T_F$ is linear.

Let us recall the notion of \emph{the essential value set} $\VE_\mu(F)$ \emph{with respect to} a measure $\mu$ on $\frakM$. We define first its complement:
\[
	\VNE_\mu(F) := \bigcup 
	\big\{ U \subset \CC : U \text{ -- open, } \mu \big( F^{-1}(U) \big) = 0 \big\}
\]
so $\VE_\mu(F) := \CC \setminus \VNE_\mu(F)$.

We repeat the analogous definition for the matrix measure $M$, just by replacing $\mu$ by $M$ above, and we get $\VE_M(F)$, being a kind of ansatz of the set $F(\Omega)$ of all the values of $F$, containing just the "topologically important values of $F$ from the point of view of the matrix measure $M$"\footnote{\label{fn:III:21}Note that, in particular, $\overline{F(\Omega)} \supset \VE_M(F)$, because $\CC \setminus \overline{F(\Omega)} \subset \VNE_M(F)$.} (note that $\VE_M(F)$ is always a closed set in $\CC$, as a complement of an open $\VNE_M(F)$).

But note also, that we do not need to make special studies of essential value set with respect to matrix measures -- we can use just the well-known properties of this classical notion for measures! Indeed: we know that $M(\omega) = 0$ for $\omega \in \frakM$ iff $\tr_M(\omega) = 0$ (see Remark~\ref{rem:III:8}). Hence
\begin{equation}
	\label{eq:III:5.2}
	\VE_M(F) = \VE_{\tr_M}(t)
\end{equation}
(and analogically for $\VNE_M$ ...). In particular, by the Lindel\"of property of the (metric + separable) space $\VNE_M(F)$
\begin{equation}
	\label{eq:III:5.2'}
	\tr_M \Big( F^{-1} \big( \VNE_M(F) \big) \Big) = 0 \quad \text{and} \quad
	M \Big( F^{-1} \big( \VNE_M(F) \big) \Big) = 0.
\end{equation}

We are ready now to formulate the theorem describing some basic operator and spectral properties of multiplication operators in $L^2(M)$ space.

\begin{theorem} \label{thm:III:28}
Assume that $F, G : \Omega \to \CC$ are $\frakM$-measurable. Then:
\begin{enumerate}[1)]
\item $T_F$ is densely defined and closed;

\item $\forall_{\lambda \in \CC \setminus \{ 0 \}} \ T_{\lambda F} = \lambda T_F$ (and\footnote{\label{fn:III:22} $0$ denotes also the zero operator in $L^2(M)$ (with $\Dom(0) = L^2(M)$).} $T_0 = 0$);

\item
	\begin{enumerate}[a)]
	\item $T_F = 0$ iff\footnote{\label{fn:III:23} $\subsetneq$ can happen only when $M(\Omega) = 0$, which means that $\VE_M(F) = \emptyset$.} $\VE_M(F) \subset \{0\}$,
	\item $T_F = T_G$ iff\footnote{\label{fn:III:24} Note the important difference between the condition "$[f] = [g]$" and for $T_F = T_G$!} $F(t) = G(t)$ for $\tr_M$-a.e. $t \in \Omega$.
	\end{enumerate}
	
\item $T_F \in \calB \big( L^2(M) \big) \Leftrightarrow \VE_M(F)$ is bounded. Moreover for $T_F \in \calB \big( L^2(M) \big)$
\begin{equation}
	\label{eq:III:5.3}
	\| T_F \| = \sup \{ |\lambda| : \lambda \in \VE_M(F) \} =: \sup |\VE_M(F)|,
\end{equation}
provided that $L^2(M) \neq \{0\}$.

\item $T_F^* = T_{\overline{F}}$ and $\Dom(T_F^*) = \Dom(T_F)$;

\item $T_F T_G \subset T_{F G}$ and $\Dom(T_F T_G) = \Dom(T_{FG}) \cap \Dom(T_G)$. In particular 
\begin{equation}
	\label{eq:III:5.4}
	T_F T_G = T_{FG} \Leftrightarrow \Dom(T_{FG}) \subset \Dom(T_{G}). 
\end{equation}
\noindent If we assume additionally that $L^2(M) \neq \{ 0 \}$, then:
\item $\sigma(T_F) = \VE_M(F) = \VE_{\tr_M}(F)$;

\item $\sigmaP(T_F) = \big\{ \lambda \in \CC : \tr_M \big( F^{-1}(\{\lambda\}) \big) \neq 0 \big\}$;

\item If $\lambda_0 \in \rho(T_F)$, then
\begin{equation}
	\label{eq:III:5.5.1}
	(T_F - \lambda_0 I)^{-1} = T_H,
\end{equation}
where $H : \Omega \to \CC$ is any measurable function satisfying
\begin{equation}
	\label{eq:III:5.5.2}
	H(t) \big( F(t) - \lambda_0 \big) = 1 \quad \text{for} \tr_M \text{-a.e. } t \in \Omega.
\end{equation}

\item If $F(\Omega) \subset \RR$ then the operator function on the Borel subsets class $\Bor(\RR)$ on $\RR$ given by
\begin{align*}
	&E_{F,M} : \Bor(\RR) \to \calB \big( L^2(M) \big), \\
	&E_{F,M}(\omega) := T_{\chi_{F^{-1}(\omega)}} \quad \text{for } \omega \in \Bor(\RR)
\end{align*}
is the (projection valued) spectral measure (i.e. the resolution of identity) for $T_F$.
\end{enumerate}
\end{theorem}

Observe, that this long formulation is a matrix measure analog of the appropriate classical result for multiplication by function operators in $L^2(\mu)$ with a measure $\mu$. Also the proof is similar.

\begin{proof}
Let us recall the notation $\gamma_{f,g}$ from Section~\ref{sec:III:2}. For any $f,g : \Omega \to \CC^d$ (previously we assumed that $f,g \in \calL^2(M)$, but it is not necessary for the definition only) $\gamma_{f,g} : \Omega \to \CC$ and
\begin{equation}
	\label{eq:III:5.6}	
	\begin{split}
		&\gamma_{f,g}(t) := \langle D_M(t) f(t), g(t) \rangle_{\CC^d}, \quad
		t \in \Omega \\
		&\gamma_f := \gamma_{f,f}.
	\end{split}
\end{equation}
Observe that:
\begin{equation}
	\label{eq:III:5.7}	
	\begin{split}
		&\gamma_{f,g} = \overline{\gamma_{g,f}} \\
		&\gamma_{F \cdot f} = |F|^2 \gamma_{f} \\
		&\gamma_{Ff,g} = \gamma_{f, \overline{F}g} = F \gamma_{f,g}
	\end{split}
\end{equation}
in particular, for any $\omega \subset \Omega$
\begin{equation}
	\label{eq:III:5.8}
	\gamma_{|\chi_\omega f|} = \chi_\omega \gamma_f.
\end{equation}

Now, we shall say something about the proofs of each point of the theorem.

\fbox{Ad. 1)} Take $f \in \calL^2(M)$ and for any $n \in \NN$ consider
\[
	\omega_n := \{ t \in \Omega : |F(t)| \leq n \}.
\]
Recall that for any measurable $g : \Omega \to \CC^d$
\begin{equation}
	\label{eq:III:5.9}
	\dScalP{g}_M = \int_{\Omega} \gamma_g \ud \tr_M
\end{equation}
Now, using also \eqref{eq:III:5.6}--\eqref{eq:III:5.8}, we can easily check that $f_n := \chi_{\omega_n} \cdot f$ for any $n \in \NN$ belongs to $\Dom(T_M)$, and by the Lebesgue dominated convergence theorem
\[
	\| [f] - [f_n] \|^2_M \to 0
\]
so in particular $\Dom(T_M)$ is dense. And the closedness we get later, as a consequence of 5) (the adjoint operator is always closed).

\fbox{Ad. 5} Take any $f,g \in \calL^2(M)$ such, that $[f] \in \Dom(T_{\overline{F}})$, $g \in \Dom(T_F)$ (obviously -- by \eqref{eq:III:5.9} and \eqref{eq:III:5.7} -- these two domains are equal). For any $h_1, h_2 \in \calL^2(M)$ we have
\begin{equation}
	\label{eq:III:5.10}
	\big\langle [h_1], [h_2] \big\rangle_M = 
	\int_\Omega \gamma_{h_1, h_2} \ud \tr_M,
\end{equation}
so, by \eqref{eq:III:5.7}
\begin{equation}
	\big\langle T_F [g], [f] \big\rangle_M =
	\int_\Omega \gamma_{Fg,f} \ud \tr_M =
	\int_\Omega \gamma_{g, \overline{F}f} \ud \tr_M =
	\big\langle [g], T_{\overline{F} [f]} \big\rangle_M,
\end{equation}
and hence $T_{\overline{F}} \subset (T_F)^*$.

Assume now, that $[f] \in \Dom \big( (T_F)^* \big)$, so, let's fix $h \in \calL^2(M)$ satisfying $[h] = (T_F)^* [f]$, i.e.:
\[
	\forall_{[g] \in \Dom(T_F)} \
	\big\langle T_F[g], [f] \big\rangle_M =
	\big\langle [g], [h] \big\rangle_M.
\]
In particular, we can use this for $g_n$ of the form $g_n = \chi_{\omega_n} \cdot g$ for the previously defined $\omega_n$'s, $n \in \NN$, and any $g \in \calL^2(M)$:
\begin{equation}
	\label{eq:III:5.11}
	\big\langle T_F[g_n], [f] \big\rangle_M =
	\big\langle [g_n], [h] \big\rangle_M, \quad
	n \in \NN, g \in \calL^2(M).
\end{equation}
But using again \eqref{eq:III:5.7} and \eqref{eq:III:5.10} we get
\[
	\big\langle T_F[g_n], [f] \big\rangle_M =
	\big\langle [g], [\chi_{\omega_n} \cdot (\overline{F} f)] \big\rangle_M
\]
and
\[
	\big\langle g_n, [h] \big\rangle_M =
	\big\langle [g], [\chi_{\omega_n} \cdot h] \big\rangle_M,
\]
so, by \eqref{eq:III:5.1} and the arbitrariness of $[g] \in L^2(M)$ we obtain:
\begin{equation}
	\label{eq:III:5.12}
	\forall_{n \in \NN} \
	\big[ \chi_{\omega_n} \cdot (h - \overline{F}f) \big] = 0.
\end{equation}

We shall prove that $(h - \overline{F}f) \in \calL^2_0(M)$. By Fact~\ref{fact:III:12}, we have to check, that $\tr_M (\tilde{\Omega}) = 0$, where
\[
	\tilde{\Omega} := 
	\big\{ t \in \Omega: (h - \overline{F}f) \notin \Ker D_M(t) \big\}.
\]
But \eqref{eq:III:5.12} means that $\tr_M (\tilde{\Omega} \cap \omega_n) = 0$, for any $n \in \NN$, so
\[
	0 = 
	\tr_M \bigg( \bigcup_{n \in \NN} \big( \tilde{\Omega} \cap \omega_n \big) \bigg) =
	\tr_M(\tilde{\Omega}).
\]
Finally $h - \overline{F}f$, $h \in \calL^2(M)$, hence also $\overline{F}f \in \calL^2(M)$ and $[\overline{F}f] = [h] = (T_F)^*[f]$, i.e. also $(T_F)^* \subset T_{\overline{F}}$.

So, the proofs of 1) and 5) are ready. Part 2) is trivial, however it is good to remember "the jump" in the domain at $\lambda=0$.

We shall use below the following notation for a function $F$, $\varepsilon > 0$, $\lambda \in \CC$ and\footnote{\label{fn:III:25} $B(\lambda, \varepsilon) := \{ z \in \CC : |\lambda - z| < \epsilon \}$.} $c \in \CC^d$:
\begin{equation}
	\label{eq:III:5.13}
	\beta_{F, \lambda, \varepsilon, c} := \chi_{F^{-1} \big( B(\lambda, \varepsilon) \big)} \cdot c.
\end{equation}
By Fact~\ref{fact:III:21} we obtain:
\begin{fact} \label{fact:III:29}
If $F : \Omega \to \CC$ is $\frakM$-measurable then:
\begin{enumerate}[(1)]
\item for any $\varepsilon > 0$, $\lambda \in \CC$, $c \in \CC^d$ $\beta_{F, \lambda, \varepsilon,c} \in \calL^2(M)$ and
\begin{equation}
	\label{eq:III:5.14}
	\| [\beta_{F, \lambda, \varepsilon,c}] \|^2_M =
	\Big\langle M \Big( F^{-1} \big( B(\lambda, \varepsilon) \big) \Big) c, c \Big\rangle_{\CC^d}
\end{equation}

\item if $\varepsilon$, $\lambda$, $c$ are such, that $c \notin \Ker \Big( F^{-1} \big( B(\lambda, \varepsilon) \big) \Big)$ then $\| [\beta_{F, \lambda, \varepsilon,c}] \|_M > 0$ and $r \in \RR$ can be chosen with $\| [\beta_{F, \lambda, \varepsilon,\tilde{c}}] \|_M = 1$ for $\tilde{c} = r c$;

\item $\omega \in \frakM$, and for any $t \in \omega$ $s \leq \big| F(t) \big| \leq S$ for some $s,S \in [0; +\infty)$, then $[\chi_\omega \cdot c] \in \Dom(T_F)$ for each $c \in \CC^d$, and 
\begin{equation}
	\label{eq:III:5.15}
	\begin{split}
		&s \| [\chi_\omega \cdot c] \|_M \leq 
		\| T_F [\chi_\omega \cdot c] \|_M \leq
		S \| [\chi_\omega \cdot c] \|_M, \\
		&\| [\chi_\omega \cdot c] \|_M =
		\sqrt{\langle M(\omega) c,c \rangle_{\CC^d}}.
	\end{split}
\end{equation}

\item for any $\varepsilon > 0$, $\lambda \in \CC$ and $c \in \CC^d$ $\beta_{F, \lambda, \varepsilon,c} \in \Dom(T_F)$.
\end{enumerate}
\end{fact}
Thus fact will be a convenient tool in some next proofs.

\fbox{Ad. 4)} Consider first $[f] \in \Dom(T_F)$, $\| [f] \|_M \leq 1$. We have by \eqref{eq:III:5.2'}
\begin{align*}
	\| T_F [f] \|_M^2 = 
	\| [F f] \|_M^2 &= 
	\int_{\Omega} |F(t)|^2 \big\langle D_M(t) f(t), f(t) \big\rangle_{\CC^d} \ud \tr_M(t) \\
	&=
	\int_{VNE_M(F)} \ldots + \int_{VE_M(F)} \ldots \\
	&= 
	\int_{\VE_M(F)} \ldots \\
	&\leq
	\big( \sup |\VE_M(F)| \big)^2 \cdot \| [f] \|^2_M,
\end{align*}
which proves "$\Leftarrow$" and "$\leq$" of \eqref{eq:III:5.3} from 4). We shall prove now "$\geq$" of \eqref{eq:III:5.3} also in the case when $\VE_M(F)$ is unbounded (with the $\| T_F \|$ being $+\infty$ for the unbounded operator $T_F$ case). So, let $s \in [0; +\infty)$ be such, that $s < \sup |\VE_M(F)|$. Then, using the definition of the supremum let us choose $\lambda_0 \in \VE_M(F)$ such that $s < |\lambda_0|$ and let $\varepsilon := \frac{|\lambda_0|-s}{2}$. Then for $\omega := F^{-1} \big( B(\lambda_0, \varepsilon) \big)$
\[
	\forall_{t \in \omega} \
	s \leq |F(t)| \leq |\lambda_0| + \varepsilon.
\]
So, by Fact~\ref{fact:III:29}, by \eqref{eq:III:5.2'} and by the definition of $\VE+M(F)$ we see that $\tr_M(\omega) > 0$ and so we can choose $c \in \CC^d$ such that for $f := \beta_{F, \lambda, \varepsilon,c}$ we have $\| [f] \|_M = 1$, $[f] \in \Dom(T_F)$ and $\| T_F [f] \| \geq s$.
Hence $\| T_F \| \geq s$, and by the arbitrariness of the choice of $s \in [0, \sup |\VE_M(F)|)$ we finally get \eqref{eq:III:5.3} (in the expanded -- "unbounded" case also), and 4) is proved.

\fbox{Ad. 3)} In the trivial case $M(\Omega) = 0$ also 3) is trivial. Suppose that $M(\Omega) \neq 0$. Then $L^2(M) \neq \{ 0 \}$, so by 4) we get a) (note, that $\VE(F) \neq \emptyset$ then, because $\VNE_M(F) \neq \CC$). Also "$\Rightarrow$" of b) obviously holds, so let us check "$\Rightarrow$". If $T_F = T_G$, then we have:
\[
	T_{F-G} \supset T_F - T_G = O_{\Dom(F)}
\]
where $0_Y$ denotes the operator with the domain $Y$ (for linear subspaces $Y$) which is constantly $0$ on $Y$. But $\Dom(T_F)$ is dense, and $T_{F-G}$ -- closed by 1), so $T_{F-G} \supset \overline{O_{\Dom(F)}} = 0$, i.e. $T_{F-G} = 0$. Hence we get our assertion by a)\footnote{\label{fn:III:26} Note, that for each measurable function $H : \Omega \to \CC$:
\[
	\VE_M(H) \subset \{ 0 \} \quad \Rightarrow \quad 
	\CC \setminus \{ 0 \} \subset \VNE_M(H) \quad \Rightarrow \quad
	\tr_M \big( \{ t \in \Omega : H(t) \neq 0 \big) = 0.
\]}.

\fbox{Ad. 6)} Let $g \in \calL^2(M)$. Then, by the definition of the multiplication by function operator (+ the footnote \ref{fn:III:20}) and the definition of the product of unbounded operators, we obtain for $g \in \calL^2(M)$:
\begin{enumerate}[(1)]
\item $[g] \in \Dom(T_F T_G)$ iff both of the equations holds 
\begin{align}
	\label{eq:III:5.16}
	&G g \in \calL^2(M) \\
	\label{eq:III:5.17}
	&F G g \in \calL^2(M)
\end{align}

\item $[g] \in \Dom(T_{FG}) \Leftrightarrow$ \eqref{eq:III:5.17} holds

\item $[g] \in \Dom(T_G) \Leftrightarrow$ \eqref{eq:III:5.16} holds
\end{enumerate}
Therefore $T_F T_G \subset T_{FG}$, and $\Dom(T_F T_G) = \Dom(T_{FG}) \cap \Dom(T_G)$. Hence 
\[
	T_F T_G = T_{FG} \quad \Leftrightarrow \quad
	\Dom(T_{FG}) \cap \Dom(T_G) = \Dom(T_{FG}) \quad \Leftrightarrow \quad
	\Dom(T_{FG}) \subset \Dom(T_G).
\]

\fbox{Ad. 7)} We shall use the following
\begin{lemma} \label{lem:III:30}
Let $G : \Omega \to \CC$ be $\frakM$-measurable, $\omega_0 := G^{-1}(\{0\})$.
\begin{enumerate}[(A)]
\item For any $f \in \calL^2(M)$
\[
	[f] \in \Ker T_G \quad \Leftrightarrow \quad
	G f \in \calL_0^2(M).
\]

\item If $c \in \CC^d$, then $[\chi_{\omega_0} c] \in \Ker T_G$.

\item $\Ker T_G = \{0\} \Leftrightarrow M(\omega_0) = 0$.

\item Suppose, that $M(\omega_0) = 0$, and let $H : \Omega \to \CC$ be a measurable function, such that 
\[
	H(t) = \frac{1}{G(t)} \quad \text{for } \tr_M \text{ a.e. } t \in \Omega \setminus \omega_0.
\]
Then
\begin{equation}
	\label{eq:III:5.18}	
	T_H T_G = \restr{I}{\Dom(T_G)}, \quad
	T_G T_H = \restr{I}{\Dom(T_H)},
\end{equation}
and $T_G : \Dom(T_G) \to \Dom(T_H)$ is a bijection onto $\Dom(T_H)$ and $T_H = (T_G)^{-1}$; moreover $(T_G)^{-1} \in \calB \big( L^2(M) \big)$ iff $0 \in \VNE_M(G)$.

\item $\rho(T_G) = \VNE_M(G)$.
\end{enumerate}
\end{lemma}
\begin{proof}[Proof of Lemma~\ref{lem:III:30}]
For $c \in \CC^d$ we get $[\chi_{\omega_0} c] \in \Dom(T_G)$ and $[\chi_{\omega_0} c] \in \Ker(T_G)$, by Fact~\ref{fact:III:29}(3) (with $s=S=0$). Hence (B) holds.

To prove (C) assume first that $\Ker T_G = \{0\}$. Thus, by (B) and \eqref{eq:III:5.15}
\[
	\forall_{c \in \CC^d} \
	0 = 
	\| [\chi_{\omega_0} c] \|_M^2 = 
	\langle M(\omega_0) c, c \rangle,
\]
therefore $M(\omega_0) = 0$.

Now, assume $M(\omega_0) = 0$ and let $f \in \calL^2(M)$, $[f] \in \Ker T_G$. Hence $G f \in \calL^2_0(M)$. Recall now, that -- by Remark~\ref{rem:III:14}(b) -- for any $g \in \calL^2(M)$ the condition $g \in \calL_0^2(M)$ is equivalent to $M \big( \supp_{\calL^2(M)}(g) \big) = 0$, where
\[
	\supp_{\calL^2(M)}(g) =
	\big\{ t \in \Omega : D_M(t) g(t) \neq 0 \big\}
\]
(it is defined "up to" a $M$-measure zero set, due to the freedom of the choice of $D_M$, but let us fix some $D_M$ here). So, we have:
\begin{align*}
	0 
	&\leq
	M \big( \supp_{\calL^2(M)}(f) \big) \\
	&=
	M \big( \omega_0 \cap \supp_{\calL^2(M)}(f) \big) +
	\big( (\Omega \setminus \omega_0) \cap \supp_{\calL^2(M)}(f) \big) \\
	&=
	M \big( \big\{ t \in \Omega : G(t) \neq 0 \text{ and } D_M(t) f(t) \neq 0 \big\} \big) \\
	&=
	M \big( \big\{ t \in \Omega : G(t) \neq 0 \text{ and } D_M(t) \big( G(t) f(t) \big) \neq 0 \big\} \big) \\
	&\leq
	M \big( \supp_{\calL^2(M)}(G f) \big) = 0
\end{align*}
hence, also $M \big( \supp_{\calL^2(M)}(f) \big) = 0$ and $[f] = 0$, i.e. $\Ker T_G = \{ 0 \}$.

Now, we consider $H$ from (D). Since $M(\omega_0) = 0$, we have $H \cdot G = G \cdot H = \mathds{1}$ $\tr_M$ -- almost everywhere on $\Omega$ (here $\mathds{1}$ is the constant $1$ function), hence, by parts 3)b) and 6) of Theorem~\ref{thm:III:28} (both parts -- already proved!) we get $T_H T_G, T_G T_H \subset T_{\mathds{1}} = I$ and $\Dom(T_{\mathds{1}}) = L^2(M)$, so \eqref{eq:III:5.18} holds, $T_G$ is "1-1" and $T_H = T_G^{-1}$. Moreover, by parts 1), 3), 4) of the theorem:
\[
	T_G^{-1} \in \calB \big( L^2(M) \big) \quad \Leftrightarrow \quad
	T_{\tilde{H}} \in \calB \big( L^2(M) \big) \quad \Leftrightarrow \quad
	\exists_{r \in (1; \infty)} \ |\VE_M(\tilde{H})| \leq r,
\]
where
\[
	\tilde{H}(t) :=
	\begin{cases}
		\frac{1}{G(t)} & \text{for } t \in \Omega \setminus \omega_0 \\
		1 & \text{for } t \in \omega_0
	\end{cases},
\]
because $T_H = T_{\tilde{H}}$ (thanks to $M(\omega_0) = 0$). But
\[
	|\VE_M(\tilde{H})| \leq r \quad \Leftrightarrow \quad
	\VNE_M(\tilde{H}) \supset \Pi_r := \{ z \in \CC : |z| > r \} \quad \Leftrightarrow \quad
	M \big( \tilde{H}^{-1}(\Pi_r) \big) = 0,
\]
because $M \Big( \tilde{H}^{-1} \big( \VNE_M(\tilde{H}) \big) \Big) = 0$ by \eqref{eq:III:5.2'}.

But for $r > 1$ $\tilde{H}^{-1} (\Pi_r) = G^{-1} \big( K(0, \tfrac{1}{r}) \big)$. So, finally
\[
	T^{-1}_G \in \calB \big( L^2(M) \big) \quad \Leftrightarrow \quad
	0 \in \VNE_M(G).
\]
To obtain (E) it suffices to use (????) and the fact that if $M(\omega_0) \neq 0$, then $0 \notin \rho(T_G)$ since $\Ker T_G \neq \{ 0 \}$ by (C) and also $0 \notin \VNE_M(G)$ because $\VNE_M(G)$ is an open subset of $\CC$ and if $0 \in \VNE_M(G)$, then $0 = M \Big( G^{-1} \big( \VNE_M(G) \big) \Big)$ (see \eqref{eq:III:5.2'}) and 
\[
	M \Big( G^{-1} \big( \VNE_M(G) \big) \Big) \geq 
	M \big( G^{-1} (\{0\}) \big) =
	M(\omega_0)  \gneq 0.
\]
\end{proof}

Now to obtain 7), it suffices to use (E) of the lemma to $G = F = \lambda_0 \mathds{1}$ for each $\lambda_0 \in \CC$ and the previous observation, that $T_G = T_F - \lambda_0 I$ for such $G$. -- This gives $\rho(T_F) = \VNE_M(F)$, so $\sigma(T_F) = \VE_M(F)$. Similarly, to get 9) using (D) and (E) of the lemma for this $G$, and we get 8) using (C). It remains to prove 10).

Assume $F(\Omega) \subset \RR$. To get 10) we should check that:
\begin{enumerate}[(i)]
\item $E := E_{F, M}$ is a spectral measure on $\Bor(\RR)$

\item $\int_{\RR} \xx \ud E = T_F$.
\end{enumerate}

Property (i) means that
\begin{itemize}
\item $E$ is additive

\item $E(\RR) = I$

\item $E(\omega)^* = E(\omega)$ for any $\omega \in \Bor(\RR)$

\item $E(\omega \cap \omega') = E(\omega) E(\omega')$ for any $\omega, \omega' \in \Bor(\RR)$

\item $\forall_{[f] \in L^2(M)} \ E_{[f]}$ is a measure on\footnote{\label{fn:III:27} Recall that for a spectral measure $E$ for a Hilbert space $\calH$ on $\sigma$-algebra $\frakM$ and for $x,y \in \calH$ $E_{x,y} : \frakM \to \CC$ is given for $\omega \in \frakM$ by $E_{x,y}(\omega) := \langle E(\omega) x, y \rangle$, and $E_x := E_{x,x}$.} $\Bor(\RR)$.
\end{itemize}

The first four properties are an easy exercise. To check the fifth property, observe that for $[f],[g] \in L^2(M)$
\begin{equation}
	\label{eq:III:5.19}
	\forall_{\omega \in \Bor(\RR)} \
	E_{[f], [g]}(\omega) =
	\langle E(\omega) [f], [g] \rangle_M =
	\dScalP{\chi_{F^{-1}(\omega)} \cdot f, g}_M =
	\int_{F^{-1}(\omega)} \gamma_{f,g} \ud \tr_M
\end{equation}
and recall, that by Lemma~\ref{lem:III:9} we have
\begin{equation}
	\label{eq:III:5.20}
	\gamma_{f,g} \in \calL^1(\tr_M).
\end{equation}
So, let us denote by $\mu_{f,g}$ the complex measure:
\begin{equation}
	\label{eq:III:5.21}
	\mu_{f,g} := \gamma_{f,g} \ud \tr_M.
\end{equation}
By standard properties of complex measures we know that the variation (measure) of $\mu_{f,g}$ is given by
\begin{equation}
	\label{eq:III:5.22}
	|\mu_{f,g}| := |\gamma_{f,g}| \ud \tr_M.
\end{equation}
In particular, for $\mu_f := \mu_{f,f}$ we have
\begin{equation}
	\label{eq:III:5.22'}
	\mu_f := \gamma_f \ud \tr_M
\end{equation}
and $\mu_f$ is a finite measure, $\mu_f \geq 0$. Recall that the complex measure $\mu_{f,g}$ can be "moved" from $\frakM$ onto $\Bor(\RR)$ by the measurable $F : \Omega \to \RR$ and the resulting complex measure $(\mu_{f,g})_F$ ($(\mu_f)_F$ for $g=f$, and then $(\mu_f)_F$ is a finite measure) is given by
\begin{equation}
	\label{eq:III:5.23}
	(\mu_{f,g})_F(\omega) = \mu_{f,g} \big( F^{-1} (\omega) \big), \quad
	\omega \in \Bor(\RR).
\end{equation}
The "change of variable" rule for complex measures says, that for any Borel function $\varphi : \RR \to \CC$
\begin{equation}
	\label{eq:III:5.24}
	\varphi \in \calL^1 \big( (\mu_{f,g})_F \big) \quad \Leftrightarrow \quad
	\varphi \circ F \in \calL^1(\mu_{f,g})
\end{equation}
and if $\varphi \in \calL^1 \big( (\mu_{f,g})_F \big)$, then
\begin{equation}
	\label{eq:III:5.24'}
	\int_\RR \varphi \ud (\mu_{f,g})_F =
	\int_{\Omega} \varphi \circ F \ud \mu_{f,g}
\end{equation}
Hence, finally
\begin{equation}
	\label{eq:III:5.25}
	E_{[f],[g]} = (\mu_{f,g})_F.
\end{equation}
Now, we can come back to the last part of (i) and we see, that $E_{[f]} = (\mu_f)_F$, so it is a finite measure, in particular.

To obtain (ii) let us notice first, that both sides of (ii): $T_F$ and $\int_\RR \xx \ud E$ are self-adjoint operators. So, it suffices to prove the inclusion
\[
	T_F \subset \int_\RR \xx \ud E
\]
to get (ii). Let $[f] \in \Dom(T_F)$, so $T_F \in \calL^2(M)$ and
\begin{equation}
	\label{eq:III:5.26}
	|F|^2 \gamma_f = \gamma_{F f} \in \calL^1(\tr_M)
\end{equation}
(see \eqref{eq:III:5.7} and use \eqref{eq:III:5.20} for $F f$), which can be written as
\begin{equation}
	\label{eq:III:5.26'}
	|F|^2  \in \calL^1(\mu_f).
\end{equation}
Moreover, we have by \eqref{eq:III:5.10} and \eqref{eq:III:5.7}
\begin{equation}
	\label{eq:III:5.27.1}
	\begin{split}
		\forall_{[g] \in L^2(M)} \
		\big\langle T_F[f], [g] \big\rangle_M 
		&=
		\big\langle [F f], [g] \big\rangle_M 
		=
		\int_{\Omega} \gamma_{F f, g} \ud \tr_M \\
		&=
		\int_{\Omega} F \cdot \gamma_{f,g} \ud \tr_M 
		=
		\int_{\Omega} F \ud \mu_{f,g}
	\end{split}
\end{equation}
and
\begin{equation}
	\label{eq:III:5.27.2}
	F \in \calL^1(\mu_{f,g})
\end{equation}
by Lemma~\ref{lem:III:9} used for $F f, g$.

We have to check, that $[f] \in \Dom \Big( \int_{\RR} \xx \ud E \Big)$ and that
\begin{equation}
	\label{eq:III:5.28}
	\forall_{[g] \in L^2(M)} \
	\bigg\langle \bigg( \int_{\RR} \xx \ud E \bigg) [f], [g] \bigg\rangle_{M} =
	\big\langle T_F [f], [g] \big\rangle_M.
\end{equation}
But by \eqref{eq:III:5.25}, by FCTh, and by the definition of the "weak" integration with respect to the spectral measure $E$
\begin{equation}
	\label{eq:III:5.29}
	[h] \in \Dom \Big( \int_{\RR} \xx \ud E \Big) \quad \Leftrightarrow \quad
	\xx \in \calL^2(E_{[h]}) = \calL^2 \big( (\mu_h)_F \big) \quad \Leftrightarrow \quad
	|\xx|^2 \in \calL^1 \big( (\mu_h)_F \big).
\end{equation}
and for $[h] \in \Dom \Big( \int_{\RR} \xx \ud E \Big)$ and $[g] \in L^2(M)$ we have
\begin{equation}
	\label{eq:III:5.30}	
	\xx \in \calL^1 \big( (\mu_{h,g}) \big) \quad \text{and} \quad
	\bigg\langle 
	\bigg( \int_{\RR} \xx \ud E \bigg)[h], [g] \bigg\rangle_M =
	\int_{\RR} \xx \ud (\mu_{h,g})_F.
\end{equation}
But using \eqref{eq:III:5.24} to $\varphi = |\xx|^2$ and $g = f$, by \eqref{eq:III:5.26'} we get $|\xx^2| \circ F \in \calL^1(\mu_f)$, hence
\[
	|\xx|^2 \in \calL^1 \big( (\mu_f)_F \big).
\]
Now by \eqref{eq:III:5.29} $[f] \in \Dom \Big( \int_{\RR} \xx \ud E \Big)$, and we can take $h=f$ in \eqref{eq:III:5.30}, i.e. LHS of \eqref{eq:III:5.28} equals $\int_\RR F \ud (\mu_{f,g})_F$. But, by \eqref{eq:III:5.27.1}, \eqref{eq:III:5.27.2}, RHS of \eqref{eq:III:5.28} equals $\int_\Omega F \ud \mu_{f,g}$, and $F \in \calL^1(\mu_{f,g})$, so taking $\varphi = \xx$ in \eqref{eq:III:5.24} and \eqref{eq:III:5.24'} we see that $\int_{\Omega} F \ud \mu_{f,g} = \int_\RR \xx \ud (\mu_{f,g})_F$, i.e. \eqref{eq:III:5.28} holds.
\end{proof}

\newpage

\section{"$\xx$MUE" Theorem -- The finitely cyclic case} \label{sec:IV}

\vspace{3ex}

\subsection{Finitely cyclic s.a. operators} \label{sec:IV:1}
Let $A$ be a linear operator in a normed space $X$. Recall that \emph{$A$ is cyclic} iff there exists a \emph{a cyclic vector $\varphi \in X$ for $A$}, which means that\footnote{Recall that $\Dom(A)$ denotes the domain of a linear operator $A$ and $\Dom(A^\infty) := \bigcup_{n \in \NN} \Dom(A^n)$.}
\begin{equation}
	\label{eq:IV:1.1'}
	\varphi \in \Dom(A^\infty) \quad \text{and} \quad
	\Orb_A(\varphi) := \big\{ A^n \varphi : n \in \NN_0 \big\} \text{ is linearly dense in } X.
\end{equation}

We define here some notions, which generalize cyclicity. First observe that by the linearity of $A$ we can consider the space $\lin \{ \varphi \}$ instead of an individual cyclic vector $\varphi$.

\begin{definition} \label{def:IV:1}
A subset $Y \subset X$ is \emph{a cyclic set for $A$} (or -- shortly -- is \emph{cyclic for $A$}) iff\footnote{Note here, that assuming more -- the density instead of linear density -- we get the \emph{hyper}cyclicity notion.}
\begin{equation}
	\label{eq:IV:1.2}
	Y \subset \Dom(A^\infty) \quad \text{and} \quad
	\Orb_A(Y) := \big\{ A^n y : n \in \NN_0, y \in Y \big\} \text{ is linearly dense in } X.
\end{equation}
\end{definition}

When $Y \underset{\lin}{\subset} X$, then we say also a "\emph{cyclic space}" instead of "cyclic set".

Similarly, when $\vec{\varphi} = (\varphi_1, \ldots, \varphi_d)$ is not "only" a finite set $\{ \varphi_1, \ldots, \varphi_d \}$, but an (ordered) system of vectors (that is $\vec{\varphi} \in X^d$, and not $\vec{\varphi} \subset X$), then we denote
\[
	\Orb_A(\varphi) := 
	\Orb_A(\{ \varphi_1, \ldots, \varphi_d \}) =
	\{ A^n \varphi_j : n \in \NN_0, j=1,\ldots,d \}
\]
and we say $\varphi$ \emph{is a cyclic system} (or \emph{is cyclic}, for short) \emph{for $A$} iff $\{ \varphi_1, \ldots, \varphi_d \}$ is a cyclic set for $A$. 

We study here "finite cyclicity" defined as follows
\begin{definition} \label{def:IV:2}
\emph{$A$ is finitely cyclic} iff there exists a finite dimensional cyclic space for $A$. Each such subspace of $X$ is called a \emph{space of cyclicity for} $A$ (note, that this is more than to be a cyclic space, because of the extra finite dimension assumption).
\end{definition}

Observe, that for any $\emptyset \neq Y \subset \Dom(A^\infty)$
\begin{equation}
	\label{eq:IV:1.3}
	\lin \Orb_A(Y) = \lin \Big( \Orb_A \big( \lin(y) \big) \Big),
\end{equation}
because, by the linearity of $A$, $\lin(Y) \subset \Dom(A^\infty)$ and
\begin{equation}
	\label{eq:IV:1.3'}
	\Orb_A \big( \lin(Y) \big) \subset \lin \big( \Orb_A(Y) \big) \quad \Leftrightarrow \quad Y \subset \Dom(A^\infty).
\end{equation}
(a very easy exercise).

Therefore we get:
\begin{fact} \label{fact:IV:3}
Suppose that $\varphi_1, \ldots, \varphi_d \in X$. TFCAE:
\begin{enumerate}
\item $A$ is finitely cyclic with $\lin \big( \{\varphi_1, \ldots, \varphi_d \} \big)$ -- a cyclic space for $A$

\item $\{ \varphi_1, \ldots, \varphi_d \}$ is a cyclic set for $A$

\item $\vec{\varphi} = (\varphi_1, \ldots, \varphi_d)$ is a cyclic system for $A$.
\end{enumerate}
\end{fact}

Obviously, the cyclicity property of $A$ is the finite cyclicity with $d=1$ and a cyclic set $\{\varphi_1, \ldots, \varphi_d \}$ or a cyclic system $\vec{\varphi}$ is a "good" analog of a cyclic vector. But note, that $d$ can be larger, in general, than the dimension of the space $\lin \big( \{ \varphi_1, \ldots, \varphi_d \} \big)$.

\vspace{5ex}

\subsection{The spectral matrix measure for a s.a. operator and a system of vectors} \label{sec:IV:2}
In "our" "$\xx$MUE" Theorem for the cyclic case (Section~\ref{sec:II}) the crucial role played the spectral measure $\mu = E_{A,\varphi}$ for $A$ and a fixed cyclic vector $\varphi$ for $A$. As one can guess now, for finitely cyclic version of $\xx$MUETh. we shall consider the multiplication by function $\xx$ operator $T_\xx$ in the space $L^2(M)$ for some matrix measure $M$ on $\Bor(\RR)$ (-- the class of Borel subsets of $\RR$). Thus matrix measure will be called \emph{the spectral matrix measure for $A$ and} (the cyclic) \emph{system $\vec{\varphi}$}, however we define it, and we denote it by $E_{A,\vec{\varphi}}$, also we do not assume the cyclicity.

So, assume that: $A$ is a self-adjoint operator in a Hilbert space $\calH$, $\vec{\varphi} \in \calH^d$, $\vec{\varphi} = (\varphi_1, \ldots, \varphi_d)$, and that $E_A: \Bor(\RR) \to \calB(\calH)$ is the projection-valued spectral measure for $A$ (other name: the resolution of the identity $I$ for $A$) -- see ...

Now consider
\begin{equation}
	\label{eq:IV:2.1}
	\begin{split}
		&E_{A,\vec{\varphi}} : \Bor(\RR) \to \Mat{d}{\CC}, \\
		&E_{A,\vec{\varphi}}(\omega) := 
		\Big( \big\langle E_A(\omega) \varphi_j, \varphi_i \big\rangle \Big)_{i,j=1,\ldots,d} \in \Mat{d}{\CC}, \quad 
		\omega \in \Bor(\RR)
	\end{split}
\end{equation}
(here $\langle \cdot, \cdot \rangle$ is for $\calH$).

\begin{lemma} \label{lem:IV:4}
If $B$ is s.a. in a Hilbert space $\calH$, $\varphi_1, \ldots, \varphi_d \in \Dom(B)$ and $C \in \Mat{d}{\CC}$ is given by
\[
	C := \big( \langle B \varphi_j, \varphi_i \rangle \big)_{i,j=1,\ldots,d},
\]
then $C$ is s.a. If moreover, $B \geq 0$, then also $C \geq 0$.
\end{lemma}
\begin{proof}
Easy exercise.
\end{proof}

Using this lemma, the fact that for any Borel set $\omega$ $E_A(\omega)$ is an orthogonal projection (so, in particular, $E_A(\omega) \geq 0$) and the "weak complex measure" properties of $E_A$, we see that $E_{A, \vec{\varphi}}$ defined above is a matrix measure on $\Bor(A)$.

\begin{definition} \label{def:IV:5}
$E_{A,\vec{\varphi}}$ (given by \eqref{eq:IV:2.1}) is called \emph{the spectral matrix measure for $A$ and $\vec{\varphi}$}.
\end{definition}

We shall prove now a convenient result, being our first connection joining spectral calculus for $A$ and \emph{some} vectors -- classes of functions -- from $L^2(E_{A, \vec{\varphi}})$.

Those "some vectors" would be just vectors from the $L^2_\Sigma(E_{A, \vec{\varphi}})$ space (see Section~\ref{sec:III:4}). And it is here, and also later, when you can see that the role (mentioned before) of $\calL_\Sigma^2(M)$ in $\calL^2(M)$ is analog now to the role of the Schwarz class function $\calS(\RR)$ in $\calL^2(\RR)$ in some popular (e.g. Rudin \cite{Rudin1991}) proofs of the unitarity of the Fourier transform.

\begin{theorem} \label{thm:IV:6}
Suppose that $A$ is s.a. in $\calH$ and $\vec{\varphi} \in \calH^d$ ($d \in \NN$). Let $M$ be the spectral matrix measure for $A$ and $\vec{\varphi}$ (i.e. $M = E_{A, \vec{\varphi}}$). Then:
\begin{enumerate}[(i)]
\item for any $f \in \calL^2_\Sigma(M)$ and $j=1,\ldots,d$
\begin{equation} 
	\label{eq:IV:2.2}
	\varphi_j \in \big( f_j(A) \big),
\end{equation}

\item the linear transformation $\tilde{W} : \calL^2_\Sigma(M) \to \calH$ given by
\begin{equation} 
	\label{eq:IV:2.3}
	\tilde{W} f := \sum_{j=1}^d f_j(A) \varphi_j, \quad
	f \in \calL^2_\Sigma(M)
\end{equation}
satisfies
\begin{equation} 
	\label{eq:IV:2.4}
	\forall_{f \in \calL^2_\Sigma(M)} \
	\|\tilde{W} \| = \tnorm{f}_M,
\end{equation}

\item the transformation $W : L^2_\Sigma(M) \to \calH$ given by
\begin{equation} 
	\label{eq:IV:2.5}
	W [f] = \tilde{W} f, \quad
	f \in \calL^2_\Sigma(M)
\end{equation}
is properly defined (i.e. $\tilde{W} f = 0$ for $f \in \calL^2_\Sigma(M) \cap \calL^2_0(M)$) and is a linear isometry into $\calH$.
\end{enumerate}
\end{theorem}

We shall need some abstract results concerning spectral complex measures and functional calculus for each self-adjoint operator to prove this theorem. The first is a generalization of a well-known result\footnote{See [Lemma 3.23(3), Rudin AF] concerning bounded functions $f$.}

\begin{lemma} \label{lem:IV:7}
Let $A$ be a s.a. operator in $\calH$, $E$ -- its projection valued spectral measure $E=E_A$, $f : \RR \to \CC$ -- a Borel function (= measurable with respect to $\Bor(\RR)$) and $x \in \calH$, $z \in \Dom \big( f(A) \big)$, then $\overline{f} \in \calL^1(E_{x,z})$\footnote{Recall that for a complex measure $\mu$ $\calL^2(\mu) = \calL^1(|\mu|)$ ($|\mu|$ -- the variation of $\mu$); of course $\overline{f} \in \calL^1 \ldots \Leftrightarrow f \in \calL^1 \ldots$.} and
\begin{equation} 
	\label{eq:IV:2.6}
	E_{x, f(A) z} = \overline{f} \ud E_{x,z}.
\end{equation}
\end{lemma}
\begin{proof}
We leave this result temporarily as an exercise for the reader (hint: use the version with $f$ bounded --- see \cite[Lemma 13.23]{Rudin1991}), however  the proof is planned in one of the next versions (in Appendix).
\end{proof}

The next result generalizes the above one. It will be an important tool for us, because it "hides" in some sense the result \eqref{eq:IV:2.4}, being in fact the isometricity part of the assertion (iii) of our theorem.

\begin{lemma} \label{lem:IV:8}
Let $A$ be a s.a. operator in $\calH$, $E$ -- its projection-valued spectral measure, $f,g : \RR \to \CC$ -- Borel functions and $x \in \Dom \big( f(A) \big)$, $y \in \Dom \big( g(A) \big)$. Then 
\begin{equation} 
	\label{eq:IV:2.7}
	f \overline{g} \in \calL^1(E_{x,y}) \quad \text{and} \quad
	E_{f(A) x, f(A) y} = f \overline{g} \ud E_{x,y}.
\end{equation}
\end{lemma}
\begin{proof}
By previous lemma (for $f,y,x$) $\overline{f} \in \calL^1(E_{y,x})$ and $E_{y, f(A)x} = \overline{f} \ud E_{y,x}$, hence
\begin{equation} 
	\label{eq:IV:2.8}
	E_{f(A)x, y} = 
	\overline{\overline{f} \ud E_{y,x}} =
	f \ud E_{x,y}.
\end{equation}
Now, by the same lemma (for $g, f(A), y$) $\overline{g} \in \calL^1 (E_{f(A)x,y})$ and
\[
	E_{f(A)x, g(A) y} = \overline{g} \ud E_{f(A) x, y}
\]
so, by \eqref{eq:IV:2.8} we get \eqref{eq:IV:2.7}, using also
\begin{align*}
	\overline{g} \in \calL^1(E_{f(A) x, y}) \quad
	&\Leftrightarrow \quad
	|g| \in \calL^1(|E_{f(A) x, y}|) =
	\calL^1(|f \ud E_{x, y}|) =
	\calL^1(|f| \ud |E_{x, y}|) \\
	&\Leftrightarrow \quad
	|f| \cdot |\overline{g}| \in \calL^1(|E_{x, y}|) \\
	&\Leftrightarrow \quad
	f \cdot \overline{g} \in \calL^1(E_{x, y}). \qedhere
\end{align*}
\end{proof}

\begin{proof}[Proof of Theorem~\ref{thm:IV:6}]
Let $E = E_A$, so by the definition of the spectral matrix measure $M$ for $A$ and $\vec{\varphi}$ we have
\begin{equation} 
	\label{eq:IV:2.9}
	M_{ij} = E_{\varphi_j, \varphi_i} \quad \text{for any }
	i,j=1, \ldots,d.
\end{equation}
If $f \in \calL^2_\Sigma(M)$ then we have (definition of $\calL^2_\Sigma$):
\[
	\forall_{j=1,\ldots, d} \
	f_j \in \calL^2(M_{jj}) = \calL^2(E_{\varphi_j}).
\]
But by "the functional calculus" for s.a. operator we know that $x \in \Dom \big( h(A) \big)$ iff $h \in \calL^2(E_x)$ for Borel $h$ and $x \in \calH$. Hence $\varphi_j \in \Dom \big( f_j(A) \big)$ for any $j=1,\ldots,d$, i.e. (i) holds. Moreover we have
\[
	\tnorm{f}_M^2 =
	\dScalP{f,f}_M =
	\dScalP{f,f}_\Sigma =
	\sum_{i,j=1}^{d} 
	\int_\RR f_j \overline{f_i} \ud M_{ij}
\]
by Fact~\ref{fact:III:25}. Now, by \eqref{eq:IV:2.9} and Lemma~\ref{lem:IV:8} used for $f_j, f_i, \varphi_j$ and $\varphi_i$ for any $i,j=1,\ldots,d$ we obtain:
\begin{align*}
	\tnorm{f}_M^2 
	&=
	\sum_{i,j=1}^{d} \int_\RR f_j \overline{f_i} \ud E_{\varphi_j, \varphi_i} \\
	&=
	\sum_{i,j=1}^{d} E_{f_j(A) \varphi_j, f_i(A) \varphi_i}(\RR) \\
	&=
	\sum_{i,j=1}^{d} \langle E(\RR) f_j(A) \varphi_j, f_i(A) \varphi_i \rangle \\
	&=
	\bigg\langle \sum_{j=1}^{d} f_j(A) \varphi_j, \sum_{i=1}^{d} f_i(A) \varphi_i \bigg\rangle 
	= 
	\| \tilde{W}(f) \|^2
\end{align*}
-- so (ii) holds.

If moreover $f \in \calL^2_0(M)$ then $\tnorm{f}_M = -$ and by \eqref{eq:IV:2.4} (just proved) $\tilde{W} f = 0$. So $W$ is properly defined on $L^2_\Sigma(M)$, and for any $f \in \calL^2_\Sigma(M)$
\[
	\| [f] \|_M = 
	\tnorm{f}_M =
	\| \tilde{W} \| =
	\| W [f] \|.
\]
So $W$ is a linear isometry.
\end{proof}

\vspace{5ex}

\subsection{The canonical spectral transformation ("CST") for $A$ and $\vec{\varphi}$} \label{sec:IV:3}
Consider again a s.a. operator $A$ in a Hilbert space $\calH$ and an ordered system of vectors $\vec{\varphi} \in \calH^d$ ($d \in \NN$). In the previous subsection we constructed the spectral matrix measure $E_{A, \vec{\varphi}}$ for $A$ and $\vec{\varphi}$ on $\Bor(\RR)$ $\sigma$-algebra of subsets of $\RR$. Moreover we constructed a transformation $W : L^2_\Sigma(E_{A, \vec{\varphi}}) \to \calH$ being a linear isometry (see Theorem~\ref{thm:IV:6}(iii)). Now, we shall extend $W$ to the whole $L^2(E_{A, \vec{\varphi}})$ to the so-called (here...) \emph{canonical spectral transformation} (\emph{CST} for short) \emph{for $A$ and $\vec{\varphi}$}.

Our main goal in Section~\ref{sec:IV} is to prove that just \emph{this CST} states a unitary equivalence between $A$ in $\calH$ and $T_\xx$ -- the multiplication by the identity ($x \mapsto x$) function on $\RR$ in the space $L^2(E_{A, \vec{\varphi}})$, if $\vec{\varphi}$ is cyclic for $A$ (being just the $\xx$MUE theorem for the finitely cyclic case).

We shall reach this goal in Section~\ref{sec:IV:4}, but here we \emph{do not need any cyclicity property}.

We start from the mentioned above extension of $W$.
\begin{theorem} \label{thm:IV:9}
Let $M := E_{A, \varphi}$. There exists exactly one operator $U \in \calB \big( L^2(M), \calH \big)$ satisfying
\begin{equation}
	\label{eq:IV:3.1}
	\forall_{f \in \calL^2_\Sigma(M)} \
	U [f] = \sum_{j=1}^d f_j(A) \varphi_j.
\end{equation}
The above unique $U$ is an isometry from $L^2(M)$ into $\calH$ ("onto the image of $U$").
\end{theorem}
\begin{proof}
The uniqueness of such $U$ is obvious from the continuity requirement and from the density of $L^2_\Sigma(M)$ in $L^2(M)$. The existence of such $U$ is the direct consequence of Theorem~\ref{thm:IV:6}(iii) and the abstract result on extension of a bounded linear map from a subspace $X_0$ of $X$ into $Y$ to the bounded linear map from the closure $\overline{X_0}$ into $Y$ for $X,Y$ -- Banach spaces. It remains only to prove that $U$ is an isometry (which is also standard and abstract ...). Using the continuity of $U$ and the density of $L^2_\Sigma(M)$, for any $[f] \in L^2(M)$ we choose $\{f_n \}_{n \geq 1}$ in $\calL^2_\Sigma(M)$ such that $[f_n] \xrightarrow{L^2(M)} [f]$, so $\| U [f_n] \| \to \| U [f] \|$, but $U$ was the extension of $W$ from Theorem~\ref{thm:IV:6}(iii), so also
\[
	\| U [f_n] \| =
	\| W [f_n] \| =
	\| [f_n] \|_M \to 
	\| [f] \|_M,
\]
hence $\|U [f] \| = \| [f] \|_M$.
\end{proof}

Note here, that this step -- the extension from $L^2\Sigma(M)$ onto $L^2(M)$ was not necessary in the case $d=1$, cause then we simply have $L^2_\Sigma(M) = L^2(M)$!

Now we are ready to formulate the announced definition.
\begin{definition} \label{def:IV:10}
Let $A$ be a s.a. operator in $\calH$ and $\vec{\varphi} \in \calH^d$. Then the unique $U \in \calB \big( L^2(E_{A, \vec{\varphi}}), \calH \big)$ satisfying \eqref{eq:IV:3.1} is called \emph{the canonical spectral transformation} (\emph{CST}) \emph{for $A$ and $\varphi$}. We denote it by
\begin{equation}
	\label{eq:IV:3.2}
	U_{A, \vec{\varphi}}.
\end{equation}
\end{definition}

\vspace{5ex}

\subsection{Vector polynomials and $\xx$MUE Theorem} \label{sec:IV:4}
Recall that the construction of the unitary transformation $U$ stating the unitary equivalence for the $d=1$ case of $\xx$MUE Theorem (see Section~\ref{sec:II}) was based on polynomials. -- We considered the subspace $\Pol_\mu(\RR) \underset{\lin}{\subset} L^2(\mu)$ for the spectral measure $\mu = E_{A, \varphi}$ (= $E_{A, (\varphi)}$, with "$\vec{\varphi} = (\varphi)$", adopting our $d$-dimensional notation to the case $d=1$), and this space played quite an important role. $\Pol_\mu(\RR)$ was just the space of classes $[f]$ of polynomials $f \in \Pol(\RR)$, $f : \RR \to \CC$. Now, for general $d \in \NN$, we should consider the appropriate "vector polynomials" $f : \RR \to \CC^d$, so we shall introduce a $\pm$ convenient notation for them:
\begin{itemize}
\item for $n \in \NN_0$ and $j=1,\ldots,d$ $\xx_j^n : \RR \to \RR^d$ is the "monomial"\footnote{Recall that $\xx^n : \RR \to \RR$ is given by $\xx^n(t) := t^n$ for $t \in \RR$ ($\xx^0 = \mathds{1}$).} $\xx^n$ on the $j$-th coordinate, and $0$ on the remaining ones, i.e.
\begin{equation}
	\label{eq:IV:4.1}
	\xx_j^n = \xx^n \cdot e_j
\end{equation}
where $e_j$ is the $j$-th vector of the canonical base $(e_1, \ldots, e_d)$ of $\CC^d$, $(e_j)_k = 
\begin{cases}
	1 & k=j \\
	0 & k \neq j
\end{cases}$.

\item $\Pol_d(\RR) := \big( \Pol(\RR) \big)^d$, so obviously we have
\begin{equation}
	\label{eq:IV:4.2}
	\Pol_d(\RR) =
	\lin \big( \{ \xx_j^n : n \in \NN_0, j=1,\ldots,d \} \big)
\end{equation}

\item if $M$ is a $d \times d$ matrix measure on $\Bor(\RR)$, such that $\Pol_d(\RR) \subset \calL^2(M)$, then
\begin{equation}
	\label{eq:IV:4.3}
	\Pol(M) := \{ [f] \in L^2(M) : f \in \Pol_d(\RR) \}.
\end{equation}
\end{itemize}

Before we formulate our main theorem -- $\xx$MUETh. we need a result which gives even more than the inclusion $\Pol_d(\RR) \subset \calL^2(M)$ for the most important for us choice of $M$.

So, we consider again a Hilbert space $\calH$.
\begin{lemma} \label{lem:IV:11}
Suppose that $A$ is s.a. in $\calH$ and $\vec{\varphi} \in \calH^d$. If $\varphi_1, \ldots, \varphi_d \in \Dom(A^\infty)$, then
\[
	\Pol_d(\RR) \subset \calL^2_\Sigma(E_{A, \vec{\varphi}}).
\]
\end{lemma}
\begin{proof}
Denote $M = E_{A, \vec{\varphi}}$. By \eqref{eq:IV:4.2} it suffices to check that for any $j=1,\ldots,d$ and $n \in \NN_0$ $\xx_j^n \in \calL^2_\Sigma(M)$, so -- by \eqref{eq:IV:4.3} and by the definition of $\calL^2_\Sigma$ (see Section~\ref{sec:III:4}) we should only check, that $\xx^n \in \calL^2(M_{jj})$ for any $n \in \NN_0$ and $j=1,\ldots,d$. Fix $j$. By the definition of the spectral matrix measure $E_{A, \vec{\varphi}}$ we have $M_{jj} = E_{\varphi_j}$ where $E = E_A$ (-- the projection valued spectral measure for $A$). So by STh+FCTh, from $\varphi_j \in \Dom(A^\infty)$ for any $n \in \NN_0$ we have $\varphi_j \in \Dom \big( \xx^n(A) \big)$, which means that $\xx^n \in \calL^2(E_{\varphi_j}) = \calL^2(M_{jj})$.
\end{proof}

So, we see now, that assuming only that each term of $\vec{\varphi}$ is in $\Dom(A^\infty)$ we have the properly defined by \eqref{eq:IV:4.3} subspace $\Pol(M)$ of $L^2(M)$ for $M = E_{A, \vec{\varphi}}$. In particular, we have this if $\vec{\varphi}$ is cyclic for $A$.

\begin{theorem}[$\xx$MUE -- the general finitely cyclic case] \label{thm:IV:xMUE}
Suppose that $A$ is a s.a. finitely cyclic operator in $\calH$. If $\vec{\varphi}$ is a cyclic system for $A$, $M = E_{A, \vec{\varphi}}$ and $U = U_{A, \vec{\varphi}}$, then
\begin{enumerate}[(1)]
\item $\Pol_d(\RR) \subset \calL^2_\Sigma(M)$ and $\Pol(M)$ is dense in $L^2(M)$;

\item $U$ is the unique operator from the set of such $U' \in \calB \big( L^2(M), \calH \big)$, that
\begin{equation}
	\label{eq:IV:4.4}
	\forall_{\substack{n \in \NN_0\\j=1,\ldots,d}} \quad
	U' [\xx_j^n] = A^n \varphi_j;
\end{equation}

\item $U$ is a unitary transformation from $L^2(M)$ \emph{onto} $\calH$;

\item $A = U T_\xx U^{-1}$, where $T_\xx$ is the multiplication by $\xx$ operator in $L^2(M)$.
\end{enumerate}
\end{theorem}
\begin{remark} \label{rem:IV:12}
Note, that we do not assume, that $\vec{\varphi}$ is a linearly independent system! So, we can choose such $\vec{\varphi}$, that $d >$ "dim space of cyclicity"!
\end{remark}
\begin{proof}
We already proved $\subset$ from (1) in Lemma~\ref{lem:IV:11}.

It is also proved that $U$ satisfies the conditions for $U'$ in (2) (still without the uniqueness), because by Theorem~\ref{thm:IV:9}, using \eqref{eq:IV:3.1} for $f = \xx^N_j$ we get $U [\xx^n_j] = \xx^n(A) \varphi_j$, but $\xx^n(A) = A^n$ by functional calculus (STh+FCTh), and
\begin{equation}
	\label{eq:IV:4.5}
	\forall_{\substack{n \in \NN_0\\j=1,\ldots,d}} \quad
	U [\xx_j^n] = A^n \varphi_j.
\end{equation}

By Theorem~\ref{thm:IV:9} we also know, that $U$ is an isometry from $L^2(M)$ onto $\tilde{\calH} \underset{\lin}{\subset} \calH$, so $\tilde{\calH}$ is a closed subspace of $\calH$ as an isometric image of the complete space $L^2(M)$. So, we shall use now the cyclicity of $\vec{\varphi}$ (see Definition~\ref{def:IV:1} and Fact~\ref{fact:IV:3}) -- it guarantees that $\lin \big( \Orb_A(\vec{\varphi}) \big)$ is dense in $\calH$. But $\lin \big( \Orb_A(\vec{\varphi}) \big) \subset \tilde{\calH}$ by \eqref{eq:IV:4.4} used here $U'=U$, hence $\tilde{\calH}$ is closed and dense in $\calH$, i.e. $\tilde{\calH} = \calH$. So, we just proved (3). Now, knowing that $U$ is unitary from $L^2(M)$ onto $\calH$, we also get easily the density from (1) and the uniqueness from (2). Indeed: we already know that $\lin \big( \Orb_A(\vec{\varphi}) \big)$ is dense in $\calH$, so -- by the unitarity of $U$ -- also $U^{-1} \Big( \lin \big( \Orb_A(\vec{\varphi}) \big) \Big)$ is dense in $L^2(M)$, but using \eqref{eq:IV:4.5} we get
\[
	U^{-1} \Big( \lin \big( \Orb_A(\vec{\varphi}) \big) \Big) =
	\lin \big( \{ [\xx_j^n] : n \in \NN_0, j=1,\ldots,d \} \big) =
	\Pol(M).
\]
Hence $\Pol(M)$ is dense, and the uniqueness of $U'$ from (2) follows directly from the "linearity + continuity + density" argument. It remains only to prove (4). The proof of (4) is very similar to the proof of the analogical part of the cyclic ($d=1$) case of Theorem (see Theorem~\ref{thm:II:xMUE} in Section~\ref{sec:II:1}). This similarity is mainly related to the similar "form" of spectral properties, including spectral projections, for the operators of multiplication by functions to its appropriate form in the case $d=1$ (see Section~\ref{sec:III:5}). 

By the uniqueness of the (projection valued) spectral measure for s.a. operator we should only prove that for any $\omega \in \Bor(\RR)$
\begin{equation}
	\label{eq:IV:4.6}
	E_A(\omega) = 
	U E_{T_\xx}(\omega) U^{-1}.
\end{equation}
Let us fix $\omega \in \Bor(\RR)$ and recall, that $E_{T_\xx}(\omega) = T_{\chi_\omega}$ (see Theorem~\ref{thm:III:28}(10)). Using the boundedness of $U$, $E_A(\omega)$, and $E_{T_{\chi_\omega}}$, and the linear density of the set $\{ [\xx_j^n] : n \in \NN_0, j=1,\ldots,d \}$ in $L^2(M)$, to get \eqref{eq:IV:4.6} it suffices to prove
\begin{equation}
	\label{eq:IV:4.7}
	E_A(\omega) U [\xx_j^n]= 
	U [\chi_\omega \xx_j^n]
\end{equation}
for any $n \in \NN_0$ and $j=1,\ldots,d$, because
\[
	E_{T_\xx}(\omega) [\xx_j^n] =
	T_{\chi_\omega} [\xx_j^n] =
	[\chi_\omega \xx_j^n].
\]
But (fortunately...) $\chi_\omega \xx_j^n \in \calL^2(M)$. -- Observe first that the unique non-zero term of $\chi_\omega \xx_j^n$ can be only the $j$-th term equal to $\chi_\omega \xx^n$, but $|\chi_\omega \xx^n| \leq |\xx^n|$ and $\xx^n \in \calL^2(M_{jj})$ because $\xx^n \in \calL^2_{\Sigma}(M)$, by Lemma~\ref{lem:IV:11} (or by (1) -- already proved). So also $\chi_\omega \xx^n \in \calL^2(M_{jj})$, which proves that $\chi_\omega \xx_j^n \in \calL^2_\Sigma(M)$.

So, we can easily prove \eqref{eq:IV:4.7} using the explicit formula for $U = U_{A, \vec{\varphi}}$ on $L^2_\Sigma(M)$, which is provided by Theorem~\ref{thm:IV:9} and Definition~\ref{def:IV:10}. -- Namely, for any $n \in \NN_0$ and $j=1,\ldots,d$ we have:
\begin{equation}
	\label{eq:IV:4.8}
	U [\chi_\omega \xx_j^n] = 
	(\chi_\omega \cdot \xx^n)(A) \varphi_j,
\end{equation}
so by \eqref{eq:IV:4.5}, we will get \eqref{eq:IV:4.7}, if we check
\begin{equation}
	\label{eq:IV:4.9}
	E_A(\omega) A^n \varphi = 
	(\chi_\omega \cdot \xx^n)(A) \varphi_j.
\end{equation}
Let us repeat here the argument from the proof of Section~\ref{sec:II}, which uses the standard FCTh fact on multiplication of two functions of the s.a. operator:
\[
	E_A(\omega) A^n =
	\chi_\omega(A) \xx^n(A) \subset
	(\chi_\omega \xx^n)(A),
\]
but $\varphi_j$ belongs to the domains of both sides\footnote{We know this simply from the formulae \eqref{eq:IV:4.5} and \eqref{eq:IV:4.8} + the boundedness of $E_A(\omega)$, but the deeper reason for it is just the property (i) from Theorem~\ref{thm:IV:6} for any $f \in \calL^2_\Sigma(M)$, which allowed us to properly define $U = U_{A, \vec{\varphi}}$.}, so \eqref{eq:IV:4.8} holds.
\end{proof}

\newpage
\begin{appendix}

\section{The restriction of the matrix measure to a subset and the related $L^2$-space} \label{sec:A}
Consider $\Omega$ -- a set, $\frakM$ -- a $\sigma$-algebra of subsets of $\Omega$. For any $\Omega' \in \frakM$ denote $\frakM_{\Omega'} := \{ \omega \in \frakM : \omega \subset \Omega' \}$\footnote{Note, that we also have $\frakM_{\Omega'} = \{ \omega \cap \Omega' : \omega \in \frakM \}$.}. Surely $\frakM_{\Omega'}$ is a $\sigma$-algebra of subsets of $\Omega'$. Assume, moreover, that $M : \frakM \to \Mat{d}{\CC}$ is a matrix measure on $\frakM$ ("on $\Omega$"). Then, obviously, $\restr{M}{\frakM_{\Omega'}}$ is a matrix measure on $\frakM_{\Omega'}$ ("on $\Omega'$"). We denote it by $M_{\Omega'}$, i.e.
\[
	M_{\Omega'} := \restr{M}{\frakM_{\Omega'}},
\]
and we call it \emph{the restriction of $M$ to $\Omega'$}. We now "fix" $\Omega'$, so let us simplify the notation as follows: $\frakM' := \frakM_{\Omega'}, M' := M_{\Omega'}$. In particular, we get
\begin{equation}
	\label{eq:A:1}
	\tr_{M'} = \restr{\tr_M}{\frakM'}
\end{equation}
Let us denote also:
\[
	\calL^2_{\Omega'}(M) := 
	\{ f \in \calL^2(M): \forall_{t \in \Omega \setminus \Omega'} \ f(t) = 0 \}.
\]
Our "basic" Hilbert space is $L^2(M)$, and the symbol $[\cdot]$ is used to denote the classes of functions in the sense used in $\calL^2(M)$ for $L^2(M)$. But we consider here also the space of $\CC^d$-vector functions $\calL^2(M')$ (on $\Omega'$, measurable with respect to $\frakM'$ etc. -- see Section~\ref{sec:III:2}) and the appropriate Hilbert space $L^2(M')$ of classes of such functions. To distinguish these two kinds of classes we use here the notation $[\cdot]'$ for $L^2(M')$. 

If $d_{M,i,j}$ for $i,j=1,\ldots,d$, and $D_M$ (a trace density) are chosen as in Section~\ref{sec:III:1}, then, by \eqref{eq:A:1} we see that we can choose $D_{M'}$ -- a trace density for $M'$ -- simply as
\begin{equation}
	\label{eq:A:2}
	D_{M'} := \restr{D_M}{\Omega'}.
\end{equation}

For any $g : \Omega' \to \CC^d$ let us denote by $g_{\mathrm{ext}}$ the "extension by $0$ of $g$ to $\Omega$", i.e. $g_{\ext} : \Omega \to \CC^d$
\[
	g_{\ext}(t) := 
	\begin{cases}
		g(t) & \text{for } t \in \Omega' \\
		0 & \text{for } t \in \Omega \setminus \Omega'.
	\end{cases}
\]
Assume that $f \in \calL^2(M')$. So $g$ is $\frakM'$-measurable and thus we easily see that $g_\ext$ is $\frakM$-measurable. Now by \eqref{eq:A:1}, \eqref{eq:A:2}
\begin{equation}
	\label{eq:A:3}
	\dScalP{g_\ext}_M = 
	\int_{\Omega'} \langle D_{M'}(t) g(t), g(t) \rangle_{\CC^d} \ud \tr_{M'}(t) =
	\dScalP{g}_{M'},
\end{equation}
so $g_\ext \in \calM^2_{\Omega'} \subset \calL^2(M)$, because $g \in \calL^2(M')$, i.e., $\dScalP{g}_{M'} < \infty$. Moreover $\calL^2(M') \ni g \mapsto [g_\ext] \in L^2(M)$ is a well-defined, linear map and by \eqref{eq:A:3} its kernel equals $\calL_0(M')$. So, there exists a unique linear factorisation to the quotient space $L^2(M) = \calL^2(M)/\calL^2_0(M)$
\[
	\Ext_{\Omega'} : L^2(M') \to L^2(M)
\]
of the above map. It is given (and is well-defined) by
\begin{equation}
	\label{eq:A:4}
	\Ext_{\Omega'}([g]') = [g_\ext], \quad 
	g \in \calL^2(M'),
\end{equation}
and, moreover, again by \eqref{eq:A:3}, it is an isometry. Denoting now $L^2_{\Omega'}(M) := \Ran(\Ext_{\Omega'})$ and using $I_{\Omega'}$ to denote $\Ext_{\Omega'}$ just treated as the map between spaces $L^2(M')$ and $L^2_{\Omega'}(M)$, we obtain:
\begin{observation} \label{obs:A:1}
$L^2_{\Omega'}(M)$ is a Hilbert space (a closed subspace of $L^2(M)$) and $I_{\Omega'}$ is a unitary transformation.
\end{observation}

Let us find now some relations between $L^2_{\Omega'}(M)$, $\calL^2_{\Omega'}(M)$ and $\calL^2(M')$. 

We have just proved that 
\[
	\{ g_\ext : g \in \calL^2(M') \} \subset \calL^2_{\Omega'}(M).
\]
But if $f \in \calL^2_{\Omega'}(M)$, then obviously $\restr{f}{\Omega'}$ is $\frakM'$ measurable and using \eqref{eq:A:3} to $g := \restr{f}{\Omega'}$ and $g_\ext = f$ we get $\restr{f}{\Omega'} \in \calL^2(M')$. Hence
\begin{equation}
	\label{eq:A:5}
	\{ g_\ext : g \in \calL^2(M') \} = \calL^2_{\Omega'}(M).
\end{equation}
Thus, recalling also that 
\[
	L^2_{\Omega'}(M) = 
	\Ran(\Ext_{\Omega'}) = 
	\{ [g_\ext] \in L^2(M) : g \in \calL^2(M') \},
\]
by \eqref{eq:A:4} and \eqref{eq:A:5} we obtain:
\begin{fact} \label{fact:A:2}
\begin{enumerate}[(i)]
	\item For any $g : \Omega' \to \CC^d$ we have $g \in \calL^2(M') \Leftrightarrow g_\ext \in \calL^2_{\Omega'}(M)$.
	
	\item $L^2_{\Omega'}(M) = \{ [f] \in L^2(M) : f \in \calL^2_{\Omega'}(M) \}$.
	
	\item If $f \in \calL^2_{\Omega'}(M)$, then $[f] = [(\restr{f}{\Omega'})_\ext] = I_{\Omega'}([\restr{f}{\Omega'}]')$, therefore
	\begin{equation}
		\label{eq:A:6}
		\forall_{f \in \calL^2_{\Omega'}(M)} \ I^{-1}_{\Omega'}([f]) = [\restr{f}{\Omega'}]'.
	\end{equation}
\end{enumerate}
\end{fact}

Let us summarise the main information on $I_{\Omega'}$.
\begin{proposition} \label{prop:A:3}
The map $I_{\Omega'} : L^2(M') \to L^2_{\Omega'}(M)$, defined by
\[
	I_{\Omega'}([g]') = [g_\ext], \quad 
	g \in \calL^2(M')
\]
is a unitary transformation and its inverse map $I^{-1}_{\Omega'}$ satisfies
\[
	I^{-1}_{\Omega'}([f]) = [\restr{f}{\Omega'}]', \quad
	f \in \calL^2_{\Omega'}(M).
\]
\end{proposition}

\begin{remark} \label{rem:A:4}
One could ask, whether the last formula
\[
	I^{-1}_{\Omega'}([f]) = [\restr{f}{\Omega'}]'
\]
holds also for other $f \in \calL^2(M)$ such that $[f] \in L^2_{\Omega'}(M)$, and not only for $f \in \calL^2_{\Omega'}(M)$. The answer is YES. To check it, let us consider $f$ as above, and let $\tilde{f} \in \calL^2_{\Omega'}(M)$ be such that $[f] = [\tilde{f}]$. Then for 
\[
	\omega :=
	\big\{ t \in \Omega : f(t) - \tilde{f}(t) \notin \Ker D_M(t) \big\}
\]
we have $\tr_M(\omega) = 0$ by Fact~\ref{fact:III:12}, so also, by \eqref{eq:A:1} $0 = \tr_M(\omega \cap \Omega') = \tr_{M'}(\omega \cap \Omega')$. But 
\[
	\omega \cap \Omega' =
	\big\{ t \in \Omega' : \restr{f}{\Omega'}(t) - \restr{\tilde{f}}{\Omega'}(t) \notin \Ker D_{M'}(t) \big\},
\]
by \eqref{eq:A:2}. Again by Fact~\ref{fact:III:12} (for $\calL_0(M')$) we get 
\[
	[\restr{f}{\Omega'}]' = 
	[\restr{\tilde{f}}{\Omega'}] =
	I^{-1}_{\Omega'}([\tilde{f}]) =
	I^{-1}_{\Omega'}([f]).
\]
\end{remark}

\vspace{10ex}

\section{The part of the multiplication by function self-adjoint operator "in the subset" of $\RR$} \label{sec:B}
Let us recall first the abstract notion of the invariant space and of the restriction to such space for self-adjoint operators.

Let $\calH$ be a Hilbert space (the complex one) and $A$ -- a self-adjoint operator (possibly unbounded) in $\calH$. Let $E_A:\Bor(\RR) \to \calB(\calH)$ be its spectral resolution (the resolution of unity for $A$). For a closed linear subspace $\calH'$ of $\calH$ we define:
$\calH'$ is \emph{invariant} (or \emph{spectrally invariant}) \emph{for} $A$ iff
\begin{equation}
	\label{eq:B:1}
	\forall_{\omega \in \Bor(\RR)} \ 
	P_{\calH'} E_A(\omega) = E_{A}(\omega) P_{\calH'},
\end{equation}
where $P_{\calH'}$ denotes the orthogonal projection onto $\calH'$ in $\calH$. It can be proved (see e.g. \cite{Schmudgen2012} for some help...)
that for $\calH'$ as above 
\begin{equation}
	\label{eq:B:2}
	P_{\calH'} A \subset A P_{\calH'},
\end{equation}
so, in particular, if $x \in \Dom(A)$, then $P_{\calH'} x \in \Dom(A)$ and if $x \in \Dom(A) \cap \calH'$, then $A x \in \calH'$. Hence $\Dom(A) \cap \calH'$ is dense in $\calH'$. Thus the restriction of $A$ to $\Dom(A) \cap \calH'$
\[
	\calH' \supset \Dom(A) \cap \calH' \ni x \mapsto Ax \in \calH'
\]
defines a densely-defined linear operator in $\calH'$. It is well-known
that it is, moreover, self-adjoint. We call it \emph{the part of $A$ in $\calH'$} and denote it by
\begin{equation}
	\label{eq:B:3}
	\restr{A}{\calH'}.
\end{equation}
An important particular case of the invariant space for $A$ is each $\Ran E_A(G)$ for $G \in \Bor(\RR)$. We denote
\[
	(A)_G := \restr{A}{\Ran E_A(G)},
\]
and we call it \emph{the part of $A$ "in $G$"}.

Now, consider $\calH = L^2(M)$ for matrix measure $M$ on $\frakM$ for $\Omega$, $M : \frakM \to \Mat{d}{\CC}$. We shall consider here the self-adjoint operator $A = T_F$ in $L^2(M)$, where $F : \Omega \to \RR$ is an $\frakM$-measurable function (see Section~\ref{sec:III:5}). So, let us fix an arbitrary $G \in \Bor(\RR)$. Our goal in this section is to "identify" the operator $(T_F)_G$ -- the part of $T_F$ "in $G$".

Let us denote here
\[
	\Omega' := F^{-1}(G)
\]
and let us adopt the notation in Section~\ref{sec:A} for this $\Omega'$ (including the simplified ones: $\frakM', M'$). Define also an $\frakM'$-measurable function $F' : \Omega' \to \RR$ just by
\[
	F' := \restr{F}{\Omega'}
\]
and consider the operator $T_{F'}$ of the multiplication by $F'$ in $L^2(M')$. The following result is "almost obvious", but it needs careful checking ...

\begin{theorem} \label{thm:B:5}
\begin{enumerate}[(i)]
\item $\Ran E_{F_T}(G) = L^2_{\Omega'}(M)$

\item $(T_F)_G$ equals to the operator in $L^2_{\Omega'}(M)$ given by
\begin{align*}
	&\Dom \big( (T_F)_G \big) = 
	\big\{ [f] \in L^2_{\Omega'}(M) : f \in \calL^2_{\Omega'}, Ff \in \calL^2(M) \big\} \\
	&(T_F)_G[f] = [Ff], \quad \text{if } f \in \calL^2_{\Omega'}(M) \text{ and } Ff \in \calL^2(M).
\end{align*}

\item $(T_F)_G \sim T_{F'}$\footnote{$\sim$ denotes the unitary equivalence relation for operators.} and moreover
\begin{equation}
	\label{eq:B:4}
	(T_F)_G = I_{\Omega'} T_{F'} I_{\Omega'}^{-1}.
\end{equation}
\end{enumerate}
\end{theorem}
\begin{proof}
We know (see Theorem~\ref{thm:III:28}(10)), that $E_{T_F}(G) = T_{\chi_{\Omega'}}$, so $x \in \Ran E_{T_F}(G)$ iff $x = [\chi_{\Omega'} f]$ for some $f \in \calL^2(M)$. But if $f \in \calL^2(M)$, then $\chi_{\Omega'} f \in \calL^2_{\Omega'}(M)$ so $[\chi_{\Omega'} f] \in L^2_{\Omega'}(M)$ by Fact~\ref{fact:A:2}(ii), and we get "$\subset$" in (i). Similarly, consider $x \in L^2_{\Omega'}(M)$. By the same Fact~\ref{fact:A:2}(ii) we know, that $x = [f]$ soe some $f \in \calL^2_{\Omega'}(M)$, hence $f = \chi_{\Omega'} f$, and thus $x = [f] = [\chi_{\Omega'} f] = T_{\chi_{\Omega'}} [f]$, which shows "$\supset$". Now we easily get (ii) of the theorem by (i), by Fact~\ref{fact:A:2}(ii) and by the definition of the part "in $G$" of the operator. Note, that we use here also the fact, that the multiplication operator in $L^2(M)$ is well-defined, as the factorisation of the multiplication by $F$ in $\calL^2(M)$.

So, it remains to check \eqref{eq:B:4}. Let us write down the domains of both sides:
\begin{align*}
	&\mathrm{LHS} := 
	\big\{ [f] \in L^2_{\Omega'}(M) : f \in \calL^2_{\Omega'}, Ff \in \calL^2(M) \big\} \\
	&\mathrm{RHS} := 
	\big\{ x \in L^2_{\Omega'}(M) : I^{-1}_{\Omega'} x \in \Dom(T_{F'}) \big\}.
\end{align*}
When $x \in \mathrm{RHS}$, then for some $f \in \calL^2_{\Omega'}(M)$ $x = [f]$ and $I^{-1}_{\Omega'} [f] = [\restr{f}{\Omega'}]' \in \Dom(T_{F'})$ by Fact~\ref{fact:A:2}(iii), so $F'(\restr{f}{\Omega'}) \in \calL^2(M')$. But thanks to $f \in \calL^2_{\Omega'}(M)$
\begin{equation}
	\label{eq:B:5}
	F' \cdot (\restr{f}{\Omega'}) = \restr{F \cdot f}{\Omega'}, \quad \text{and} \quad
	\big( F' \cdot (\restr{f}{\Omega'}) \big)_\ext = F f,
\end{equation}
and by Fact~\ref{fact:A:2}(i) $Ff \in \calL^2(M)$, which gives $x \in \mathrm{LHS}$.

Now, consider some $x \in \mathrm{LHS}$. We have $x = [f]$ for some $f \in \calL^2_{\Omega'}(M)$ and $Ff \in \calL^2(M)$, thus again by Fact~\ref{fact:A:2} $I^{-1}_{\Omega'} x = [\restr{f}{\Omega'}]'$ and using \eqref{eq:B:5} we get $F' \cdot (\restr{f}{\Omega'}) \in \calL^2(M')$. Therefore $I^{-1}_{\Omega'} x \in \Dom(T_{F'})$, i.e., $x \in \mathrm{RHS}$. But moreover by \eqref{eq:B:5} we have (see \eqref{eq:A:4})
\[
	I_{\Omega'} T_{F'} I_{\Omega'}^{-1} x = 
	I_{\Omega'} [F' (\restr{f}{\Omega'})]' =
	[(F'(\restr{f}{\Omega'}))_\ext] =
	[Ff] = T_F x,
\]
i.e. $\mathrm{LHS} = \mathrm{RHS}$ and \eqref{eq:B:4} holds.
\end{proof}

\vspace{5ex}

\section{The absolute continuity and the a.c. spectrum of self-adjoint finitely cyclic operators} \label{sec:C}
We prove here an abstract result (Theorem~\ref{thm:C:7}) showing a simple and natural relation between some "measure theory" properties of the trace measure $\tr_{E_{A, \vec{\varphi}}}$ of the matrix measure $E_{A, \vec{\varphi}}$ for fin. cycl. self-adjoint operator $A$ and its cyclic system $\vec{\varphi}$ and some spectral properties of $A$ such as the absolute continuity in a subset of $\RR$ and properties of the a.c. spectrum.

Recall, that for a s.a. operator $A$ in a Hilbert space $\calH$ and $G \in \Bor(\RR)$ the symbol $(A)_G$ denotes the part of $A$ (see Section~\ref{sec:B}) in the (invariant) subspace $\calH_G(A) := \Ran E_A(G)$, where $E_A$ denotes as usual the spectral resolution for $A$. The notion \emph{$A$ is absolutely continuous (a.c.) in $G$} means that $H_G(A) \subset H_{\mathrm{ac}}(A)$, where
\[
	H_{\mathrm{ac}}(A) :=
	\{ x \in \calH : E_{A,x} \text{ is a.c. with respect to the Lebesgue measure on } \Bor(\RR) \},
\]
and $E_{A,x}$ denotes the scalar spectral measure for $A$ and $x$ (given by $E_{A,x}(\omega) = \|E_A(\omega) x\|^2$ for $\omega \in \Bor(\RR)$). Recall also that\footnote{We use the convention that the spectrum $\sigma(Z)$, of the zero operator $Z$ in the space $\{0\}$ is empty.} $\sigmaAC(A) := \sigma((A)_{\mathrm{ac}})$, where $(A)_{\mathrm{ac}} := \restr{A}{H_{\mathrm{ac}}(A)}$ (-- the part of $A$ in $H_{\mathrm{ac}}(A)$).

Suppose that $A$ is a finitely cyclic self-adjoint operator in a Hilbert space $\calH$ and $\vec{\varphi} \in \calH^d$ is a cyclic system for $A$ (with some $d \in \NN$). Let $M := E_{A, \vec{\varphi}}$ be the spectral matrix measure for $A$ and $\vec{\varphi}$ (see Definition~\ref{def:IV:5}), $M : \Bor(\RR) \to \Mat{d}{\CC}$. Denote $\mu := \tr_M$ -- the trace measure of $M$, and by $\muAC$, $\muS$ denote the absolutely continuous and the singular part of $\mu$ with respect to the Lebesgue measure $|\cdot|$ on $\Bor(\RR)$, respectively.

Consider $G \in \Bor(\RR)$. Denote by $\overline{G}^{\mathrm{Leb}}$ its "$|\cdot|$-closure" defined by
\begin{equation}
	\label{eq:C:1}
	\overline{G}^{\mathrm{Leb}} := 
	\big\{ 
		t \in \RR : \forall_{\varepsilon > 0} \
		|G \cap (t-\varepsilon; t+ \varepsilon)| > 0
	\big\}.
\end{equation}
Let $T_{\xx_G, M}$ be the multiplication by the function $\restr{\xx}{G}$ operator in the space $L^2(M_G)$ for the restriction $\restr{\xx}{G}$ to $G$ of the identity function $\xx : \RR \to \RR$, where $M_G$ is the restriction of $M$ to $G$ (see Section~\ref{sec:A}).

To formulate and prove the main result of this section let us study first some abstract problems related to some supports of measures and to the resolution of measure into its "ac" and "sing" parts with respect to another measure.

Suppose that $\frakM$ is a $\sigma$-algebra of subsets of $\Omega$ and $\mu, \nu$ are two measures on $\frakM$. Recall that $S \in \frakM$ is \emph{a minimal support of $\mu$ with respect to $\nu$} iff $S$ is a support of $\mu$ (i.e. $\mu(\Omega \setminus S) = 0$) and for each $S' \in \frakM$ such that $S'$ is a support of $\mu$ and $S' \subset S$, $\nu(S \setminus S') = 0$.

\begin{lemma} \label{lem:C:6}
Suppose that both $\mu$ and $\nu$ are $\sigma$-finite and $\muAC, \muS$ are the absolutely continuous and the singular part of $\mu$ with respect to $\nu$, respectively. Let $S_{\mathrm{ac}} \in \frakM$ be a minimal support of $\muAC$ with respect to $\nu$. then
\begin{enumerate}[(i)]
\item 
$\begin{aligned}[b]
	\forall_{\substack{\omega \in \frakM\\\muS(\omega) = 0}} \quad
	\big( \nu(\omega) = 0 \Rightarrow \mu(\omega) = 0 \big),
\end{aligned}$

\item
$\begin{aligned}[b]
	\forall_{\substack{\omega \in \frakM\\\omega \subset S_{\mathrm{ac}}}} \quad
	\big( \mu(\omega) = 0 \Rightarrow \nu(\omega) = 0 \big).
\end{aligned}$
\end{enumerate}
\end{lemma}
\begin{proof}
Let $\omega \in \frakM$.

\fbox{(i)} If $\muS(\omega) = 0$, then $\mu(\omega) = \muAC(\omega) + \muS(\omega) = \muAC(\omega)$, but when also $\nu(\omega) = 0$, then $\muAC(\omega) = 0$, so $\mu(\omega) = 0$.

\fbox{(b)} If $\omega \subset S_{\mathrm{ac}}$ and $\mu(\omega) = 0$, then also $\muAC(\omega) = 0$ (because $0 \leq \muAC \leq \muAC + \muS = \mu$), hence $S_{\mathrm{ac}} \setminus \omega$ is also a support of $\muAC$, but $S_{\mathrm{ac}} \setminus \omega \subset S_{\mathrm{ac}}$, so $0 = \nu \big( S_{\mathrm{ac}} \setminus (S_{\mathrm{ac}} \setminus \omega) \big) = \nu(\omega)$.
\end{proof}

We are ready now to formulate and prove the man result of this section for "our" self-adjoint finitely cyclic operator $A$.

\begin{theorem} \label{thm:C:7}
Let $G \in \Bor(\RR)$. Then:
\begin{enumerate}[(1)]
\item $(A)_G \sim T_{\xx_{G,M}}$.

\item If
\begin{equation}
	\label{eq:C:2}
	\muS(G) = 0,
\end{equation}
then $A$ is absolutely continuous in $G$.

\item If \eqref{eq:C:2} holds, and moreover $G \subset S$ for some $S \in \Bor(\RR)$ being a minimal support of $\muAC$ with respect to $|\cdot|$, then
\begin{equation}
	\label{eq:C:3}
	\overline{G}^{\mathrm{Leb}} \subset \sigmaAC(A).
\end{equation}
\end{enumerate}
\end{theorem}
\begin{proof}
\fbox{(1)} By "$\xx$MUE" Theorem for fin. cycl. case (see Theorem~\ref{thm:IV:xMUE}) we get
\begin{equation}
	\label{eq:C:4}
	A \sim T_{\xx},
\end{equation}
where $T_{\xx}$ denotes the multiplication by $\xx$ operator in $L^2(M)$. So, we get
\begin{equation}
	\label{eq:C:5}
	(A)_G \sim (T_{\xx})_G,
\end{equation}
and by Theorem~\ref{thm:B:5} we obtain
\begin{equation}
	\label{eq:C:6}
	(A)_G \sim (T_{\xx})_G \sim T_{\xx_{G,M}}.
\end{equation}

\fbox{(2)} By \eqref{eq:C:4}, to get that $A$ is a.c. in $G$, it suffices to prove that $T_\xx$ is a.c. in $G$ that is
\begin{equation}
	\label{eq:C:6'}
	H_G(T_\xx) \subset H_{\mathrm{ac}}(T_\xx).
\end{equation}
So, let $f \in \calL^2(M)$ be such, that $[f] \in H_G(T_\xx) = \Ran E_{T_\xx}(G)$. By Theorem~\ref{thm:III:28} we have
\begin{equation}
	\label{eq:C:7}
	\forall_{\omega \in \Bor(\RR)} \
	E_{T_\xx}(\omega) = T_{\chi_\omega}
\end{equation}
(here $T_{\chi_\omega}$ is the operator in $L^2(M)$ of multiplication by $\chi_\omega$ -- the characteristic function of $\omega$). So, the condition $[f] \in H_G(T_\xx)$ means that
\begin{equation}
	\label{eq:C:8}
	[f] = E_{T_\xx}(G)[f] = [\chi_G f].
\end{equation}
Moreover, by \eqref{eq:C:7} and by \eqref{eq:III:3.1} and \eqref{eq:III:2.3} for any $\omega \in \Bor(\RR)$ and $g \in \calL^2(M)$
\begin{equation}
	\label{eq:C:9}
	\begin{split}
		E_{T_\xx, [g]}(\omega) 
		&=
		\langle [\chi_\omega g], [g] \rangle_M \\
		&=
		\dScalP{\chi_\omega g, g}_M \\
		&=
		\int_\RR \langle D_M(t) \chi_\omega(t) g(t), g(t) \rangle_{\CC^d} \ud \tr_M(t) \\
		&=
		\int_\omega \langle D_M(t) g(t), g(t) \rangle_{\CC^d} \ud \mu(t).
	\end{split}
\end{equation}
Hence, by \eqref{eq:C:8}, for any $\omega \in \Bor(\RR)$
\begin{equation}
	\label{eq:C:10}
	\begin{split}
		E_{T_\xx, [f]}(\omega) 
		&=
		E_{T_{\xx, [\chi_G f]}}(\omega) \\
		&=
		\int_\omega \langle D_M(t) \chi_G(t) f(t), \chi_G(t) f(t) \rangle_{\CC^d} \ud \mu(t) \\
		&=
		\int_{\omega \cap G} \langle D_M(t) f(t), f(t) \rangle_{\CC^d} \ud \mu(t).
	\end{split}
\end{equation}
But we have $\mu = \muAC + \muS$, so by \eqref{eq:C:2} and \eqref{eq:C:10} for any $\omega \in \Bor(\RR)$
\begin{equation}
	\label{eq:C:11}
	E_{T_\xx, [f]}(\omega) = 
	\int_{\omega \cap G} \langle D_M(t) f(t), f(t) \rangle_{\CC^d} \ud \muAC(t).
\end{equation}
But, "by definition" $\muAC$ is a.c. with respect to $|\cdot|$, so when $|\omega| = 0$, then $|\omega \cap G| = 0$ by \eqref{eq:C:11}. Thus proves that $E_{T_{\xx, [f]}}$ is a.c. with respect to $|\cdot|$, so $[f] \in H_{\mathrm{ac}}(T_\xx)$, and \eqref{eq:C:6'} holds.

\fbox{(3)} From part (2) by \eqref{eq:C:2}, we know that \eqref{eq:C:6'} holds. Hence, by spectral theory for the parts of s.a. operators in its invariant subspaces,
we have
\begin{equation}
	\label{eq:C:12}
	\sigma((T_\xx)_G) = 
	\sigma(\restr{T_\xx}{H_G(T_\xx)}) \subset 
	\sigma(\restr{T_\xx}{H_\mathrm{ac}(T_\xx)}) = 
	\sigmaAC(T_\xx).
\end{equation}
Moreover, from part (1) we know that \eqref{eq:C:6} and \eqref{eq:C:4}  hold, so by \eqref{eq:C:12} 
\[
	\sigma(T_{\xx_{G, M}}) \subset \sigmaAC(A).
\]
To finish the proof we should only show, that
\begin{equation}
	\label{eq:C:13}
	\sigma(T_{\xx_{G, M}}) = 
	\overline{G}^{\mathrm{Leb}}.
\end{equation}
But using the definition of the restriction $M_G$ of the matrix measure (see Section~\ref{sec:A}, Theorem~\ref{thm:III:28}7)) and the definition of the essential value set (see Section~\ref{sec:III:3}) we get\footnote{It suffices to observe the trivial fact, that
\[
	\tr_{M_G}(\omega) =
	\tr \big( M_G(\omega) \big) = 
	\tr \big( M(\omega) \big) =
	\tr_M(\omega) = 
	\mu(\omega) = \mu_G(\omega)
\]
for any $\omega \in \Bor(\RR)$ such that $\omega \subset G$.}
\begin{equation}
	\label{eq:C:14}
	\sigma(T_{\xx_{G, M}}) = 
	\VE_{\mu_G}(\restr{\xx}{G}),
\end{equation}
where $\mu_G$ is the restriction of $\mu$ to the Borel subsets of $G$ (being the $\sigma$-algebra of subsets of $G$), and
\begin{align*}
	\VE_{\mu_G}(\restr{\xx}{G}) 
	&=
	\big\{ t \in \RR : \forall_{\varepsilon > 0} \ \mu_G \big( (t-\varepsilon; t+\varepsilon) \cap G \big) > 0 	
	\big\} \\
	&=
	\big\{ t \in \RR : \forall_{\varepsilon > 0} \ \mu \big( (t-\varepsilon; t+\varepsilon) \cap G \big) > 0 	
	\big\}.
\end{align*}
By \eqref{eq:C:2} and $G \subset S$, for any $\omega \in \Bor(\RR)$ of the form $\omega = (t-\varepsilon; t+\varepsilon) \cap G$, with $t \in \RR$, $\varepsilon > 0$, we have
\[
	\muS(\omega) = 0 \quad \text{and} \quad 
	\omega \subset S.
\]
Hence, for $\omega$ as above we can use Lemma~\ref{lem:C:6} for $\nu = |\cdot|$ and $S_{\mathrm{ac}} = S$ and we get $|\omega| > 0 \Leftrightarrow \mu(\omega) > 0$. Thus proves that (see \eqref{eq:C:1}) $\VE_{\mu_G} (\restr{\xx}{G}) = \overline{G}^{\mathrm{Leb}}$, and by \eqref{eq:C:14} we get \eqref{eq:C:13}.
\end{proof}

\begin{fact}[Properties of $\overline{G}^{\mathrm{Leb}}$] \label{fact:C:8}
\begin{enumerate}[(i)]
\item For any $G \in \Bor(\RR)$
\begin{equation}
	\label{eq:C:15}
	\overline{G}^{\mathrm{Leb}} \subset
	\overline{G}.
\end{equation}

\item If $G \subset \RR$ is a sum of any family of intervals of non-zero length (of any kind, including infinite ones)\footnote{i.e. sum of any family of connected non-singletons.}, then $\overline{G}^{\mathrm{Leb}} = \overline{G}$,

\item If $G$ is open, then $\overline{G}^{\mathrm{Leb}} = \overline{G}$.

\item If $|G| = 0$, then $\overline{G}^{\mathrm{Leb}} = \emptyset$.
\end{enumerate}
\end{fact}
\begin{proof}
\fbox{(i)} Obviously $\RR \setminus \overline{G} \subset \RR \setminus \overline{G}^{\mathrm{Leb}}$.

\fbox{(ii)} By (i) it suffices to prove $\overline{G} \subset \overline{G}^{\mathrm{Leb}}$. But for $t \in \overline{G}$ and any $\varepsilon > 0$ $G \cap (t-\varepsilon; t + \varepsilon)$ contains a non-empty intersection of $I \cap (t-\varepsilon; t + \varepsilon)$, where $I$ is an interval with $|I| > 0$. Hence $I \cap (t-\varepsilon; t + \varepsilon)$ is also an interval and also $|I \cap (t-\varepsilon; t + \varepsilon)| > 0$, thus also $|G \cap (t-\varepsilon; t + \varepsilon)| > 0$, and $t \in \overline{G}^{\mathrm{Leb}}$.

\fbox{(iii)} follows from (ii) for open intervals, and (iv) is obvious.
\end{proof}

\begin{remark} \label{rem:C:9}
Using Fact~\ref{fact:C:8} we can obtain the stronger assertion
\begin{equation}
	\label{eq:C:16}
	\overline{G} \subset \sigmaAC(A),
\end{equation}
adding the extra condition on $G$ as in (ii) and (iii) of this Fact to the assumptions of Theorem~\ref{thm:C:7}.
\end{remark}

\end{appendix}

\newpage

\begin{bibliography}{literature.bib, block.bib, jacobi.bib}
	\bibliographystyle{amsplain}
\end{bibliography}

\end{document}